\pdfoutput=1
\documentclass[11pt]{amsart}
\usepackage[T1]{fontenc}
\usepackage[utf8]{inputenc}
\usepackage{amssymb,amsmath,amsthm,mathrsfs,array,graphicx,bbm,enumerate,wasysym,dsfont}
\usepackage{lineno}
\usepackage[usenames, dvipsnames]{color}
\usepackage[all]{xy}
\usepackage{bbm}
\usepackage{fullpage}
\usepackage{float}
\usepackage{verbatim}
\usepackage{soul}
\usepackage{hyperref}

\newtheorem{thm}{Theorem}[section]
\newtheorem{lemma}[thm]{Lemma}
\newtheorem{rem}[thm]{Remark}
\newtheorem{defn}[thm]{Definition}
\newtheorem{prop}[thm]{Proposition}
\newtheorem{cor}[thm]{Corollary}
\newtheorem{claim}[thm]{Claim}

\numberwithin{equation}{section}

\newcommand{\Q}{\mathbb Q}

\newcommand{\R}{\mathbb R}
\newcommand{\C}{\mathbb C}
\renewcommand{\H}{\mathbb H}

\newcommand{\MM}{\mathfrak{M}}
\newcommand{\N}{\mathbb N}
\newcommand{\D}{\mathbb D}
\newcommand{\E}{\mathbb E}
\newcommand{\CC}{\mathcal C}

\renewcommand{\1}{\mathbf 1}
\newcommand{\A}{\mathds A}
\newcommand{\ABP}[2]{\protect\rotatebox[origin=c]{180}{$#1\A$}}
\newcommand{\AB}{{\mathpalette\ABP\relax}}
\newcommand{\B}{\mathcal B}
\newcommand{\ED}{\operatorname{EL}}
\newcommand{\M}{\operatorname{M}}

\newcommand{\LL}{\mathbb L}
\newcommand{\F}{\mathcal F}

\renewcommand{\epsilon}{\varepsilon}
\newcommand{\JJ}{\mathcal J}

 	\definecolor{pakistangreen}{rgb}{0.0, 0.4, 0.0}

\newcommand{\eps}{\epsilon}	
\newcommand{\crad}{\operatorname{CR}}

\renewcommand{\a}{\alpha}

\renewcommand{\d}{{d}}

\begin{document}
\title{First passage sets of the 2D continuum Gaussian free field}
\author{Juhan Aru \and Titus Lupu \and Avelio Sep\'ulveda}

\address {
Institute of Mathematics,
EPFL, CH-1015 Lausanne,
Switzerland}
\email
{juhan.aru@math.ethz.ch}

\address{CNRS and LPSM, UMR 8001, Sorbonne Université, 4 place Jussieu, 75252 Paris cedex 05, France}
\email
{titus.lupu@upmc.fr}

\address{Univ Lyon, Université Claude Bernard Lyon 1, CNRS UMR 5208, Institut Camille Jordan, 69622 Villeurbanne, France}
\email
{sepulveda@math.lyon-1.fr}

\begin{abstract}
We introduce the first passage set (FPS) of constant level $-a$ of the two-dimensional continuum Gaussian free field (GFF) on finitely connected domains. Informally, it is the set of points in the domain that can be connected to the boundary by a path on which the GFF does not go below $-a$. It is, thus, the two-dimensional analogue of the first hitting time of $-a$ by a one-dimensional Brownian motion. We provide an axiomatic characterization of the FPS, a continuum construction using level lines, and study its properties: it is a fractal set of zero Lebesgue measure and Minkowski dimension 2 that is coupled with the GFF $\Phi$ as a local set $A$ so that $\Phi+a$ restricted to $A$ is a positive measure. One of the highlights of this paper is identifying this measure as a Minkowski content measure in the non-integer gauge $r \mapsto \vert\log(r)\vert^{1/2}r^{2}$, by using Gaussian multiplicative chaos theory. 
\end{abstract}


\maketitle

\section{Introduction}

The continuum Gaussian free field (GFF) is a canonical model of a Gaussian field satisfying a spatial Markov property. It first appeared in Euclidean quantum field theory, where it is known as bosonic massless free field \cite{Simon1974EQFT,GawedzkiCFT}.
In the probability community the study of the 2D continuum GFF has reflourished in the 2000's due to its various connections to Schramm's SLE processes \cite{She05,Dub,SchSh2,MS1,MS2,MS3,MS4}, Liouville quantum gravity measures \cite{DS,BSS} and Brownian loop-soups \cite{LeJan2010LoopsRenorm,LeJan2011Loops}. 

The seminal papers that connected SLE processes and the free field showed that SLE$_4$ can be seen as a level line of the GFF - more precisely, in \cite{SchSh} Schramm and Sheffield showed that the level lines of the discrete GFF converge in law to SLE$_4$ and in \cite{SchSh2} they gave a purely continuum definition of the limiting coupling, giving rise to the study of level lines of the continuum GFF. 

In \cite{ASW}, it was further shown that level lines give a way to define other geometric subsets of the GFF. More precisely, in \cite{ASW} the authors introduce two-valued local sets $\A_{-a,b}$. Heuristically, the set $\A_{-a,b}$ corresponds to the points in the domain that can be connected to the boundary by some path along which the height of the GFF remains in $[-a,b]$. The mathematical definition of these sets is based on thinking of the 2D GFF as a generalization of the Brownian motion, and it relies on the strong Markov property of the free field. In the case of the Brownian motion, the same geometric heuristic defines the set $[0,T_{-a,b}]$, where $T_{-a,b}$ is the first exit time from the interval $[-a,b]$, and in \cite{ASW} it is proved that $\A_{-a,b}$ satisfies many properties expected from the analogy with these exit times.

In the current article, we introduce a further geometric subset of the GFF: the first passage set (FPS) $\A_{-a}$. Heuristically, it corresponds to the points in the domain that can be connected to the boundary by paths along which the height of the GFF is greater or equal to $-a$, i.e., it is the analogue of $[0,T_{-a}]$, where $T_{-a}$ is the one-sided first passage time of a BM. We provide an axiomatic characterization of the continuum FPS, a construction using  iterations of level lines as in \cite{ASW} and study several of its properties. 

There are two key aspects which make the FPS interesting to study. First of all, compared to most of the geometric subsets of the GFF studied so far, this set is large in the sense that the restriction of the GFF to this set is a non-zero distribution.  This not only requires new ways of working with this set, but also introduces interesting phenomena - as one of the key results we show that even though the GFF on this set is non-trivial, it is measurable with respect to the underlying set. Even more, we show that the restriction of the GFF to its FPS can be identified with the Minkowski content measure of the underlying set in the gauge $r\mapsto |\log(r)|^{1/2} r^2$. Secondly, in the case of the FPS the geometric definition given above can be made precise in the following sense: in a follow-up article \cite{ALS2}, we show that the first passage sets of the metric graph GFF introduced in \cite{LupuWerner2016Levy} converge to the continuum FPS. This will, among other things, allow us to identify the FPS with the trace of a clusters of Brownian loops and excursions, and to prove convergence results for the level lines of the GFF. Moreover, in a subsequent work we will use the results of these two papers to prove an excursion decomposition of the 2D continuum GFF \cite{ALS4}. 

First passage sets have already proved useful in studying Gaussian multiplicative chaos (GMC) measures of the GFF: in \cite{APS}, the authors confirm that a construction of \cite{Ai} converges to the GMC measure. This was done using the fact that GFF can be approximated by its FPS of increasing levels. Moreover, in \cite{APS2}, and using the same FPS based construction, the authors prove that a ``derivative'' of subcritical GMC measures coincides with a multiple of the critical measure. This confirmed a conjecture of \cite{DRSV}.

In this article we construct the FPS not only for simply-connected domains, but on any planar domain with finitely many holes. To do this, we also extend the construction of the two-valued local sets, first introduced in \cite{ASW} for simply connected domains, to multiply-connected domains. Admittedly, the choice to work in a non-simply-connected setting makes the paper more technical. This more general setting has, however, several important motivations. First of all, in the follow-up article \cite{ALS2} we relate the FPS to the clusters of Brownian loops and Brownian excursions. Doing it in multiply connected domains emphasizes that this relation is not specific to the geometry of the domain, and stems from very general considerations known as ``isomorphism theorems'', such as the Dynkin's isomorphism. As a consequence, we obtain convergence results also for SLE-type curves in multiply-connected domains. Moreover, in a different article \cite{ALS3} by the same authors, we use two-valued local sets and FPS in annular domains in order to calculate explicit laws of extremal distances related to CLE$_4$ loops. In fact, the need to deal with multiply connected domains  arises naturally in the simply connected setting: take a two valued-set or an FPS of a GFF on a disk and then remove a finite number of its holes. One then gets a multiply connected domain. Thus, in order to study the conditional law of a two-valued set or an FPS given a finite number of its holes, one has to deal with multiply connected domains, even if the initial domain is simply connected. This said, let us stress that in order to prove the results of the current paper for the simply-connected case, one does not need to pass through the multiply-connected set-up.

\subsection{Overview of results}

Let us give a more detailed overview of the results presented in the paper. To do this, first recall the local set coupling of a random set $A$ with the Gaussian free field $\Phi$ in a domain $D$. It is a coupling $(\Phi, A)$ that induces a Markovian decomposition of $\Phi$. That is to say $\Phi$ can be written as a sum between $ \Phi_A$ and $\Phi^A$, where $\Phi_A$ is a random distribution that is a.s. harmonic on $D\backslash A$ and  conditional on $(A,\Phi_A)$, $\Phi^A$ is a zero boundary GFF on $D\backslash A$. We denote by $h_A$ the harmonic function corresponding to $\Phi_A$ outside of $A$. The local set condition implies that conditional on $A$ and 
$\Phi_A$, the GFF $\Phi$ restricted to $D\backslash A$ is given by the sum of $h_A$ and $\Phi^A$. 

The two valued local sets (TVS) $\A_{-a,b}$ for a simply connected domain $D$, studied in \cite{ASW}, can be then defined as the only thin local sets of the GFF such that $h_A \in \{-a,b\}$ . Here, thin means that at $\Phi_A$ contains no extra mass on $A$, i.e. for any smooth function $f$, we have that $(h_A,f) = (\Phi_A,f)$.

Our first task is to generalize two-valued sets to multiply-connected domains and to more general boundary conditions $u$, and to show that the main properties of TVS proved in \cite{ASW} remain true also in this setup. The generalization to more general boundary conditions requires only slight modifications in the definitions, and slight extensions of the proofs. The case of multiply-connected domains, however, requires both new ideas and technical work. In particular, as technical result of independent interest we prove in Proposition \ref{propext} \textit{that level lines in multiply-connected domains are continuous up to and at hitting a continuation threshold.}

We next introduce the first passage set (FPS). For a zero boundary GFF, the FPS of level $-a$, denoted $\A_{-a}$ with $a \geq 0$ is then defined as a local set of the Gaussian free field on $D$ satisfying the following properties: 
{\it \begin{itemize}
		\item Conditional on $\A_{-a}$, the law of the restriction of the GFF on $D$ to  $D\backslash \A_{-a}$ is that of a GFF on $D\backslash \A_{-a}$ with boundary condition $-a$, or in other words $h_{\A_{-a}} = -a$.
		\item The GFF on $\A_{-a}$ is greater than or equal to $-a$, in the sense that for any positive test function $f$ we have that $(\Phi_{\A_{-a}}+a,f) \geq 0$ - that is to say $\Phi_{\A_{-a}}+a$ is a positive measure.
	\end{itemize}}
	
	The full definition for general boundary conditions $u$ is given in Definition \ref{Def ES}. There, we also define the FPS in the other direction: $\AB_{b}$ will heuristically correspond to the local set $A$ such that $h_A = b$ and the GFF on $A$ is smaller than $b$. As proved in Theorem \ref{FPSthm}  in the setting of more general boundary conditions, the first passage set
	{\it
		\begin{itemize}
			\item is unique in the sense that any other local set with the above conditions is a.s. equal to the FPS, and thus, it is a measurable with respect to the GFF it is coupled with;
			\item is monotone in the sense that for all $a \leq a'$ almost surely $\A_{-a} \subset \A_{-a'}$;
			\item as in the case of the Brownian motion can be constructed as a limit of two-valued local sets,
			$\A_{-a} = \lim_{b \to \infty} \A_{-a,b}$.
		\end{itemize}}
		
		\begin{figure}	   
			\centering
			\includegraphics[width=5in]{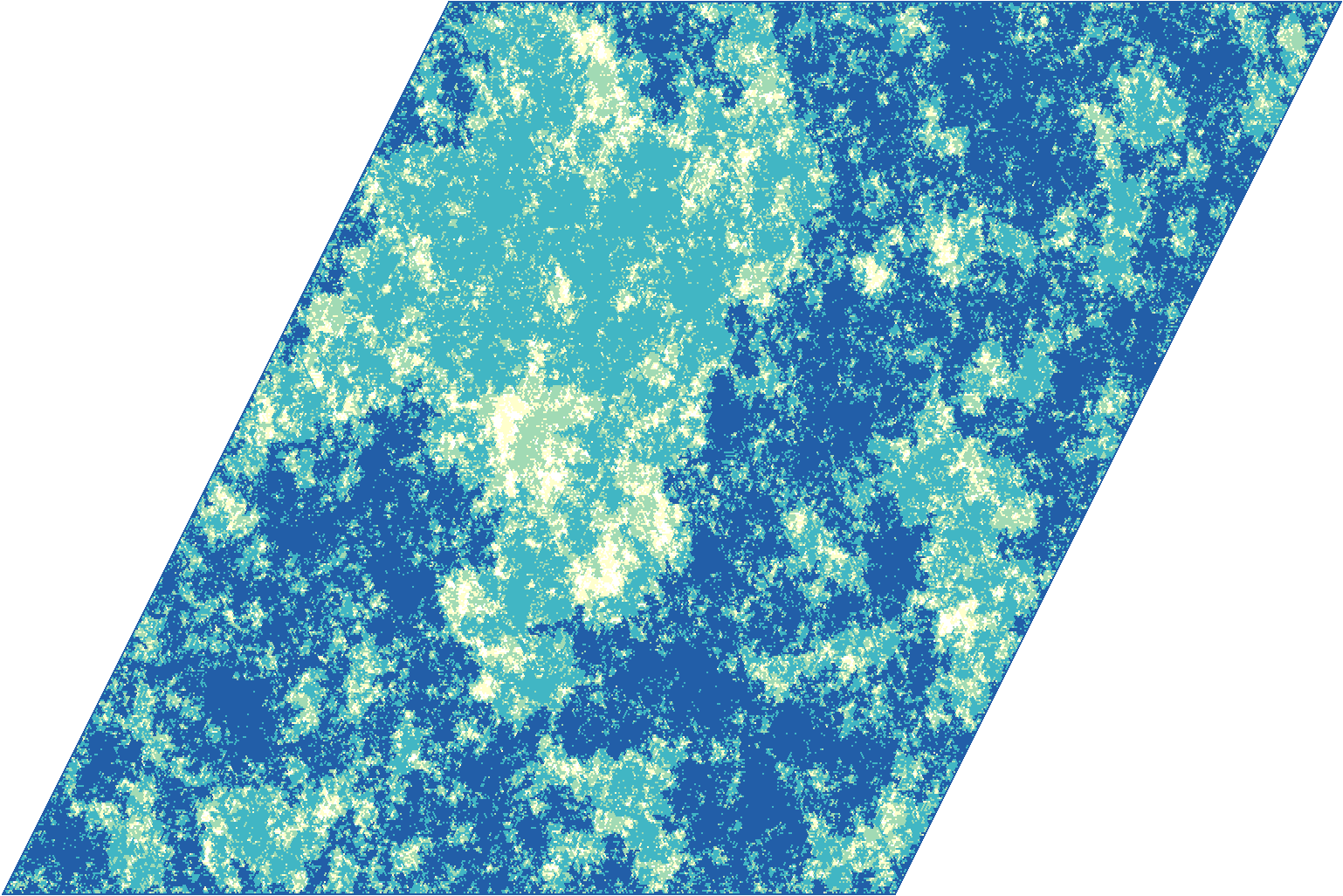}
			\caption {A simulation of four nested First passage sets. The first passage set $\A_{-\lambda}$ with $\lambda=\sqrt{\pi/8}$ is in dark blue. The difference between $\A_{-2\lambda}$ and $\A_{\lambda}$ is in lighter blue, difference between $\A_{-2\lambda}$ and $\A_{-3\lambda}$ in green and yellow depicts the missing part of $\A_{-4\lambda}$. Image done by B. Werness.}
			\label{FPS}
		\end{figure}
		In fact the relation to two-valued sets is even stronger: we will show that {\it the intersection of two FPS $\A_{-a}$ and $\AB_{b}$ is precisely $\A_{-a,b}$ in the simply connected case}. Proposition \ref{EaEbAab} generalizes this result to multiply connected domains and it is the key to show the uniqueness of TVS.
		
		One can show that {\it $\A_{-a}$ has zero Lebesgue measure} but, contrary to $\A_{-a,b}$,  {\it its Minkowski dimension is 2 and $\A_{-a}$ is not a thin local set}, i.e., $\Phi$ ``charges'' $\A_{-a}$. Also, quite surprisingly {\it $\Phi_{\A_{-a}}$ is measurable function of just the set $\A_{-a}$ itself}. Even more, {\it this measure is in fact equal to one half times its Minkowski content measure in the gauge $r\mapsto |\log(r)|^{1/2} r^2$} (Proposition \ref{prop:: mes LQG}). Notice that in fact both of these statements are non-trivial! Our proof uses the recent construction of Liouville quantum gravity measures via the local sets \cite{APS}, the fact that GFF is a measurable function of of its Gaussian multiplicative chaos measures \cite{BSS}, and a deterministic argument to link different measures on fractal sets, that could turn out to be useful in a more general setting. As a cute consequence we observe in Corollary \ref{cor::cute} that {\it the GFF can be seen as a limit of recentered Minkowski content measures of a sequence of growing random sets}.

		Finally, our techniques allow us also to compute explicitly the laws of several observables. In Propositions \ref{LawELBddFPS} and \ref{LawELBdd}, we compute the extremal distance between FPS or TVS started from a given boundary to the rest of the boundary. This is the continuum analogue of some of the results obtained in the metric graph setting \cite{LupuWerner2016Levy}, where the extremal distance replaces the effective resistance. In an upcoming paper \cite{ALS3}, we further use these techniques to calculate the law of the extremal distance between the CLE$_4$ loop surrounding zero and the boundary, and the joint laws between different nested loops.

		\subsection{Outline of the paper}
		
		The rest of the article is structured as follows:
		
		Section \ref{sectC} contains the preliminaries: a summary of general potential theory objects, two-dimensional continuum GFF, local sets and basic results about Gaussian multiplicative chaos. The only novel parts are Propositions \ref{BProcess}, \ref{BProcess2} that heuristically allow us to parametrize local set processes using their distance to a part of the boundary. 
		
		In Section \ref{section:BTLS},  we extend the theory of two-valued local set to the finitely-connected case. This will require a detailed study of the generalized level lines in multiply-connected domains. After that, in Section \ref{sec:: FPS}, we define and characterize the continuum FPS and prove several of its properties. Finally, in Section \ref{sec:: Minkowski} we show that the measure $\Phi_{\A_{-a}}+a$ corresponds to a constant times the Minkowski content measure (in a certain gauge) of the underlying set.
		

\section{Preliminaries}\label{sectC} 
In this section, we describe the underlying objects and their key properties. First, we go over the conformally invariant notion of distance in complex analysis - the extremal length; then we discuss the continuum two-dimensional GFF and its local sets. The only new contributions of this section are Proposition \ref{BProcess} and Proposition \ref{BProcess2}.

We denote by $D\subseteq \mathbb{C}$ an open planar bounded domain with a non-empty and non-polar boundary. Here, by a non-polar set we mean a set on the plane, that is a.s. not hit by a Brownian motion started from outside of this set (see e.g. Section 8.3 in \cite{PMBM}). By conformal invariance, we can always assume that $D$ is a subset of the unit disk $\D$.
The most general case that we work with are domains $D$ such that the complement of $D$ has at most finitely many connected component and no complement being a singleton. Recall that by the Riemann mapping for multiply-connected domains \cite{Ko2}, such domains $D$ are known to be conformally equivalent to a circle domain (i.e. to $\D \backslash K$, where $K$ is a finite union of closed disjoint disks, disjoint also from $\partial \D$). 

\subsection{Extremal distance and conformal radius}\label{secED}

In multiply-connected domains, the natural way to measure distances between the components is the \textit{extremal length} (it is a particular case of \textit{extremal distance}) and its reciprocal \textit{conformal modulus}. Both of the quantities are conformally invariant and extremal distance is the analogue of the effective resistance on electrical networks \cite{Duffin1962ELRes}. We introduce it shortly here and refer to \cite{Ahlfors2010ConfInv}, Section 4 for more details.

If $\rho(z)\vert dz\vert$ is a metric on $D$ conformally equivalent to the Euclidean metric, we will denote by
\begin{displaymath}
\operatorname{Length}_{\rho}(\gamma):=\int_{\gamma}\rho(z)\vert dz\vert
\end{displaymath}
the $\rho$-length of a path $\gamma$, and by
\begin{displaymath}
\operatorname{Area}_{\rho}(D):=\int_{D}\rho(z)^{2}dz
\end{displaymath}
the $\rho$-area of $D$. Now, let $\B_{1}$ and $\B_{2}$ be unions of finitely many boundary arcs of $D$, such that $d(\B_{1},\B_{2})>0$. The extremal distance between
$\mathcal{B}_{1}$ and $\mathcal{B}_{2}$ is defined as
\begin{displaymath}
\operatorname{EL}(\mathcal{B}_{1},\mathcal{B}_{2})=
\sup_{\rho}
\inf_{\substack{\gamma~\text{connecting}\\
		\mathcal{B}_{1}~\text{and}~\mathcal{B}_{2}}}
\dfrac{\operatorname{Length}_{\rho}(\gamma)^{2}}
{\operatorname{Area}_{\rho}(D)}.
\end{displaymath}
The conformal modulus $\M(\mathcal{B}_{1},\mathcal{B}_{2})$ is then defined as $\ED(\mathcal{B}_{1},\mathcal{B}_{2})^{-1}$. We state here also a theorem giving an explicit formula for the extremal distance using the Dirichlet energy .

\begin{thm}[Theorem 4-5 of \cite{Ahlfors2010ConfInv}]
	\label{thmEL}
	Let $D$ be finitely connected, $\B_{1}$ and $\B_{2}$ be unions of finitely many boundary arcs, such that $d(\B_{1},\B_{2})>0$. 
	
	If $\partial D\backslash (\B_{1}\cup\B_{2})$ is piece-wise smooth, then $M(\B_{1},\B_{2}) = \ED(\mathcal{B}_{1},\mathcal{B}_{2})^{-1}$ is given by the Dirichlet energy 
	$\int_D |\nabla \bar u|^2$ of the harmonic function $\bar u$ equal to $0$ on $\B_{1}$, $1$ on $\B_{2}$, and having zero normal derivative on
	$\partial D\backslash (\B_{1}\cup\B_{2})$.

 If $\B_1$, resp. $\B_2$, has piecewise smooth boundary, then $\int_D |\nabla \bar u|^2$ is equal to $-\int_{\B_1} \partial_n \bar u$, resp. $\int_{\B_2} \partial_n \bar u$, where $\partial_n$ is the outward derivative.
\end{thm}

This theorem gives in particular a relation between the extremal distance and the boundary Poisson kernel. To explain this, we define the Green's function $G_D$ of the Laplacian (with Dirichlet boundary conditions) in $D$. It is often useful to write 
\begin{equation}
\label{EqLogSing}
G_D(z,w) = (2\pi)^{-1} \log(1/|z-w|) + g_{D}(z,w),
\end{equation} 
where $g_{D}(z,\cdot)$ is the bounded harmonic function with boundary values given by $(2\pi)^{-1} \log(|z-x|)$ for $x \in \partial D$. It can be shown that the Green's function is conformally invariant. Additionally, note that in simply connected domains, $g_{D}(z,z)$ equals the log conformal radius:
\begin{displaymath}
g_{D}(z,z)=\frac{1}{2\pi}\log(\operatorname{CR}(z,D)).
\end{displaymath}

The Green's function can be used to define the Poisson kernel and the boundary Poisson kernel. In the case of domains with locally analytic boundary, for $z \in D$ and $x \in \partial D$ the Poisson kernel is given by:
\begin{equation}\label{l.Poisson Kernel}
P_{D}^z(x) := -\partial_{n_x} G_D(x,z) = \lim_{\eps \to 0} \eps^{-1}G_D(x-n_x\eps,z),
\end{equation}
where $n_x$ is the outward unit normal vector at $x$, and $\partial_{n_x}$ is the outward normal derivative in the first component. 
Similarly, the boundary Poisson kernel equals:
\begin{equation}\label{l.Boundary Poisson Kernel}
H_{D}(x,y)=\partial_{n_{x}}\partial_{n_{y}} G_{D}(x,y),~
x,y\in\partial D,
\end{equation}
where $\partial_{n_{x}}$ and $\partial_{n_{y}}$ are the normal derivatives in the first and second component respectively. If $D$ and $D'$ are domains with locally analytic boundaries and $f$ is a conformal transformation from $D$ to $D'$, then
\begin{equation}
\label{EqConfCovP}
P_{D'}^{f(z)}(f(x))= \vert f'(x)\vert P_{D}^{z}(x)
\end{equation}
and
\begin{equation}
\label{EqConfCovHD}
H_{D'}(f(x),f(y))= \vert f'(x)\vert \vert f'(y)\vert H_{D}(x,y).
\end{equation}
One can see the Poisson kernel as a measure on $\partial D$ by setting
\begin{displaymath}
P_{D}^z(dx)=P_{D}^z(x) dx.
\end{displaymath}
It then equals the harmonic measure in $D$ seen from $z$ and solves the Dirichlet problem: for any bounded harmonic function $u$ we have that 
$$u(z) = \int_{\partial D}u(x)P_D^z(dx).$$
Similarly, one can see the boundary Poisson kernel as a measure on
$\partial D\times\partial D$ rather than a function, by setting
\begin{displaymath}
H_{D}(dx,dy)=H_{D}(x,y) dx dy,
\end{displaymath}
where on the right-hand side $dx$ and $dy$ denote the length measure on
$\partial D$. Both $P_D^z(dx)$ and $H_D(dx,dy)$ are conformally invariant by \eqref{EqConfCovP} and
\eqref{EqConfCovHD}. In addition, $H_D(dx,dy)$ has infinite total mass due to diagonal divergence. For any domain $D$ with locally connected boundary, we can define the Poisson kernel $P_D^z(dx)$ and the boundary Poisson kernel $H_{D}(dx,dy)$ as the push-forward measures of, respectively, the Poisson kernel and boundary Poisson kernel on a domain $D'$ with locally analytic boundary, under a conformal transformation taking $D' \to D$. This is true even in the case 
$\partial D$ has locally infinite length, e.g. in the case where the boundary ``looks like'' an $\hbox{SLE}_{4}$ curve.

Finally, notice that from the definition using the Green's function and Theorem \ref{thmEL} we see that the extremal length introduced above can be expressed using the boundary Poisson kernel. Indeed, let $\B$ be a union of finitely many boundary components. Then
\begin{equation} \label{Modulus}
M(\B,\partial D \backslash \B)=\ED(\B,\partial D \backslash \B)^{-1}=
\iint_{\B\times\partial D \backslash \B}
H_{D}(dx,dy) 
,
\end{equation}
where the last equality follows from Theorem \ref{thmEL}.
In general, if $\B_{1},\B_{2}\subseteq\partial D$ are disjoint,
\begin{displaymath}
M(\B_{1},\B_{2})=\ED(\B_{1},\B_{2})^{-1}\geq
\iint_{\B_{1}\times\B_{2}}
H_{D}(dx,dy)
.
\end{displaymath}
\subsection{The continuum GFF}

The (zero boundary) Gaussian Free Field (GFF) in a domain $D$  can be viewed as a centered Gaussian process $\Phi$ 
indexed by the set of continuous functions with compact support in $D$, with covariance given by the Green's function:
\begin{equation*} \E [(\Phi,f_{1}) (\Phi,f_{2})]  =  
\iint_{D\times D} f_{1}(z) G_D(z,w) f_{2}(w) \d z \d w.\end{equation*}

In this paper $\Phi$ always denotes the zero boundary GFF. We also consider GFF-s with non-zero Dirichlet boundary conditions - they are given by $\Phi + u$ where $u$ is some bounded harmonic function that is piecewise constant\footnote{Here and elsewhere this means piecewise constant that changes only finitely many times.} boundary data on $\partial D$.

Because the covariance kernel of the GFF blows up on the diagonal, it is impossible to view $\Phi$ as a random function. However, it can be shown that the GFF has a version that lives in the Sobolev space 
$H^{-1}(D)$ of generalized functions, justifying the notation $(\Phi,f)$ for $\Phi$ acting on functions $f$ (see for example \cite{Dub}). In fact, for any domain $D$ included in the unit disk $\D$, a GFF $\Phi$ in $D$ also belongs to $H^{-1}(\D)$ and the expected value of the square of its norm is uniformly bounded by
\begin{equation}\label{e.H-1 norm}
\E\left[\|\Phi\|_{H^{-1}(\D)}^2 \right]=\iint_{\D\times \D} G_D(z,w)G_{\D}(z,w)dzdw\leq \iint_{\D\times \D} G_{\D}^2(z,w)dzdw<\infty.
\end{equation}

Moreover, let us remark that it is in fact possible and useful to define the random variable $(\Phi,\mu)$ for any fixed Borel measure $\mu$, provided the energy 
$\iint \mu (dz) \mu (dw) G_D (z,w)$ is finite.

Finally, there is an important numerical constant $\lambda$, that depends on the normalization of the GFF and is in the current setting (where $G_D(z,w) \sim -(2\pi)^{-1}\log(|z-w|)$ as $w\to z$) is given by 
\begin{equation*}
\lambda=\sqrt{\pi/8},
\end{equation*}
where $2\lambda$ is the \textit{height gap} of the GFF
\cite{SchSh,SchSh2}.
Sometimes, other normalizations are used in the literature: if ${G_D (z,w) \sim c \log(1/|z-w|)}$ as $z \to w$, then $\lambda$ should 
be taken to be $(\pi/2)\times \sqrt {c}$.

\subsection{Local sets: definitions and basic properties}
Let us now discuss more thoroughly the local sets of the GFF, that were introduced in Lemma 3.9 of \cite{SchSh2}. We only discuss items that are directly used in the current paper. For a more general discussion of local sets and thin local sets (not necessarily of bounded type), we refer to 
\cite{SchSh2,WWln2,Se}.

Even though, it is not possible to make sense of $(\Phi,f)$ when $f=\1_A$ is the indicator function of an arbitrary random set $A$, local sets form a class of random sets where this is (in a sense) possible.

\begin{defn}[Local sets]\label{d.local}
	Consider a random triple $(\Phi, A,\Phi_A)$, where $\Phi$ is a  GFF in $D$, $A$ is a random closed subset of $\overline D$ and $\Phi_A$ a random distribution that can be viewed as a harmonic function when restricted to
	$D \backslash A$.
	We say that $A$ is a local set for $\Phi$ if conditionally on $(A,\Phi_A)$, $\Phi^A:=\Phi - \Phi_A$ is a  GFF in $D \backslash A$. 
	\end {defn}
	Throughout this paper, we use the notation $h_A: D\rightarrow\R$ for the function that on $D\backslash A$ is harmonic and equal to $\Phi_A$, and is set to $0$ on $A$. We will sometimes also talk about the values of $h_A$ on the boundary of $D \backslash A$ and by this we mean the extension of this harmonic function $h_A$ to the (prime-end) boundary. This extension to the boundary does not necessarily exist for all local sets, but in the current paper we only ever use this notation when we a priori know that $h_A$ is given on the boundary by a piece-wise constant function, changing value finitely many times.
		
	Let us list a few properties of local sets that are used in this paper:.
	\begin{lemma}\label{BPLS}    $\ $
		\begin {enumerate}
		\item Any local set can be coupled in a unique way with a given GFF: Let $(\Phi,A,\Phi_A,\widehat \Phi_A)$ be a coupling where $(\Phi,A,\Phi_A)$ and $(\Phi,A,\widehat \Phi_A)$ satisfy the conditions of Definition \ref{d.local}. Then, a.s. $\Phi_A=\widehat \Phi_A$. Thus, being a local set is a property of the coupling $(\Phi,A)$, as  $\Phi_A$ is a measurable function of $(\Phi,A)$. 
		\item If $A$ and $B$ are local sets coupled with the same GFF $\Phi$, and $(A, \Phi_A)$ and $(B, \Phi_B)$ are conditionally independent given $\Phi$, then $A \cup B$ is also a local set coupled with $\Phi$. Additionally, $B\backslash A$ is a local set of $\Phi^A$ with $(\Phi^A)_{B\backslash A} = \Phi_{B\cup A}-\Phi_{A}$.
		\item Let $(\Phi,A_n)$ be such that for all $n\in \N$, $(\Phi,A_n)$ is a local set coupling, the sets $A_n$ are non-decreasing in $n$, measurable w.r.t. $\Phi$, such that the cardinal of connected components of $A_n\cup \partial D$ is uniformly bounded in $n$, and each connected component is larger than a point. Then, $\overline{\bigcup  A_n}$ is also a local set and $\Phi_{A_n}\to\Phi_{\overline{\bigcup  A_n}} $ in probability as $n\to \infty$, in, say, $H^{-1-\eps}(D)$ for bounded $D$ and in $H_{loc}^{-1-\eps}(D)$ otherwise.
	\end{enumerate}
\end{lemma}
	
\begin{proof}
Property (1) comes from Lemma 3.9 in \cite{SchSh2}, and property (2) from Lemma 3.10 in the same paper. So let us only explain the property (3) for bounded $D$: the convergence in law of $(\Phi, A_n, \Phi_{A_n}, \Phi^{A_n})$ follows because:
\begin{itemize}
	\item $A_n$ is non-decreasing as thus converges to some $A$ (given by $\overline{\bigcup  A_n}$);
	\item $\Phi^{A_n}$, conditioned on $A_n$ convergence in law in $H^{-1}(D)$ to the GFF $\Phi^{A}$ in $D \backslash A$: indeed, by calculating the expected $H^{-1}(D)$-norm we get tightness (using \eqref{e.H-1 norm}), and the Beurling estimate (Theorem 3.76 of \cite{LawC}) ensures that $G_{D\backslash A_n}\to G_{D\backslash A}$ as $n\to \infty$;
	\item For any smooth $f \in \CC_0^\infty(D)$, $(\Phi_{A_n},f)$ is a martingale with $\E\left[(\Phi_{A_n},f)^2\right]\leq \E[(\Phi,f)^2]$;
	\item Finally, as $\Phi_{A_n}$ is harmonic in $D\backslash A_n$, and for any $z \in D\backslash A$, we have $z \in D\backslash A_n$ for $n$ large enough, we obtain that $\Phi_A$ is harmonic in $D \backslash A$.
\end{itemize}\end{proof}
Often one is interested in a growing family of local sets, which we call local set processes.

\begin{defn}[Local set process]
	We say that a coupling $(\Phi,(\eta_t)_{t\geq 0})$ is a local set process if $\Phi$ is a GFF in $D$, $\eta_0\subseteq \partial D$, and $\eta_t$ is an increasing continuous family of local sets such that for all stopping time $\tau$ of the filtration $\F_t:=\sigma(\eta_s:s\leq t)$, 
	$(\Phi,\eta_\tau)$ is a local set.
\end{defn}

Let us note that in our definition $\eta_t$ is actually a random set. In the rest of the paper, we are mostly interested in local set processes that are equal to the trace of a continuous curve. In those cases, we are going to denote by $\eta(t)$ the tip of the curve at time $t$. In other words, in our notation $\eta_t=\eta([0,t])$.

Local processes can be naturally parametrized from the viewpoint of any interior point $z$: the expected height $h_{\eta_t}(z)$ (i.e. the harmonic part of $\Phi_{\eta_t}$ as defined just after Definition \ref{d.local}) then becomes a Brownian motion. More precisely, we have that:

\begin{prop}[Proposition 6.5 of \cite{MS1}]\label{BMinterior}
	For any $z \in D$ if $(\eta_t)_{t\geq 0}$ is parametrized such that 
	$(G_{D}- G_{D\backslash \eta_t})(z,z)=t$, then $(h_{\eta_t}(z))_{t\geq 0}$ has (a modification with) the law of a Brownian motion. 
\end{prop}

Whereas in \cite{MS1} this was proved only in the simply-connected setting, the proof extends directly to the multiply-connected setting as well. Also, notice that the normalization of the GFF in \cite{MS1} differs from ours.

\begin{rem}
	Notice that whereas $G_D$ diverges on the diagonal, the difference of Green's functions can be given a canonical sense, using
	\eqref{EqLogSing}. In fact when $D$, and $D\backslash \eta_t$ are simply connected domains, it is a difference of logarithms of conformal radii: 
	$$(  G_{D}- G_{D\backslash \eta_t})(z,z)=
	\dfrac{1}{2\pi}\log (\operatorname{CR}(z,D))-
	\dfrac{1}{2\pi}
	\log(\operatorname{CR}(z,D\backslash \eta_t)).$$ 
\end{rem}

In fact, one can also parametrize local set processes $\eta_t$ using their distance to the boundary. As the boundary values of the GFF away from $\eta_t$ do not change, it is natural to look at normal derivatives. In order to obtain a conformally invariant quantity, notice that if $\B \subseteq \partial D$ and $h$ is a harmonic function, then by Green's identities the quantity $\int_{\B} \partial_n h$ can be given a conformally invariant meaning: $\int_{\B} \partial_n h = \int_D \nabla h \nabla \bar u$ where $\bar u$ is the harmonic extension of the function that takes the value $1$ on $\B$ and $0$ on $\partial D\backslash \B$. 

We will first consider the case, where the local set process is parametrized by its extremal distance to a whole boundary component. Recall the notation $h_A: D\rightarrow\R$ just below the Definition \ref{d.local}.

\begin{prop} \label{BProcess} Let $D$ be finitely connected circle-domain and $(\Phi,\eta_t)$ be a local set process with $\Phi $ a GFF in $D$. Take $\B \subseteq \partial D$ be a union of finitely many boundary components. Then, if $\eta_t$ is parametrized by its conformal modulus, i.e. such that $$t = \M(\B, \mathcal (\partial D \cup \eta_t) \backslash \B) - \M(\B, \partial D \backslash \B)$$ then $W_t :=\int_{\B} \partial_n h_{\eta_t}$ has (a modification with) the law of a standard Brownian motion started from $0$.
	
	Equivalently, when parametrized by the extremal length $$t=\ED(\B, \partial D \backslash \B)-\ED(\B, \mathcal (\partial D \cup \eta_t) \backslash \B),$$ the process $$\widehat{W}_t := \ED(\B, \mathcal (\partial D \cup \eta_t) \backslash \B)\int_{\B} \partial_n h_{\eta_t},$$ has (a modification with) the law of a Brownian bridge from 0 to 0 with length $\ED(\B, \mathcal \partial D \backslash \B)$.
	
	Moreover, the same holds (with the appropriate definitions) for any finitely connected domain $D$ with all boundary components larger than a point. 
\end{prop}

\begin{proof}
	Using the conformal invariance both of the quantity $W_t$, the Gaussian free field and the extremal length, it suffices to work in a circle domain and consider the case where $\B$ is the union of a subset of the circles. In fact, for simplicity, we only prove the case where $\B=\partial B(0,1)$ and $D=B(0,1)\backslash \bigcup_{i=1}^k B(x_i,r_i)$, with $|x_i|<1-r_i$, and a $\eta_t$ is a local set process started from an interior boundary, i.e. from a boundary different from $\partial B(0,1)$. The general case follows by exactly the same argument.
	
	We claim the following.
	\begin{lemma}\label{l.characterstic function}For all $t > 0$, we have that 
		\[\E \left[ \exp\left (\lambda i\int_{\B} \partial_n h_{\eta_t}(z) dz\right) \mid \eta_t\right] = \exp \left(-\frac{\lambda^2}{2}(\M(\B, \mathcal (\partial D \cup \eta_t) \backslash \B) - \M(\B, \partial D \backslash \B))\right )\]	
	\end{lemma}
	Notice that	this lemma implies the first part of the proposition: indeed, for any $t_0$, by redefining $D' = D\backslash \eta_{t_0}$, the claim implies that conditioned on $\eta_{t_0}$ the increment $W_{t+t_0} - W_{t_0}$ is a Gaussian of variance $t$. This implies that $W_t$ has the same finite-dimensional distributions as the standard Brownian motion.

	\begin{proof}[Proof of Lemma \ref{l.characterstic function}]
	As we are working in a circle domain, all circles $\B_\eps := \partial B(0,1-\eps)$ are contained in $D$ for all $\eps > 0$ small enough. Notice that as the boundary value of $h_{\eta_t}$ on $\B$ equals zero, we have that $\eps^{-1}\int_{\B_\eps}  h_{\eta_t}$ converges a.s. to $\int_{\B}\partial_n h_{\eta_t}$. Moreover, we can write $h_{\eta_{t}} = \Phi - \Phi^{\eta_t}$, and for any $\epsilon>0$ we have that: 
	\begin{align*}
	&\sigma^2_\eps:=\E\left[\left (\eps^{-1}\int_{\B_\eps}  \Phi\right )^2 \right]= \eps^{-2}\iint_{{\B_\eps}\times {\B_\eps}}G_{D}(z,w)dz dw,\\
	&\sigma_{\eps,t}^2:=\E\left[\left (\eps^{-1}\int_{\B_\eps}  \Phi^{\eta_t}\right )^2 \mid \eta_t \right]= \eps^{-2}\iint_{{\B_\eps}\times {\B_\eps}}
	G_{D\backslash \eta_t}(z,w) dz dw.
	\end{align*}
	When $\eps \to 0$, both terms individually diverge. However, we will see that the difference converges.	
	\begin{claim}\label{c.Modulus}
		Suppose $\eta_t$ is any fixed set connected to $\partial D\backslash \B$ and at  distance $\delta>0$ from $\B$. Then, as $\eps \to 0$, the difference $\sigma_\eps^2 - \sigma_{\eps,t}^2$ converges to
		$$\M(\B, \mathcal (\partial D \cup \eta_t) \backslash \B) - \M(\B, \partial D \backslash \B).$$
	\end{claim}
	
	Let us postpone the proof of the claim, and show how to conclude the lemma from it. We start by noting that 
		\[\exp\left (-\frac{\lambda^2}{2}\sigma_{\eps}^2\right ) = \E\left[ \exp\left (\lambda i\eps^{-1}\int_{\B_\eps} \Phi\right )\right].\]
		But, when first conditioning on $\eta_t$ this also equals 
		\[\E\left[\E\left [\exp\left (\lambda i\eps^{-1}\int_{\B_\eps} (h_{\eta_t}+\Phi^{\eta_t})\right )\mid \eta_t \right]\right] = \E\left[ \exp\left (-\frac{\lambda^2}{2}\sigma_{\eps,t}^2\right )\E\left( \exp\left(\lambda i\eps^{-1}\int_{\B_\eps} h_{\eta_t}(z) dz\right)\mid \eta_t \right)\right].\]
		In particular,
		\[\E\left[ \exp\left( -\frac{\lambda^2}{2}\left (\sigma_{\eps,t}^2-\sigma_{\eps}^2\right ) \right)\E\left[ \exp\left(\lambda i\eps^{-1}\int_{\B_\eps} h_{\eta_t}(z) dz\right)\mid \eta_t \right]\right] = 1\]
		Using now the a.s.  convergence of 
		$\eps^{-1}\int_{\B_\eps} h_{\eta_t}(z)\to \int_{\B}\partial_n h_{\eta_t}(z) dz$, and the limiting expression for $\sigma_\eps^2 - \sigma_{\eps,t}^2$  the lemma follows by dominated convergence.
	
	To finish the proof of Lemma \ref{l.characterstic function}, we now prove the Claim \ref{c.Modulus}.

	\begin{proof}[Proof of Claim \ref{c.Modulus}]
	Indeed, consider any $x,y \in \B$. Then writing $G_D(z,w) = -(2\pi)^{-1}\log \lvert z - w \rvert + g_D(z,w)$, we have that
	$$u(w,z) = G_D(z,w) -G_{D \backslash \eta_t}(z,w) = g_D(z,w) - g_{D \backslash \eta_t}(z,w)$$
	is a harmonic function both for fixed $w, z \in D\backslash \eta_t$. The boundary values (for fixed $w$) are $0$ on $\partial D$ and are continuous and bounded on $\eta_t$. In particular, if $n_x, n_y$ are outward unit normal vectors at $x,y$, then
	\[\eps^{-2}\left(G_D(x-n_x \eps, y - n_y \eps) - G_{D\backslash \eta_t}(x-n_x \eps, y - n_y \eps)\right) = \eps^{-2}u(x-n_x \eps, y - n_y \eps).\]
	As $u$ is harmonic and bounded in whole of $D \backslash \eta_t$ both in $z,w$, and equals $0$ when either $z \in \B$ or $w \in \B$, the derivative $\partial_{n_x} \partial_{n_y} u(x,y)$ exists for all $x,y \in \B^2$, is continuous jointly in both $x, y$ and is equal to the above limit \cite{Kell}.  Thus in particular the difference $\sigma_\eps^2 - \sigma_{\eps,t}^2$ converges almost surely.
	
	It remains to argue that this limit equals $\M(\B, \mathcal (\partial D \cup \eta_t) \backslash \B) - \M(\B, \partial D \backslash \B)$.	First notice that the above considerations imply that the limit of $\sigma_\eps^2 - \sigma_{\eps,t}^2$ also equals the limit of
	$$\eps_1^{-1}\eps_2^{-1}\iint_{{\B_{\eps_1}}\times {\B_{\eps_2}}}\left(G_{D}(z,w)-G_{D \backslash \eta_t}(z,w)\right)dz dw$$
	when we first let $\eps_2 \to 0$ and then let $\eps_1 \to 0$.
	
	Now, as $\B$ is locally analytic, for any $z \in D$, by definition $\eps_2^{-1}G_D(z, x-n_x\eps_2)$ converges to the Poisson kernel $P_D^z(x)$ (see Section \ref{secED}). Hence, denoting by $\bar u$ and $\bar u_t$ the harmonic functions in $D$, $D \backslash \eta_t$ respectively, with boundary condition $1$ on $\B$ and $0$ elsewhere, we can write for any $z\in D$
$$\lim_{\eps_2 \to 0}\eps_2^{-1}\int_{\B_{\eps_2}}\left(G_{D}(z,x-n_x\eps_2)-G_{D \backslash \eta_t}(z,x-n_x\eps_2)\right)dx = \bar u(z) - \bar u_t(z).$$
For $\widetilde u = 1 - \bar u$, $\widetilde u_t = 1 - \bar u_t$ we have that 
$$\lim_{\eps_1 \to 0}\eps_1^{-1}\int_{\B_{\eps_1}} \widetilde u_t(x-n_x \eps_1) - \widetilde u(x-n_x\eps_1) dx = \int_{\B}  (\partial_{n_x}\widetilde u(x) - \partial_{n_x}\widetilde u_t(x))dx.$$
Thus 
$$\lim_{\eps_1 \to 0}\eps_1^{-1}\int_{\B_{\eps_1}} \bar u(x-n_x \eps_1) - \bar u_t(x-n_x\eps_1) dx = \int_{\partial_{\B}}  (\partial_{n_x}\bar u_t(x) - \partial_{n_x}\bar u(x))dx.$$
Finally, by Theorem \ref{thmEL} $\int_{\B} \partial_{n_x}\bar u(x) dx = M(\B, \partial D \backslash \B)$, and thus the claim follows.
\end{proof}

	\end{proof}

	It remains to justify the second part of the proposition. This follows from the following general claim:
	\begin{claim}
		Let $t_0$ be positive and suppose that $(B(t-t_0))_{t \geq t_0}$ has the law of a standard Brownian motion (started from $0$). Let $s_0 = t_0^{-1}$. Then $$\widehat B(s) := (s_0-s)B\left (\frac{1}{s_0-s}-\frac{1}{s_0}\right )$$ has the law of a Brownian bridge on $[0,s_0]$ from $0$ to $0$ .
	\end{claim}

	\begin{proof}
	This just follows by calculating the covariance for $s \leq t \in [0, s_0]$:
	\[\E \left[\widehat B(s)\widehat B(t)\right] = (s_0-s)(s_0-t)\left (\frac{1}{s_0-s}-\frac{1}{s_0}\right ) = (s_0-t)\frac{s}{s_0}. \qedhere\]
	\end{proof}
\end{proof}

	Now, let us see how to extend this proposition to the case where the local set process is parametrized by its ``distance'' to a part of the boundary. One of the obstacles here is that when the growing set that we want to parametrize is of $0$ Euclidean distance from this part of the boundary, then the naive conformal modulus between the set and this boundary part diverges. However, similarly to the reduced extremal distance (see Chapter 4.14 of \cite{Ahlfors1966CA}), the difference between the moduli is still non-trivial. 
	
	Let us explain this a bit more precisely: assume that $\partial D$ can be partitioned as $\B_1\cup \B_2$, where $\B_2$ is a connected subset of $\partial D$ (that is not necessarily a whole boundary component). Let $B\subseteq \overline{D}$ a closed set that remains at positive distance from $\B_2$ and that intersects $\B_1$. Define as $u_0$ to be the bounded harmonic function that takes value $1$ on $\B_2$ and $0$ on $\B_1$, and $u_B$ to be the bounded harmonic function taking again value $1$ on $\B_2$ but is equal to $0$ on $\B_1\cup B$. Then $u_B-u_0$ is a bounded harmonic function that takes values $0$ in $\B_1\cup \B_2$ and $-u_0<0$ in $B$. In particular,
	\[ 0\leq \int_{\B_2} \partial_n(u_B-u_0)<\infty.\]
	even when the conformal modulus between $\B_1$ and $\B_2$ is 0. Thus, we can define $M(\B_2,\B_1\cup B)-M(\B_2,\B_1)$ as $\int_{\B_2} \partial_n(u_B-u_0)$ even when both terms individually are infinite.
	
	Using this observation, it is possible to use a proof similar to that of  Proposition \ref{BProcess} to parametrize local sets using only a part of a boundary component:
	
	\begin{prop} \label{BProcess2} Let $D$ be a finitely connected circle domain, $(\Phi,\eta_t)$ a local set process with $\Phi $ a GFF in $D$. Now, let us partition $\partial D$ in two sets: $\B_1$ and $\B_2$, where $\B_2$ is a connected subset of $\partial D$ (not necessarily a whole boundary component) and suppose that $\eta_0\subseteq \B_1$ and $d(\eta_0,\B_2)>0$.
		
		Then, if $\eta_t$ is parametrized by the difference of conformal moduli (as the individual conformal moduli may not exist)
		\[t=M(\B_2, \B_1 \cup \eta_t)-M(\B_2, \B_1):= \int_{\B_2}\partial_n (\bar u_t-\bar u_0),\]
		and $\bar u_t$ is defined as the harmonic function taking the value $1$ on $\B_2$ and the value $0$ on $\B_1\cup \eta_t$. Then, the process $W_t=\int _{\B_2} \partial_n h_{\eta_t}$ has (a modification with) the law of a Brownian motion. 
		
		Moreover, the same holds (with the appropriate definitions) for any finitely connected domain with all boundary components larger than a point.
	\end{prop}

	\subsection{Gaussian multiplicative chaos}\label{GMC} Finally, let us summarize the definition and some properties of the Gaussian multiplicative chaos (GMC) associated to the GFF. The GMC measures were first introduced in the realm of self-interacting Euclidean field theories \cite{HK}, and named by Kahane in his seminal article \cite{Kah}. We refer to e.g. \cite{Aru2017LQG} for more detailed proofs and properties for the GMC measures, and to \cite{RhodesVargasReview} for an overview of the GMC measures and their applications.
	
	To define the GMC measure, one usually passes through an approximation procedure. Denote by $\Phi_\epsilon(z)$ the circle-average process of the GFF in $D$: i.e. the GFF tested against the unit measure on the circle of radius $\epsilon$ around $z$. For $\epsilon>0$ and $\gamma \in \R$, we then set
	\begin{linenomath}
		\begin{align*}
		d\mu^{\epsilon}_\gamma:=\eps^{\frac{\gamma^2}{2}2\pi}\exp\left (\gamma\sqrt{2\pi}\Phi_\epsilon(z)\right )dz.
		\end{align*}
	\end{linenomath}
	The $\sqrt{2\pi}$ factor comes from the fact that in the GMC literature the GFF is normalized differently (i.e. usually with covariance of that behaves like $- \log |z-w|$ near the diagonal). 
	
	In \cite{DS} it was shown that for $\gamma\in (-2,2)$, as $\epsilon$ goes to $0$ along the dyadics, $\mu^\epsilon_\gamma$ converges towards a measure $\mu_\gamma$ 
	a.s. and in $\LL^1$. Notice that for any $z \in D$ and any $\eps > 0$ small enough, we have that 
	\begin{equation}\label{LVexp}
	\E(\mu_\gamma^\eps(z)) = \crad(z,D)^{\frac{\gamma^2}{2}2\pi},
	\end{equation}
	and thus the same holds in the limit $\eps \to 0$.
	
For us it is important that the GMC measures are $\mathcal C^1$ functions of the parameter $\gamma$ (in fact, they depend analytically on $\gamma$ but this won't be needed here). More precisely, the following theorem (which can be, for example, found in Section 1 of \cite{Aru2017LQG}) suffices for our needs.
	\begin{prop}\label{thm::Derivative}
		For any $\gamma\in(-\sqrt 2,\sqrt 2)$ and any continuous compactly supported function $f$ on $D$, there exists a modification of 
		$((\mu^{\epsilon}_\gamma,f))_{\vert\gamma\vert<\sqrt{2}}$, and a deterministic sequence $\epsilon_k\to 0$ such that a.s. and in $\LL^2$, 
		$((\mu^{\epsilon_{k}}_\gamma,f))_{\vert\gamma\vert<\sqrt{2}}$ converges in the space $\mathcal{C}^1((-\sqrt{2},\sqrt{2}))$ to 
		$((\mu_\gamma,f))_{\vert\gamma\vert<\sqrt{2}}$. 
		
		In particular, the map 
		$\gamma\mapsto (\mu_\gamma,f)$ is $\mathcal{C}^1$ on 
		$(-\sqrt{2},\sqrt{2})$ and furthermore
		\begin{linenomath}
			\begin{align}\label{Derivative}
			(\Phi,f)=\lim_{k\to +\infty}(\Phi_{\epsilon_k},f)=
			\frac{1}{\sqrt {2\pi}}\lim_{k\to +\infty}\partial_{\gamma} (\mu^{\epsilon_{k}}_\gamma,f)\mid_{\gamma=0}=\frac{1}{\sqrt {2\pi}}\partial_{\gamma} (\mu_\gamma,f)\mid_{\gamma=0},
			\end{align}
		\end{linenomath}
		where all the limits are a.s. and in $\LL^2$.	
	\end{prop}
	
	\begin{rem}
		In fact, in \cite{Aru2017LQG} this is explained for a different normalization, instead of $\mu_\gamma$, one considers $\tilde \mu_\gamma$ defined by
		\[d\tilde \mu^{\epsilon}_\gamma:=\exp\left (\gamma\sqrt{2\pi}\Phi_\epsilon(z)-\frac{\gamma^2}{2}2\pi \E\left[\Phi_\epsilon(z)^2\right] \right )dz.\]
		Since $\mu_\gamma(dz) = \crad(z)^{\frac{\gamma^2}{2}2\pi}\tilde \mu_\gamma(dz)$, Proposition \ref{thm::Derivative} for $\mu_\gamma$ is equivalent to that for $\tilde \mu_\gamma$. 
	\end{rem}
	
	\begin{rem} \label{Measurability GMC} Note that from \eqref{Derivative} it follows that $\Phi$ is a deterministic measurable function of
		$(\mu_{\gamma})_{|\gamma|<\sqrt{2}}$ (see e.g. Corollary 1.6 in \cite{Aru2017LQG}). The measurability of the field w.r.t. its GMC measure was further studied in \cite{BSS}, where the authors show that in fact  $\Phi$ is a deterministic function of $\mu_\gamma$ for any fixed $\gamma\in(-2,2)$.
	\end{rem}

	\section{Two-valued local sets}\label{section:BTLS}
	Next, we discuss a specific type of local sets introduced in \cite{ASW}: two-valued local sets. In \cite{ASW}, these sets were defined and studied in the case of the zero boundary GFF; we will extend this definition to general boundary conditions and to n-connected domains. We will also calculate the size of the set seen from interior points and from boundary components that do not intersect the set.
	
	First, it is convenient to review a larger setting, that of bounded type local sets (BTLS) introduced in \cite{ASW}. These sets are thin local set $A$, for which its associated harmonic function $h_A$ remains bounded. Here, by a thin local set (see \cite{ WWln2, Se}) we mean the following condition:
	\begin{itemize}
		\item For any smooth test function $f \in \mathcal{C}^\infty_0$, the random variable $(\Phi,f)$ is almost surely equal to  $(\int_{D \backslash A }  h_A(z) f(z) dz)    + (\Phi^A,f)$. 
	\end{itemize}
	
	This definition assumes that $h_A$ belongs to $\mathbb L^1(D\backslash A)$ \, which is the case in our paper. For the general definition see \cite{Se}. In order to say that the union of two thin sets is thin, it is more convenient to use a stronger condition. Indeed, it is not hard to show that (see  Proposition 4.3  of \cite{Se} for a proof):
	\begin{itemize}
		\item If $h_A$ is $\mathbb L^1(D\backslash A)$  and for any compact set $K\subseteq D$, the Minkowski dimension of $A\cap K$ is strictly smaller than 2 then $A$ is thin. 
	\end{itemize}		 
	Now, we can define the bounded type local sets.		 
	\begin{defn}[BTLS]\label{BTLSCND}
		Consider  a  closed subset $A$ of $\overline D$  and $\Phi$ a GFF in $D$ defined on the same probability space.
		Let $K>0$, we say that $A$ is a $K$-BTLS for $\Phi$ if the following three conditions are satisfied: 
		\begin {enumerate}
		\item $A$ is a thin local set of $\Phi$.
		\item Almost surely, $|h_A| \le K$ in $D \backslash A$.  
		\item Almost surely, $A$ contains no isolated points and each connected component of $A$ that does not intersect $\partial D$ has a neighbourhood that does not intersect any other connected component of A.
		\end {enumerate}
		If $A$ is a $K$-BTLS for some $K$, we say that it is a BTLS.
		\end {defn}
		
		\subsection{Generalized level lines}

		One of the simplest family of BTLS are the generalized level lines, first described in \cite{SchSh2}, that correspond to SLE$_4(\rho)$ processes. 
		
		\begin{defn}[Generalized level line]
		Let $D:=\H\backslash \bigcup_{k=1}^n C_k$, where 
		$(C_k)_{1\leq k\leq n}$ is a finite family of disjoint closed disks, be a circle domain in the upper half plane. Further, let $u$ be a harmonic function in $D$. We say that $\eta(\cdot)$, a curve parametrized by half plane capacity, is the generalized level line for the GFF $\Phi + u$ in $D$ up to a stopping time $\tau$ if for all $t \geq 0$:
		
		\begin{description}
			\item[$(**)$]The set $\eta_{t}:=\eta[0, t \wedge \tau]$ is a BTLS of the GFF $\Phi$, with harmonic function satisfying the following properties: $h_{\eta_{t}} + u$ is a harmonic function in $D \backslash \eta_t$ with boundary values $-\lambda$  on the left-hand side of $\eta_t$, $+ \lambda$ on the right side of $\eta_t$, and with the same boundary values as $u$ on $\partial D$. 
		\end{description}
	\end{defn}
		
		The first example of level lines comes from \cite{SchSh2}: let $u_0$ be the unique bounded harmonic function in $\H$ with boundary condition $-\lambda$ in $\R-$ and $\lambda$ in $\R^+$. Then it is shown in \cite{SchSh2} that there exists a unique $\eta$ satisfying $(**)$ for $\tau= \infty$, and its law is that of an SLE$_4$. 
		
		Several subsequent papers \cite{SchSh2,MS1,WaWu,PW} have studied more general boundary data in simply-connected case and also level lines in a non-simply connected setting \cite{ASW}. The following lemma is a slight variant of the latter, stating existence of level lines until it either accumulate at another component, or hit the continuation threshold\footnote{The continuation threshold is the first time $\tau$ in which the level line hits a boundary point, $x\in \partial D$, such that there is no level line of $\Phi+u+h_{\eta_\tau}$ starting from $x$ in $O$. Here $O$ is the non bounded connected component of $D\backslash \eta_\tau$. This condition can be described explicitly using boundary values. See, for example, Definition 2.14 in \cite{WaWu}.} on $\R$. It is a consequence of Theorem 1.1.3 of \cite{WaWu} and Lemma 15 of \cite{ASW}.
		
		\begin{lemma}[Existence of generalized level line targeted at $\infty$]\label{lemext} Let $u$ be a bounded harmonic function with piecewise constant boundary data such that $u(0^-) < \lambda$ and $u(0^+)>-\lambda$. Then, there exists a unique law on random simple curves $(\eta(t), t\geq 0)$ coupled with the GFF such that $(**)$ holds for the function $u$ and possibly infinite stopping time $\tau$ that is defined as the first time when $\eta$ hits or accumulates at a point $x \in \partial D \backslash \R$ or hits a point $x\in \R$ such that $x \geq 0$ and $u(x^+) \leq -\lambda$ or $x \leq 0$ and $u(x^-) \geq \lambda$. Furthermore, $\eta$ is measurable w.r.t $\Phi$ and if $\eta(\tau) \in \R$, then $\eta$ is continuous on $[0,\tau]$.
		\end{lemma}
		
		\begin{rem}
			In Section \ref{sectlfc} we will be able to show that $\eta$ is continuous up to $\tau$ even if $\eta(\tau) \notin \R$. 
		\end{rem}
		
		In simply-connected domains Theorem 1.1.3 of \cite{WaWu} and Lemma 15 of \cite{Dub2} also give us precise information on the subset of the boundary where the level line can hit:
		
		\begin{prop}[Hitting of level lines in simply-connected domains]\label{propnotouchsc}
			Let $D = \H$ and let $u$ be a bounded harmonic function with piecewise constant boundary data such that $u(0^-) < \lambda$ and $u(0^+)>-\lambda$. 
			Let $\eta$ be a generalized level line in $\H$ starting from $0$. If either $u \geq \lambda$ or $u \leq -\lambda$ on some open interval $J \subset \R$, then $\eta_t$ stays at a positive distance of any $x \in J$. Moreover, if $u \leq -\lambda$ on a neighborhood in $\R$ to the left of $x$, or $u \geq \lambda$ on a neighborhood in $\R$ to the right of $x$, then almost surely $\eta_{t}$ stays at a positive distance of $x$.
		\end{prop}
		
		\begin{rem}
			Observe that the boundary points described by this lemma correspond exactly to points, from where one cannot start a generalized level line of $-\Phi - u$.
		\end{rem}
		
		A simple, but important corollary of this result allows us to check whether a level line can enter a connected component of the complement of a bounded-type local set. This observation was key in \cite{ASW}, where it was used for simply-connected domains and it followed just from the facts that 1) a generalized level line does not hit itself; 2) it has to exit such a component in finite time. We will prove the generalization of this lemma to finitely-connected setting below in Lemma \ref{donotenter}; the same proof could be used in the simply-connected setting.
		
		\begin{lemma}\label{donotentersimplyconnected}
			Let $\eta$ be a generalized level line of a GFF $\Phi + u$ in $\H$ as above and $A$ a BTLS of $\Phi$ conditionally independent of $\eta$. Take $z\in \H$ and define $O(z)$ the connected component of $\H\backslash A$ containing $z$. On the event where on any connected component of $\partial O(z)$ the boundary values of $(h_A+u)\mid _{O(z)}$ are either everywhere $\geq \lambda$ or everywhere $\leq -\lambda$ , we have that a.s. $\eta([0,\infty])\cap O(z)=\emptyset$.
		\end{lemma}

		In Section \ref{sectlfc}, we extend all these results to finitely-connected domains, in particular, we extend the definition of generalized level lines by showing that they remain continuous until the stopping time $\tau$. That is to say, that level lines remain continuous up to its accumulation point, even if it is on other boundary component. To do this, we will however first have to gain a better understanding of certain type of BTLS in simply-connected domains, called two-valued local sets.

		\subsection{Two-valued local sets in simply connected domains}

		\begin{figure}[ht!]	   
			\centering
			\includegraphics[width=5.in]{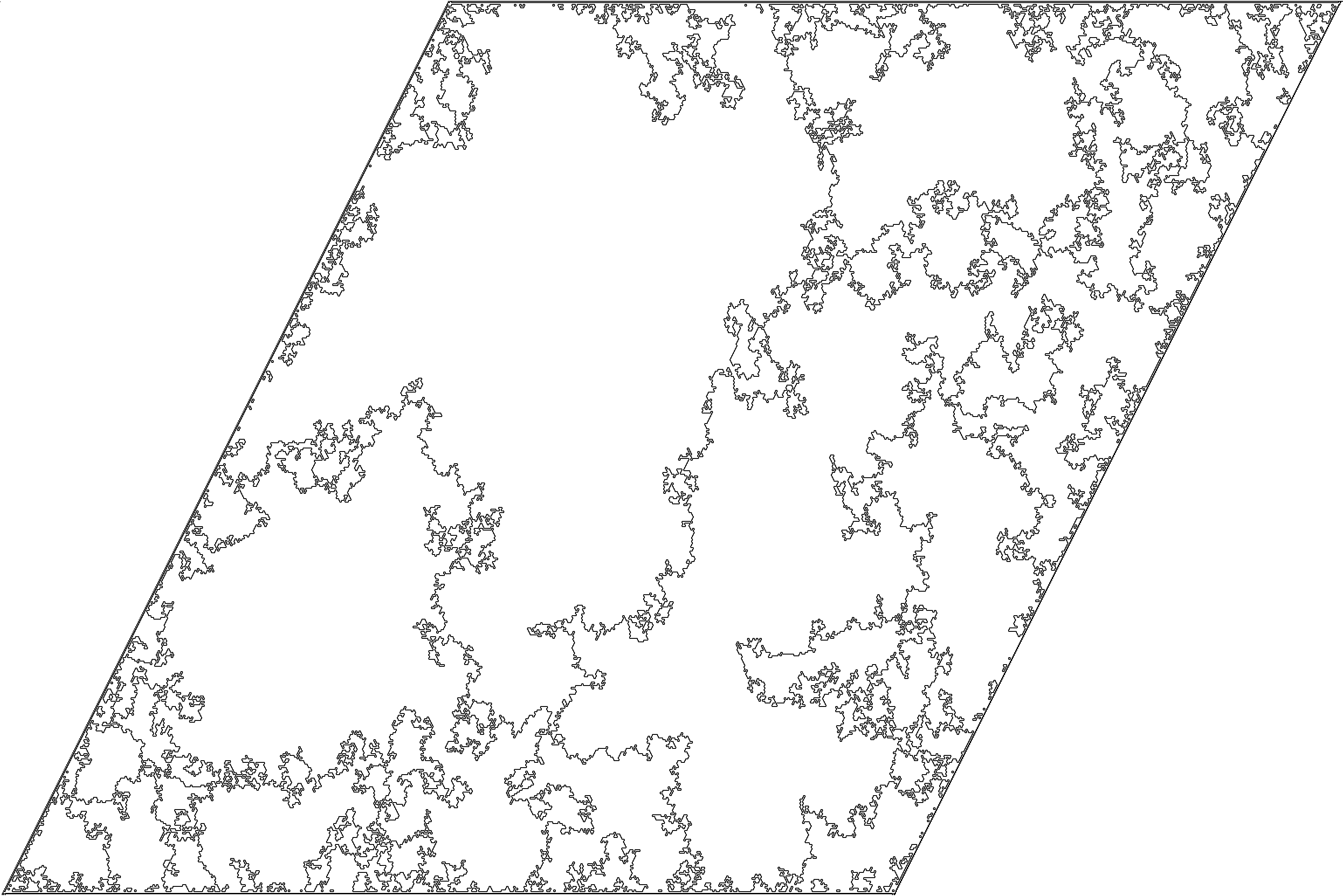}
			\caption{Simulation of $\A_{-\lambda,\lambda}$ done by B. Werness.}\label{SimALE}
		\end{figure}
		
		Another family of useful BTLS is that of two-valued local sets. In \cite{ASW}, two-valued local sets of the zero boundary GFF were introduced in the simply connected case, which we assume to be $\D$ for convenience. Two-valued local sets are thin local sets $A$ such that the harmonic function $h_A$ takes precisely two values.	More precisely, take $a,b>0$, and consider BTLS $\A_{-a,b}$ coupled with the GFF such that $h_{\A_{-a,b}}$ is constant in each connected component of $\D\backslash \A_{-a,b}$ and for all $z\in \D\backslash \A_{-a,b}$, $h_{\A_{-a,b}}\in \{-a,b\}$. It is somewhat more convenient to assume that the two-valued local sets and first passage sets introduced later also by convention contain the boundary.
		
		\cite{ASW} dealt with the construction, measurability, uniqueness and monotonicity of two-valued local sets in the case of the zero boundary GFF in simply connected domains. Here we state a slight generalization of this main theorem for more general boundary values.
		
		In this respect, let $u$ be a bounded harmonic function with piecewise constant boundary values. Take $a,b>0$ and define $u^{-a,b}$ to be the part of the boundary where the values of $u$ are outside of $[-a,b]$. As long as $u^{-a,b}$ is empty, the harmonic function $h_A$ still takes only two values $-a$ and $b$. Otherwise, we also allow for components where some of the boundary data for $h_A$ (corresponding to $u^{-a,b}$) is not equal to $-a$ or $b$. More precisely, we define the two-valued local sets $\A_{-a,b}^u$ for a GFF $\Phi + u$ as follows:
			
		\begin{defn}[Two-valued local sets] Let $u$ be a bounded harmonic function with piecewise constant boundary values. We call a BTLS $\A^{u}_{-a,b}$ coupled with $\Phi + u$ a two-valued set, if the complement of $\A^u_{-a,b}$ has exactly two types of components $O$: 
		\begin{enumerate}
			\item Those where $\partial O \cap \partial \D$ is a totally disconnected set. In these components $h_{\A_{-a,b}^u}+u$ takes the constant value: $-a$ or $b$.
			\item Those where $\partial O \cap \partial \D \subset u^{-a,b}$. In these components $h_{\A_{-a,b}^u} + u$ takes boundary values $u$ on the part $\partial O \cap \partial u^{-a,b}$ and has either constant boundary value $-a$ or $b$ on the rest of $\partial O$, in such a way that $h_{\A_{-a,b}^u}+u$ is a bounded harmonic function that is either greater or equal to $b$ or smaller or equal to $-a$ throughout the whole component. 
		\end{enumerate}
		\end{defn}
		
		The next proposition basically says that all the properties of the zero-boundary case generalize to the general boundary.
		
		\begin {prop}
		\label {cledesc3}
		Consider a bounded harmonic function $u$ as above. If $|a+b| \geq 2\lambda$ and\\ $[\min(u),\max(u)] \cap (-a,b) \neq \emptyset$, then  it is possible to construct $\A_{-a,b}^u\neq \emptyset$ coupled with a GFF $\Phi$ . Moreover, the sets $\A_{-a,b}^u$ are
		\begin{itemize} 
			\item unique in the sense that if $A'$ is another BTLS coupled with the same $\Phi$,  such that a.s. it satisfies the conditions above,	then $A' = \A^u_{-a, b}$ almost surely;  
			\item measurable functions of the GFF $\Phi$ that they are coupled with;
			\item 	monotone in the following sense: if $[-a,b] \subset [-a', b']$ with $b+a \ge 2\lambda$,  then almost surely, $\A^u_{-a,b} \subset \A^{u}_{-a', b'}$. 
			
		\end{itemize}
		\end {prop}
		The proof is an extension of the proof of Proposition 2 and the arguments in Sections 6.1 and 6.2 of \cite{ASW}.
		\begin{proof}
			\textsc{Construction:}  We know from \cite{ASW} that the condition $|a+b| \geq 2\lambda$ is necessary. Also, if $[\min(u),\max(u)] \cap (-a,b) \neq \emptyset$ does not hold, then the empty set satisfies our conditions. Thus, suppose that $|a+b| \geq 2\lambda$ and $[\min(u),\max(u)] \cap (-a,b) \neq \emptyset$. We start by constructing the basic sets with $|b+a| = 2\lambda$. Observe that in this basic case one can only concentrate on $\A_{-\lambda,\lambda}^u$ as for any other $(a,b)$ with $b=-a+2\lambda$ it is enough to construct $\A_{-\lambda,\lambda}^{u+a-\lambda}$.
			
			So let us  build $\A_{-\lambda,\lambda}^u$. To do this, partition the boundary $\partial \D = \bigcup_{k=1}^m \B_k$ such that each $\B_k$ is a finite segment, throughout each $\B_k$ the function $u$ is either larger or equal to $\lambda$, smaller or equal to $-\lambda$, or is contained in $(-\lambda, \lambda)$, and $m$ is as small as possible. Call $m$ the \textit{boundary partition size}. Notice that $m$ is finite by our assumption. We will now show the existence by induction on $m$.
			
			In fact, the heart of the proof is the case $m = 2$, so we will start from this. If $u$ is, say, larger than $\lambda$ on $\B_1$ and smaller than $-\lambda$ on $\B_2$, then by Lemma \ref{lemext} we can draw a generalized level line from one point in $\partial \B_1$ to the other one, by Proposition \ref{propnotouchsc} it almost surely finishes at the other point of $\partial \B_1$ and decomposes the domain into components satisfying (2). 
			
			So suppose $u$ is larger than $\lambda$ on $\B_1$ but in $(-\lambda, \lambda)$ on $\B_2$. Then, we can similarly start a generalized level line from one point in $\partial \B_1$ targeted to the other one. Again, we know that it finishes there almost surely. It will decompose the domain into one piece that satisfies the condition (2) and possibly infinitely many simply connected pieces that have a boundary partition size equal to 2. We can iterate the level line in each of these components. Now for any $z \in \D \cap \Q^2$, denoting the local set process arising from the construction and continuing always in the connected component containing $z$ by $A_t^z$, we have (say from Proposition \ref{BMinterior}) that $h_{A_t}(z)$ is a martingale. We claim that from this it follows that any $z$ is in a component satisfying (1) or (2) above after drawing a finite number of generalized level lines. Indeed, fix some $z \in \D \cap \Q^2$; then any level line iterated in a component containing $z$ that stays on the same side of $z$ than the previous level line will have a larger harmonic measure than the previous one; as the sign of the level line facing $z$ changes, we see that $h_{A_t}(z)$ changes by a bounded amount. This can happen only a finite number of times and thus the claim. Hence we have shown the construction in the case $m=2$.
			
			Now, if $m = 1$, then the only possible case is that $u$ takes values in $(-\lambda, \lambda)$. In this case the generalized level lines can be started and ended at all points of the boundary. In choosing any two different points on the boundary and drawing a level line, we will decompose $D$ into simply-connected components such that their boundary partition size equal to 2 (See Figure \ref{ALE const}). 
			
			\begin{figure}[ht!]	   
				\centering
				\includegraphics[width=2in]{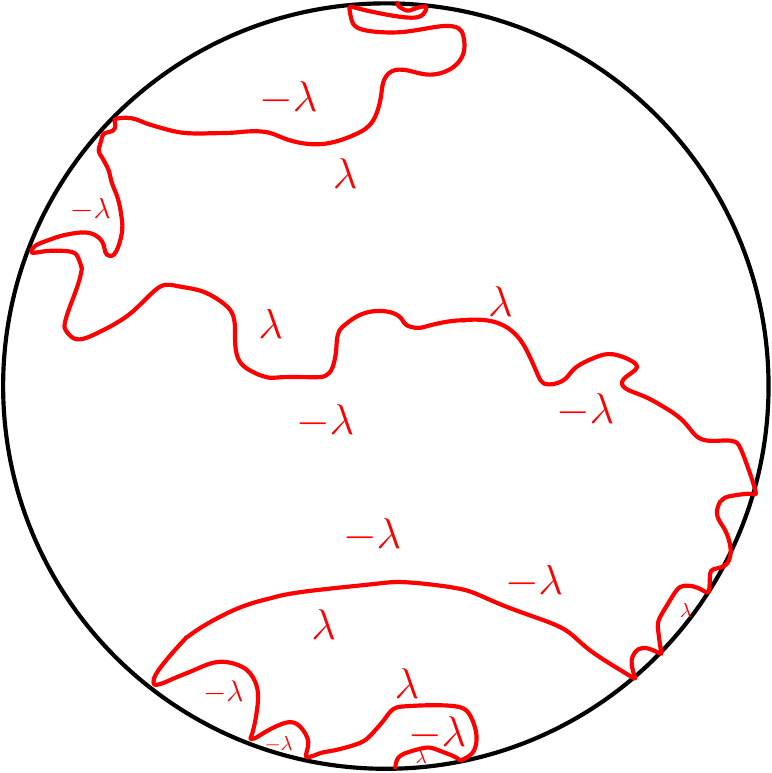}
				\includegraphics[width=2in]{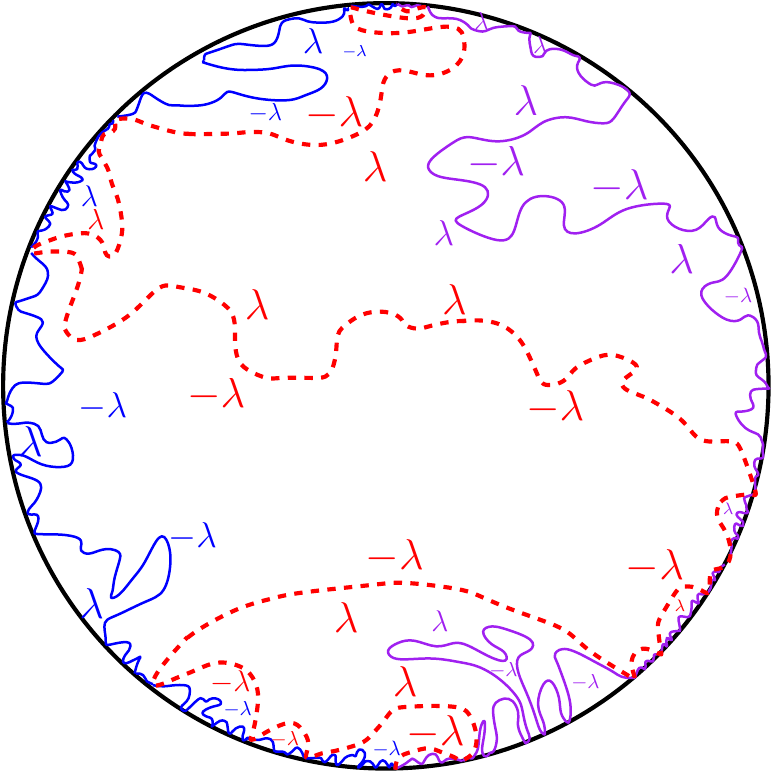}
				\caption {An illustration of the first two steps in the construction of $\A^u_{-\lambda,\lambda}$. }
				\label{ALE const}
			\end{figure}
			
			For $m \geq 3$, we must have at least two $\B_k$, say $\B_1$ and $\B_2$ (not necessarily adjacent) such that $|u| \geq \lambda$ on them. We then start our generalized level line from a possible starting point in 
			$\partial \B_1$ towards a possible target point in $\partial \B_2$. By Lemmas \ref{lemext} and Proposition \ref{propnotouchsc} it stops at either at its target point or at a point between two $\B_i$, $\B_j$ such that on both  $|u| \geq \lambda$. One can verify that in each of these cases, in each component cut out the boundary partition size is strictly smaller than $m$.
			
			To construct $\A_{-a,b}^u$ in the general case, one can follow the proof in Section 6 in \cite{ASW} or in Section 3.2.2 of \cite{AS2}. For the sake of completeness, let us provide the only slightly modified argument here:
			
			We first the construct  $\A^u_{-a,b}$ for some ranges of values of $a$ and $b$, and then describe the general case. 
			
				\begin{itemize}
	\item\underline{\textit{$a=n_1\lambda$ and $b=n_2\lambda$, where $n_1$ and $n_2$ are positive integers:}}   Define $A^1=\A^u_{-\lambda,\lambda}$, and define $A^{n+1}$ iteratively in the following way: consider a connected component $O$ of $D\backslash A^n$ not yet satisfying the conditions of the definition of the TVS. Suppose that $h_{A^n} + u $ is equal to $\widetilde n \lambda$ on $\partial O \cap A^n$. Then we explore $\A^{h_{A^n}+u-\widetilde n\lambda}_{-\lambda,\lambda}$ of $\Phi^{A^n}$ restricted to $O$. Define $A^{n+1}$ as the closed union between $A^n$ and the explored sets. Then we set: $\A^u_{-a,b}=\overline{ \bigcup A^n}$.
	
	\item \underline{\textit{$a+b = n\lambda$ where $n \geq 3$ is an integer}}: Pick $r\in [0, 2\lambda)$ such that there exists two integers $n_1\geq 0, n_2\geq 2$ with $a=r+n_1\lambda$ and $b =-r+n_2\lambda$. Let us start with $A:=\A_{-r,-r+2\lambda}^u$. Inside each connected component $O$ of $D\backslash A$ such that on $\partial O \cap A$ the harmonic function $h_A + u$ equals $-r$, resp. $-r+2\lambda$, explore $A^{h_A + u+r}_{-n_1 \lambda,  n_2 \lambda}$ of $\Phi^{A}$ restricted to $O$, resp. $\A^{h_A + u + r - 2\lambda}_{-(n_1+2)\lambda,  (n_2 -2) \lambda}$ of $\Phi^{A}$ restricted to $O$. We have that $\A_{-a,b}$ is the closed union of $A$ with the explored sets.
	\item\underline{\textit{General case with $b+a > 2\lambda$}}: As $\A^u_{-a,b}(\Phi)=\A^{-u}_{-b,a}(-\Phi)$, we may assume that $b>\lambda$.  Let $m\in \N$ such that $m \lambda-a \in (b-\lambda,b]$ and note that $m\geq 2$. Define $A^1:= \A^u_{-a,m\lambda-a}$, and iteratively construct $A^n$ in the following way:
	\begin{itemize}
		\item If $n$ is odd, then each component $O$ of $D\backslash A^n$ is either ready, or $h_{A^n} + u$ equals $m\lambda-a$ on $\partial O \cap A^n$. In the latter components, we explore $\A^{h_{A^n}+u-m\lambda +a}_{b+a-2m\lambda,b+a-m\lambda}$ of $\Phi^{A^n}$  restricted to $O$. Define $A^{n+1}$ the closed union of $A^{n}$ with the explored sets. Then all loops of $A^{n+1}$ are either finished, or $h_{A^n} + u$ equals $b-m\lambda\in [-a,-a+\lambda)$ on $\partial O \cap A^{n+1}$.
		\item If $n$ is even, then each component $O$ of $D\backslash A^n$ is either ready, or $h_{A^n} + u$ equals $b-m\lambda$ on $\partial O \cap A^n$. In the latter components, we explore $\A^{h_A +u +m\lambda -b}_{-a-b+m\lambda,-a-b+2m\lambda}$ of $\Phi^{A^n}$ restricted to $O$. Define $A^{n+1}$ the closed union of $A^{n}$ with the newly explored sets. All loops of $A^{n+1}$ are either finished, or $h_{A^n} + u$ equals $m\lambda-a$ on $\partial O \cap A^{n+1}$.
	\end{itemize}
	Then $\A^u_{-a,b}:=\overline{ \bigcup A^n}$.
\end{itemize}
			
		\end{proof}
		
		Let us now make the following remarks:
		\begin{enumerate}[(i)]
			\item In the construction we only need to use level lines of $\Phi+u$ whose boundary values are in $[-a,b]$.
			\item For a fixed point $z\in \D$ a.s. we only need a finite number of level lines to construct the connected component of $\D\backslash \A^u_{-a,b}$ containing $z$ - this is just because for each point $z$, the harmonic function $h_{A^n}$ of the construction is a discrete-valued, discrete-time martingale stopped when it first exits an interval.
			
			\item As none of the level lines is started inside $u^{-a,b}$ nor can touch $u^{-a,b}$, any connected component of $u^{-a,b}$ belongs entirely to the boundary of a single connected component of $\D \backslash A_{-a,b}^u$. In particular, each connected component $O$ of $\D\backslash\A_{-a,b}^u$ with $\partial O \cap \partial \D \neq \emptyset$ has only finitely many intersection points and by Lemma 10 in \cite{ASW} we can assign them any values, in particular those that $u$ already takes on $\partial \D$.
		\end{enumerate}

		\begin{proof}\textsc{Uniqueness, measurability and monotonicity:} Measurability follows from the construction, as each level line is measurable w.r.t the GFF. To prove uniqueness and monotonicity, we can follow the proof of Proposition 2 in \cite{ASW}. 
			
		Indeed, to prove uniqueness suppose that there is another BTLS $A'$ satisfying the conditions of a TVS. Now sample the $\A^u_{-a,b}$ constructed above conditionally independently of $A'$, given $\Phi$. By the Remarks (i), (ii), (iii) and Lemma \ref{donotentersimplyconnected}, we see that none of the level lines used in the construction of $\A^u_{-a,b}$ can enter any connected component $O$ of $D \backslash  A'$. Thus we obtain $\A^u_{-a,b} \subset A'$.
		
		To show the opposite inclusion, notice that by conditions (1) and (2) of the TVS we have that in any connected component $O$ of $\D\backslash \A^u_{-a,b}$ the boundary values of $h_{\A^u_{-a,b}}+u$ are either larger or equal to $b$ or smaller or equal to $-a$. In particular, $$|h_{\A^u_{-a,b}}+u+a/2-b/2|\geq |h_{A'}+u+a/2-b/2|.$$ Thus,  given that $A'=A'\cup \A^u_{-a,b}$ we can just use Lemma 9 of \cite{ASW}, where instead of $k$ we use $u+a/2-b/2$ (the proof is exactly the same in this case). 
	
		Finally, monotonicity follows easily from uniqueness: to construct $\A^u_{-a',b'}$ with $a' \geq a$, $b' \geq b$ we can first construct $\A^u_{-a,b}$ and then further construct suitable two-valued sets inside each component of $D\backslash \A^u_{-a,b}$ to finish the construction of $\A^u_{-a',b'}$.
		\end{proof}
		
		\begin{rem}
			In fact in the monotonicity statement, one could also include the changes in the harmonic function: if $u, u'$ are two bounded harmonic function with piecewise constant boundary data, take $\delta=\sup_{\partial \D} |u-u'|$ if $[-a,b] \subset [-a'+\delta, b'-\delta]$ with $b+a \ge 2\lambda$,  
			then almost surely, $\A^u_{-a,b} \subset \A^{u'}_{-a', b'}$. The proof just follows from the construction and Lemma \ref{donotentersimplyconnected}.
		\end{rem}

		In \cite{AS2}, the authors studied further properties of the TVS of the GFF in a simply connected domain with constant boundary condition. Let us mention two results that are important for us in analyzing generalized level lines in finitely-connected domains. First, it was proved in \cite{AS2} that $\A_{-\lambda,\lambda}^u$ is a.s. equal to the union of all level lines of $\Phi + u$ (see Lemma 3.6 of \cite{AS2}). The exact same proof works in this context and implies that:
		\begin{lemma}\label{lem::unionofll}
			Let $Z$ be any countably dense set of points on $\partial D$. Then	$\A_{-\lambda,\lambda}^u$ is equal to the closure of the union of all the level lines of $\Phi+ u$ going between two different points of $Z$.
		\end{lemma}
		\begin{rem}
			Note that the level line of $\Phi +u$ going from $x$ to $y$ is equal to the level line of $-\Phi-u$ going from $y$ to $x$ (Theorem 1.1.6 of \cite{WaWu}).
		\end{rem}
		
		Second, the TVS $\A_{-a,b}$ are locally finite for $a, b \leq 2\lambda$, again for zero boundary GFF (Proposition 3.15 of \cite{AS2}). A rather direct generalization is as follows:
		
		\begin{lemma}\label{prop::ALE locally finite}
			In simply connected domains,	 $\A_{-\lambda,\lambda}^u$ is locally finite, i.e., a.s. for each $\epsilon>0$ there are only finitely many connected components of $D\backslash \A_{-\lambda,\lambda}^u$ with diameter larger than $\epsilon$. 
		\end{lemma}
		
		This result will be proved in more generality in Proposition \ref{prop:: locally finite} (relying on the results of the paper \cite{ALS2}), but we will sketch a direct argument here too:
		\begin{proof}[Proof sketch:]
			First assume that the boundary condition changes twice and that one boundary value corresponds to either $-\lambda$ or $\lambda$, and the other to a constant $v\in(-\lambda,\lambda)$. Let us argue that in this case $\A_{-\lambda,\lambda}^u$ is locally finite, by reducing it to the case of $\A_{-\lambda,\lambda}^v$ (that has the same law as $\A_{-\lambda-v,\lambda-v}$).
				Indeed, consider $\Phi$ a GFF  in $\D$  and $\eta$ a generalized level line  of $\Phi+v$ joining two different boundary points. Then, $\A_{-\lambda,\lambda}^v$ restricted to any connected component $O$ of  $D\backslash \eta^v$ is given by $\A_{-\lambda,\lambda}^u$ of $\Phi^{O}$. But by Proposition 3.15 of \cite{AS2}, we know that $\A_{-\lambda-v,\lambda-v}$ is locally finite, and moreover by following the proof of that proposition, we can deduce that if $\A_{-a,b}^u$ is locally finite inside some simply-connected set, it is locally finite in all simply-connected sets.
			
			Now, the proof follows from two observations:
			\begin{itemize}
				\item By uniform continuity, a continuous curve in $\D$ parametrized by $[0,1]$ only separates finitely many components of diameter larger than $\eps$ for any $\eps \geq 0$;
				\item From the construction of TVS with piecewise-constant boundary conditions given above, we see that after drawing a finite number of level lines (that are continuous up to their endpoint by Proposition \ref{propnotouchsc}) we can construct a local set $A \subset \A^u_{-\lambda,\lambda}$ such that the connected components of $D\backslash A$ either already correspond to those of $D\backslash \A^u_{\lambda,\lambda}$, or have boundary conditions as above: it changes twice between a part that is $\pm \lambda,$  and a part where it takes a value in $(-\lambda, \lambda)$.
				\qedhere 
			\end{itemize}
		\end{proof}

		These two lemmas allow us to explicitly describe a part of $\A_{-\lambda, \lambda}$ in an $\eps$-neighborhood of the boundary with a local set.
		
		\begin{lemma} \label{lem:: close boundary}
			Fix $\epsilon>0$ and let $Z$ be any countably dense set of points on $\partial \D$. Define $A$ as  closed union of the level lines of $\Phi+u$ and $-\Phi-u$ starting in $\D$ and stopped the first time they reach distance $\epsilon$ from $\partial \D$. Let $\widehat A$ be equal to the union of the connected components of $\A^u_{-\lambda,\lambda} \backslash ((1-\eps)\D)$ that are connected to $\partial \D$. Then almost surely, $A$ is equal to $\widehat A$. Furthermore, $A$ is a local set such that $A\cap \partial ((1-\eps)\D)$ is a finite set of points. 
			
			If we define $O$ to be the connected component of $\D\backslash A$ containing $0$, then the boundary values of $h_A+u$ are, in absolute value bigger than or equal to $\lambda$. Furthermore, they are piece-wise constant, change their value only finitely many times, and change their sign only at points situated on $\partial ((1-\eps)\D)$.
		\end{lemma}
		
		\begin{figure}[h!]
			\includegraphics[width=0.35 \textwidth]{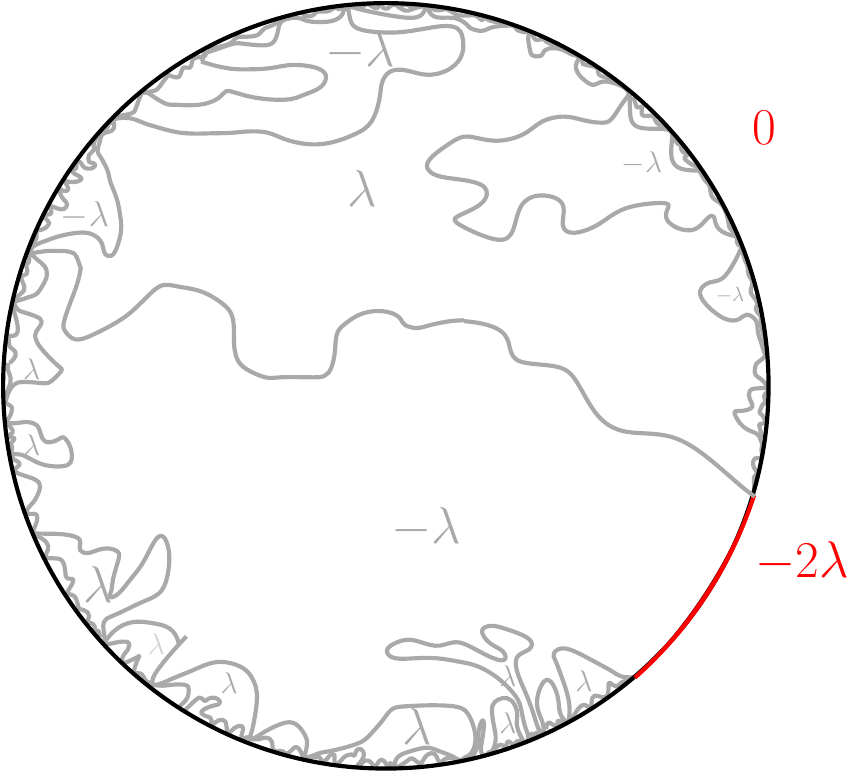}
			\includegraphics[width=0.35\textwidth]{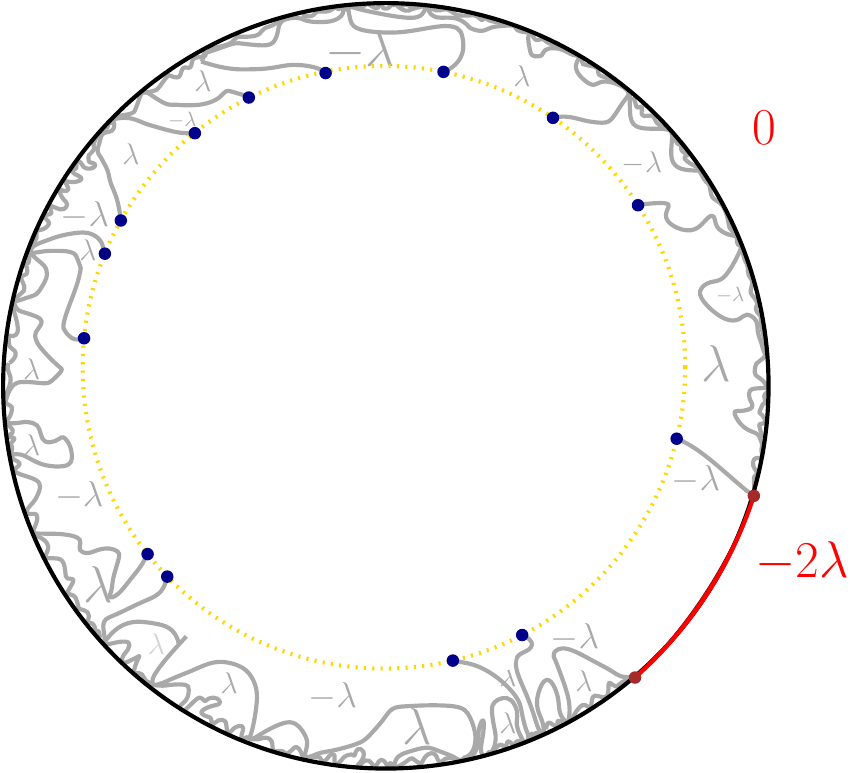}
			\caption{To the left the set $\A_{-\lambda,\lambda}^u$ where $u$ is equal to $-2\lambda$ on an arc, and equal to $0$ on the rest of the boundary. To the right the set $A$ described in Lemma \ref{lem:: close boundary}. Note that the boundary values of $h_A+u$ on $A$ are $\lambda$, $-\lambda$, equal to $u$ on the parts of the boundary where $u \notin (-\lambda,\lambda)$ and can only change its sign at the points drawn in blue, i.e. at the points where $A$ intersects $\partial (1-\eps)\D$.}		\label{Flowers}
		\end{figure}
		
		\begin{proof}
			$A$ is a local set as it can be written as the closed union of countably many local sets. Now, by Lemma \ref{lem::unionofll}, the set $\A^u_{-\lambda,\lambda}$ is equal to the union of all level lines with starting and endpoints in $Z$. By taking the intersection of this union of level lines with $\D \backslash (1-\eps)\D$, and throwing away parts that are not connected to the boundary, we obtain exactly the union of level lines stopped at distance $\eps$ from $\partial D$.
			
			Thus, to show that $A\cap \partial (1-\eps)\D$ is a finite set of points, it suffices to prove this claim for $\widehat A$, the connected component of $\A^u_{-\lambda,\lambda} \cap (\D \backslash ((1-\eps)\D))$ that is connected to $\partial \D$. Notice that this number of intersection points is bounded by twice the number of ``excursions'' of $\A^u_{-\lambda,\lambda}$ between two boundary points that intersect $(1-\eps)\D$, where by an excursion we mean a connected subset of $\A_{-\lambda,\lambda}^u\cap \D$ that intersects $\partial \D$ only at its two endpoints. However, by Lemma \ref{prop::ALE locally finite} $\A^u_{-\lambda, \lambda}$ is a.s. locally finite. This implies that there are a.s. also only finitely many excursions of $\A^u_{-\lambda,\lambda}$ intersecting $(1-\eps)\D$, as one can associate to each such excursion a unique connected component of $\D\backslash \A_{-\lambda,\lambda}^u$ (for example the component that is separated from the point $1$ by this excursion).			
		\end{proof}
		
		\subsection{Generalized level lines in finitely connected domains}\label{sectlfc}
		
		In this section, we prove that the generalized level lines are continuous up to their stopping time $\tau$, i.e. that any accumulation point described in Lemma \ref{lemext} is in fact a hitting point. WLoG, let us work in the circle domain $D=\H\backslash \bigcup_{k=1}^n C_k$, where $(C_k)_{1\leq k\leq n}$ is a finite family of disjoint closed disks. Note that, from the point of view of the GFF, $\R$  plays a similar role as the boundary of any $C_k$ because straight lines and circles are equivalent under Möbius transforms.
		
		We first need to extend Proposition \ref{propnotouchsc} to finitely-connected domains, showing that we cannot accumulate near the points that cannot be hit by a generalized level line:
		
		\begin{lemma}[Non-intersecting regime]\label{notouch}
			Let $\eta$ be a generalized level line of $\Phi +u$ in $D$ starting from $0$. If either $u \geq \lambda$ or $u \leq -\lambda$ on some open interval $J \subset \partial D$, then a.s. $\eta$ stays at a positive distance of any $x \in J$. 
			
			Moreover, if $x\in \partial D \backslash \R$ and $u \geq \lambda$ on a neighborhood in $\partial D$ clockwise of $x$, or $u \leq -\lambda$ on a neighborhood in $\partial D$ counter-clockwise of $x$, then almost surely $\eta$ stays at a positive distance of $x$. 
		\end{lemma}

	\begin{proof}
		Let us start from the first claim. To do this, consider some compact interval $J' \subset J$ and parametrize $\eta_t$ using the conformal modulus $t = M(J', \eta_t \cup \partial\H) - M(J',\partial\H)$ as in Proposition \ref{BProcess2}. 
		
		WLoG assume that the boundary values on $J'$ satisfy $u \geq \lambda$. Suppose for contradiction that $\eta$ comes arbitrarily close to $J'$. This means by Proposition \ref{BProcess2} that $W_t = \int_{J'}\partial_n h_{\eta_t}$ is a Brownian motion on $[0,\infty)$. We argue that $\int_{J'}\partial_n h_{\eta_t}$ is bounded from below. Notice that this would give the desired contradiction, as then $W_t$ would be a Brownian motion that  remains bounded from below for all times.
		
		Let $u'$ be the harmonic function that is equal to $\lambda$ in $J$ and to $u$ in $\partial D\backslash J$. Because $u'\leq u$, the harmonic function $h_{\eta_t} + u'$ has boundary values on $\eta_t$ that are smaller than or equal to $\lambda$. Let $\bar h_t$ be the bounded harmonic function with boundary values equal to $\lambda$ in $\eta_t\cup J$ and to $\sup |u| + 1 > \lambda$ in the rest of the boundary. As the minimum of $\bar h_t - (h_{\eta_t} + u')$ is attained on $J$ and equals zero over the whole of $J$, we have that the outward normal derivative of $\bar h_t - (h_{\eta_t} + u')$ on $J'$ is negative and thus $\int_{J'} \partial_n (h_{\eta_t}+u') \geq \int_{J'} \partial_n \bar h_t$. To finish the first claim, we note that $\int_{J'} \partial_n u'<\infty$, and $\int_{J'}\partial_n \bar h_t$ is non-decreasing in $t$: indeed, the harmonic function $\bar h_t \geq \lambda$ on any $\eta([t,t+\eps])$ and thus $\lambda \leq \bar h_{t+\eps} \leq \bar h_t$. As the value of both $\bar h_t, \bar h_{t+\eps}$ on $J'$ equals $\lambda$, we have that $\int_{J'}\partial_n (\bar h_t - \bar h_{t+\eps}) \leq 0$, as $\bar h_t - \bar h_{t+\eps}$ is a non-negative harmonic function which is equal to $0$ on $J'$.

		For the second claim, assume WLoG that $u \geq \lambda$ on some neighborhood in $\partial D$ clockwise of $x$. Let us denote this neighborhood by $J_L$. If $u \geq \lambda$ also on a neighborhood counter-clockwise, then we are done by the first claim. Otherwise, let $\eps>0$ be very small, in particular smaller than the distance between any boundary components. Notice that we can start a generalized level line $\hat \eta$ of $\Phi+u$ from $x$ stopped at $\hat \tau_\eps$ the first time the level line reaches distance $\epsilon$ from the connected component it started from, or has reached its continuation threshold on the same connected component of $\partial D$ that contains $x$. Note that $\hat \eta$ cannot hit $J_L$ before time $\hat \tau_\eps$ by absolute continuity w.r.t. to the simply-connected case (see Lemma 16 of \cite{ASW}). On the other hand, it a.s. hits the counter-clockwise boundary neighborhood if $u \in (-\lambda,\lambda)$ on this neighborhood, and a.s. does not hit it if $u \leq -\lambda$. 
		
		In both cases, by the first claim of this proposition, we have that $\eta$ cannot intersect nor accumulate in $J_L \cup \hat \eta_t$ before either accumulating or hitting $\hat \eta(\hat \tau_\eps)$, or finishing by accumulating somewhere on the boundary $\partial D$ at positive distance from $x$. We now argue that $\eta$ cannot accumulate at $\hat \eta(\hat \tau_\eps)$ without hitting. Indeed, because Lemma \ref{lemext} allows us to further grow $\hat \eta$ for a positive amount $s$ after time $\hat \tau_\eps$, the first claim of this proposition implies again that $\eta$ cannot accumulate near $\hat \eta(\hat \tau_\eps)$, before hitting or accumulating $\hat \eta(\hat \tau_\eps + s)$. So it remains to argue that $\eta$ cannot hit $\hat \eta(\hat \tau_\eps)$. This now follows, as hitting $\hat \eta(\hat \tau_\eps)$ would in particular mean that $\eta$ is continuous up to this hitting time $\tau_1$ and thus stay in some open neighborhood of $\hat \eta(\hat \tau_\eps)$ after some time $s < \tau_1$. Thus, by the absolute continuity of the GFF in this neighborhood, i.e., the same proof as for Lemma 16 of \cite{ASW}, one can show that $\eta$ cannot hit $\hat \eta(\hat \tau_\eps)$.
	\end{proof}

		The following proposition tells us that when  level lines approach the boundary they hit it in only one point and thus they are continuous until that time. This implies that, it is not possible for level lines to accumulate on the boundary without hitting it
		
		\begin{prop}\label{prop::target} Let $\eta$ be a generalized level line of GFF $\Phi + u$ in $D$ starting from $0$. Let $T_p$ be the set of boundary points on $\partial D \backslash \R$ that $\eta$ could potentially hit, i.e. set of points $x$ for which $u> -\lambda$ in a neighborhood in $\partial D$ counter-clockwise of $x$ and $u<\lambda$ in some neighbourhood in $\partial D$ clockwise of $x$. Then a.s. either $\eta$ hits $T_p$ or stays at a positive distance from it.
		\end{prop}
		
		It is useful to observe that the set $T_p$ is precisely the set of points for which one can start a level line of $-\Phi-u$. The following lemma says that when we start simultaneously a generalized level line of $\Phi + u$ and a level line of $-\Phi - u$ from different boundary components, then these level lines either agree on a continuous curve joining two boundary conditions or stay at a positive distance from each other.
		
		\begin{lemma}\label{lem::rev}
			Let $\eta$ be a level line of $\Phi+u$ started at $x \in \R$ and let $\eta'$ be a level line of $-\Phi-u$ started at $y \in \partial D\backslash \R$, stopped at times $\tau, \tau'$ respectively, that correspond to the first times they intersect a connected component of $\partial D$ different to their starting one (or when they hit the continuation threshold on their own component). Then, either $\eta\cap \eta'=\emptyset$ or $\eta\cap \eta'$ is a connected set that intersects $\R$ and the connected component of $\partial D$ that contains $y$. In the latter case, the level lines $\eta$ and $\eta'$ are continuous until the first time they intersect a connected component of the boundary that does not contain their respective starting point.
		\end{lemma}
		
		\begin{proof}
			Let us note that for any fixed $s\in \R$, $\eta$ stopped at the first time it hits $\eta'_s$ (or at $\tau$) is a level line of $\Phi+u+h_{\eta'_s}$ (see the right image of Figure \ref{f.rev}). Thus, by Lemma \ref{notouch} the only point it can hit or accumulate in $\eta'_s$ is in its tip: $\eta'(s)$. To argue it cannot accumulate at $\eta'(s)$ without hitting it, we can use the same argument as in the proof of the second claim in Lemma \ref{notouch}: we can continue the level line $\eta'$ for another positive time $u$ to see that $\eta(t)$ cannot accumulate at $\eta'(s)$ before accumulating or hitting at $\eta(s+u)$. Thus, $\eta$ can either hit $\eta'(s)$, or stay at a positive distance of $\eta'_s$.  
			
			Now, let us note that there is no rational $s$ such that that $\eta'(s) \notin \eta$, but $\eta'_s \cap \eta \neq \emptyset$. Indeed, if there were such a time $s$, then $\eta$ would hit a point in $\eta'_s$ different from $\eta'(s)$. It is clear that the same holds when we switch the roles of $\eta'$ and $\eta$. Thus, we see that if $\eta, \eta'$ intersect at some time-points $s, s'$ respectively, then $\eta'([s',\tau']) \subset \eta$ and $\eta([s,\tau]) \subset \eta'$. But this means in fact that in this case $\eta([\sigma,\tau]) = \eta'([\sigma',\tau'])$, where $\sigma, \sigma'$ are respectively the last times before $\tau, \tau'$ where the level lines $\eta, \eta'$ touch the component of the boundary containing their starting point. Morever, we see that if $\eta$ and $\eta'$ intersect they hit the same points but in inverse orders, and are thus both continuous up to and including the hitting time of the other boundary.
		\end{proof}
			\begin{figure}[h!]
				\includegraphics[width=0.45\textwidth]{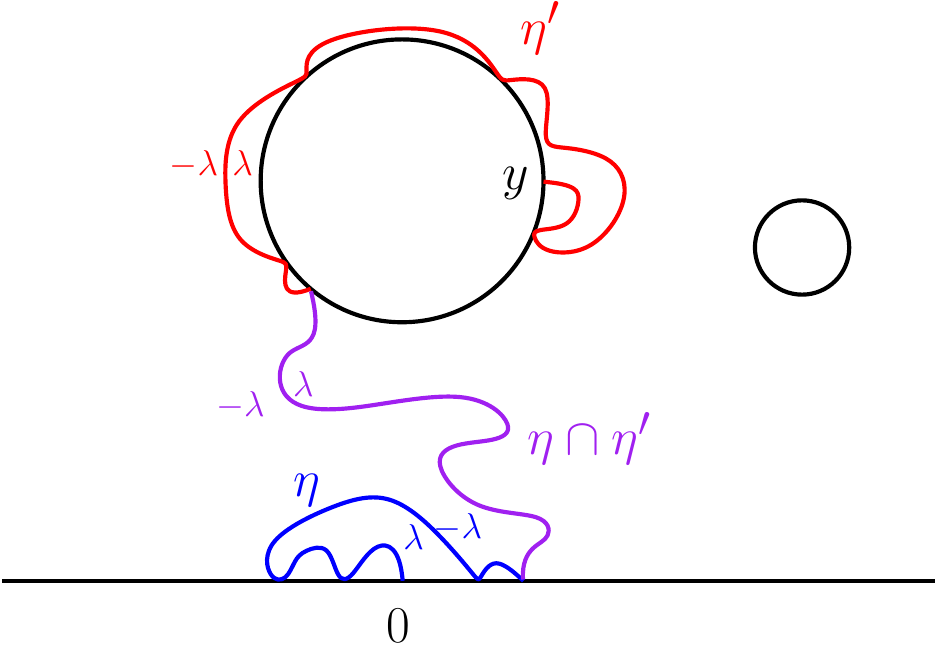}
				\includegraphics[width=0.45\textwidth]{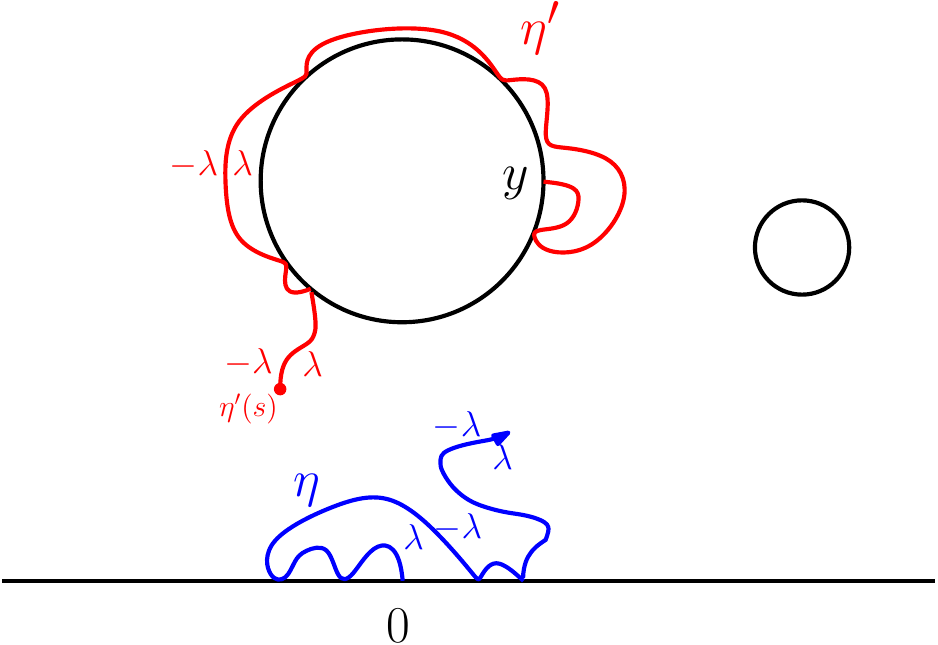}
				\caption{The left image illustrates the non-trivial option of Lemma \ref{lem::rev}. The right image gives us an idea of the proof: for any rational time $s$, we stop the construction of the red curve and show that the blue level line can only hit the red curve at its tip.}
				\label{f.rev}
			\end{figure}

		We can now prove Proposition \ref{prop::target}.
		\begin{proof}[Proof of Proposition \ref{prop::target}] Fix $\epsilon>0$ very small (say much smaller than the minimal distance between two connected components of $\C\backslash D$). By Lemma \ref{lem:: close boundary} and the absolute continuity of the GFF (Proposition 13 and Corollary 14 of \cite{ASW}), for all $C_k$ we can construct a local set $A_k$ in an $\epsilon$-neighbourhood of $C_k$ such that its boundary values are either $\geq \lambda$ or $\leq -\lambda$ in some open interval around any boundary point of $A_k$ that is of distance smaller than $\epsilon$ to the boundary. By Lemma \ref{notouch} the generalized level line $\eta$ of $\Phi+u$ started at $0$ will stay at a positive distance from all these points. Thus, it can only get infinitely close to one of $C_k$ by first hitting or accumulating at one of the points of $A_k$ that is exactly at distance $\epsilon$ from a boundary component $C_k$. Notice that by Lemma \ref{lem:: close boundary} there are only finitely many of such points. Moreover, each such point belongs to either a level line of $\Phi + u$ or a level line of $-\Phi-u$ started from the boundary component $C_k$. By Lemma \ref{notouch}, we know that the level line $\eta$ will stay at a positive distance from the first type of points. Finally, by Lemma \ref{lem::rev} we know that if it gets infinitely close to one of the other type of points, it will actually agree with this level line until hitting the boundary component containing its starting point.
		\end{proof}

		For further reference, let us resume the existence and continuity of level lines in finitely-connected domains in a proposition - it follows directly by combining Lemmas \ref{lemext} and \ref{notouch} with Proposition \ref{prop::target}.
		
		\begin{prop}[Existence and continuity of generalized level lines]\label{propext}
			Let $u$ be a bounded harmonic function with piecewise constant boundary data such that $u(0^-) < \lambda$ and $u(0^+)>-\lambda$. Then, there exists a unique law on random simple curves $(\eta(t), t\geq 0)$ coupled with the GFF such that $(**)$ holds for the function $u$ and possibly infinite stopping time $\tau$ that is defined as the first time when $\eta$ hits a point $x \in \partial D \backslash \R$ or hits a point $x\in \R$ such that $x \geq 0$ and $u(x^+) \leq -\lambda$ or $x \leq 0$ and $u(x^-) \geq \lambda$. We call $\eta$ the generalized level line for the GFF $\Phi + u$, and it is measurable w.r.t $\Phi$. 
			
			Moreover, $\eta_t$ is continuous on $[0,\tau]$, and it can only hit points from which we can start a generalized level line of $-\Phi - u$ (note that on $\partial D\backslash \R$ these are the points $x$ such that $u < \lambda$ on some interval of $\partial D$ clockwise of $x$ and $u > -\lambda$ on some interval of $\partial D$ counterclockwise of $x$).
		\end{prop}

		As a consequence, we can prove the generalization of Lemma \ref{donotentersimplyconnected} for finitely-connected domains, the only additional input being the continuity up to its stopping time, and the precise description of the hitting points stated in the proposition above.
		
		\begin{lemma}\label{donotenter}
			Let $\eta$ be a generalized level line of a GFF $\Phi + u$ in $D$ as above and $A$ a BTLS of $\Phi$ conditionally independent of $\eta$. Take $z\in D$ and define $O(z)$ the connected component of $D\backslash A$ containing $z$. On the event where on every connected component of $\partial O(z)$ the boundary values of $(h_A+u)\mid _{O(z)}$ are either everywhere $\geq \lambda$ or everywhere $\leq -\lambda$ , we have that a.s. $\eta([0,\infty])\cap O(z)=\emptyset$.
		\end{lemma}
		
		Notice that we allow for the situation where on some boundary components the value is $\geq \lambda$ and on some it is $\leq -\lambda$. One of the key ingredients is Lemma \ref{BPLS} (2), that reads as follows: if $A, B$ are conditionally independent local sets, then the boundary values on $A \cup B$ do not change at any point that is of positive distance of the boundary of $\partial A \cap \partial B$.
		
		\begin{proof}
			Define as $E^z$ the event where on any connected component of $\partial O(z)$ the boundary values of $(h_A+u)\mid _{O(z)}$ are either everywhere $\geq \lambda$ or everywhere $\leq -\lambda$. Suppose for contradiction that on the event $E^z$, $\eta([0,\infty]) \cap (O(z))\neq \emptyset$ with positive probability. 
			
			Take $\epsilon >0$ and $\tau:=\tau(\epsilon)$ the first time such that $\eta(\tau)\in O(z)$ and is at distance $\epsilon>0$ of $\partial O(z)$. Note that under our assumption for small enough $\epsilon$, the event $\{\tau<\infty\}\cap E^z$ has non-zero probability. One can verify that $\bar \eta(t):= \eta(t+\tau)$ is a generalized level line of $\Phi^{A\cup \eta_\tau}+( \Phi_{A\cup \eta_\tau}+u)$. As the generalized level-line is a simple continuous curve, it stays at a positive distance of $\partial O(z) \cap \eta_\tau$. Additionally, from Lemma \ref{BPLS} (2) it follows that the boundary values  of $( \Phi_{A\cup \eta_{\tau}}+u)$  are $\geq \lambda$ or $\leq -\lambda$ around any point on $\partial O(z)$ that is at positive distance from $\eta_\tau$. Thus, by  Lemma  16 of \cite{ASW}, $\bar \eta$ cannot exit $O(z)$ through any point that is at positive distance of $\eta_\tau$. But, by Proposition \ref{propext}, we have that $\bar \eta(\infty)$ ends at a point on $\partial D$ that is different from any of its previously visited points, staying continuous up to and at the moment it hits the boundary, giving a contradiction.
		\end{proof}
		
		A particularly useful corollary is the following.
		
		\begin{cor}
			Let $A$ be a local set satisfying condition (3) of Definition \ref{BTLSCND} and $\eta$ a generalized level line, both coupled with the same GFF. Then, $\eta$ cannot enter a connected component of $D\backslash A$ inside which $h_A$ is smaller than the boundary values of $h_\eta$.
		\end{cor}

		\subsection{Two-valued local sets in n-connected domains}\label{sec:: cons TVS}
		
		In order to study the FPS in n-connected domains, it is useful to first extend the definition of two-valued sets to this setting. In this respect, let $u$ be a bounded harmonic function with piece-wise constant boundary values. Recall that by $u^{-a,b}$ we denoted the part of $\partial D$ where the values of $u$ are outside of $[-a,b]$.The two-valued set $\A^{u}_{-a,b}$ in n-connected domains is then a BTLS such that in each connected component $O$ of $D\backslash \A_{-a,b}^u$ the bounded harmonic function $h_{\A_{-a,b}^u}$ satisfies the following \hypertarget{tvs} {conditions:} 
		\begin{enumerate}
			\item[(\twonotes)] On every boundary component of $\partial O\backslash u^{-a,b}$ the harmonic function $h_{\A_{-a,b}^u}+u$ takes constant value $a$ or $-b$, and on $\partial O \cap u^{-a,b}$ agrees with $u$. Additionally, we require that in every connected component of $\partial O$  either $h_{\A_{-a,b}^u}+u\leq -a$ or $h_{\A_{-a,b}^u}+u\geq b$ holds. 
		\end{enumerate}
		Note that in particular when $-a\leq \inf u \leq \sup u\leq b$,  the condition \hyperlink{tvs}{(\twonotes)} simplifies to $h_{\A_{-a,b}^u}+u$ takes constant values $-a$ or $b$ in every connected component of $\partial O$. 
		Now, we announce the fundamental proposition for two-valued sets in n-connected domains.
		\begin {thm}\label {two-value gen f}
		Let $D$ be an n-connected domain and $u$ be a bounded harmonic function that has piecewise constant boundary values. Consider $-a < 0 < b$, with $a+b\geq 2\lambda$. Then, $\A^u_{-a,b}$ exist and are non-empty if $[\min(u),\max(u)] \cap (-a,b) \neq \emptyset$. 
		Moreover, $\A^u_{-a,b}$ satisfy the following properties: 
		\begin{itemize} 
			\item They are unique in the sense that if $A'$ is another BTLS coupled with the same $\Phi$,  such $A'$ satisfies \hyperlink{tvs}{(\twonotes)} almost surely, $\A_{-a,b}^u=A'$.
			\item They are all measurable functions of the GFF $\Phi$ that they are coupled with.
			\item They are monotonic in the following sense: if $[a,b] \subset [a', b']$ and $-a < 0 < b$ with $b+a \ge 2\lambda$,  then almost surely, $\A^u_{-a,b} \subset \A^u_{-a', b'}$. 
			\item For each $\A^u_{-a,b}$, there are at most $n$ connected components of $D\backslash \A^u_{-a,b}$ which are not simply connected.
		\end{itemize} 
		\end {thm}
		
		\begin{figure}
			\includegraphics[width=0.5 \textwidth]{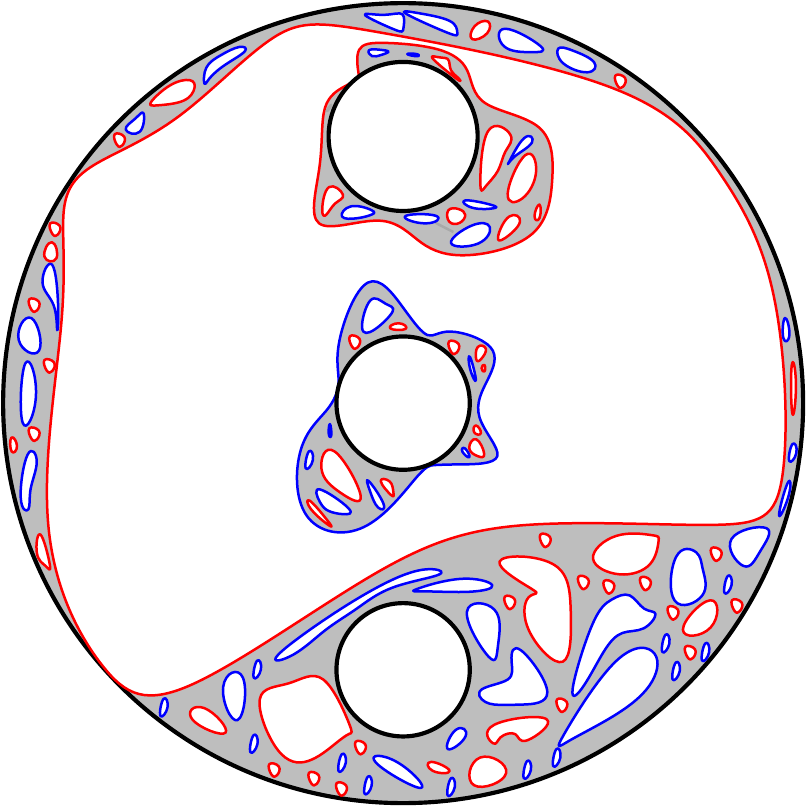}
			\caption{Representation of a TVS in a 4-connected domain. The boundaries of the domain are black, the boundaries of the TVS with boundary values $-a$ are coloured red and those with values $b$ are coloured blue. }
		\end{figure}
		We will now construct an instance of this set and prove the measurability of this construction. To prove the uniqueness, we will first in fact prove the uniqueness of the FPS in the next section; it will be a consequence of Lemma \ref{EaEbAab}. Finally, monotonicity follows from uniqueness as in the proof of Proposition \ref{cledesc3} above. Until having proved uniqueness, we mean by $\A^u_{-a,b}$ always the set constructed just below.
		
		The proof of existence is in its spirit very similar to the proof of Proposition \ref{cledesc3}. However, we do need extra arguments to treat the multiply-connected setting.
		
		\begin{proof}
			
			\textsc{Construction} Again, we can assume that we are in the non-trivial case, in other words that $[\min(u),\max(u)] \cap (-a,b) \neq \emptyset$. As in the proof of Proposition \ref{cledesc3}, it suffices to construct $\A_{-\lambda,\lambda}^u$. 
			
			This time we need a double induction. Let $N$ be the number of boundary components and as in proof of Proposition \ref{cledesc3}. We take the minimal partition of any boundary component $\B_i$ as $\B_i = \bigcup_{k=1}^{m_i} \B_i^k$,  such that each $\B_i^k$ is a finite segment, throughout each $\B_i^k$ the function $u$ is either larger or equal to $\lambda$, smaller or equal to $-\lambda$, or contained in 
			$(-\lambda, \lambda)$. Recall that we called $m_i$ boundary partition size of $\mathcal{B}_i$. We will now use induction on pairs $(N, \max_{i \leq N} m_i)$.
			
			The case $(1,m)$ is given by Proposition \ref{cledesc3}. The key case is $(N,2)$, so let us prove this by inducting on the number of boundary components $N$. 
			
			On any $\mathcal{B}_i$ satisfying $|u| \geq \lambda$, draw a generalized level line starting from one point of $\partial \B_i^1$ to the other. If it hits some other boundary component, we stop it and we have reduced the number of boundary components in each of the domains cut out and we can use induction hypothesis. Otherwise, by Proposition \ref{propext}, it ends at the other point of $\partial \B_i^1$ and reduces the boundary partition size of this boundary component to $1$. As in this case the level line is continuous up to its hitting point, we know that it stays at a positive distance of other boundary components. Hence we can suppose that the only boundary components with boundary partition size equal to 2 have one part with $u \in (-\lambda, \lambda)$. 
			
			Now, pick any such component, say, $\mathcal{B}_1$ and suppose $u$ is larger than $\lambda$ on $\B_1^1$. Then, we can start a generalized level line from one points on $\partial \B_1^1$ towards the other one. If the generalized level line hits some other component or cuts the domain into subdomains with strictly less than $N$ boundary components,  we can use induction hypothesis. Otherwise, we have finished all components $O$ such that $\partial O \cap u^{-a,b}$ is non-zero. It now remains to see that all `inner components' are also finished in finite time. This follows similarly to the proof of Proposition \ref{cledesc3} by using the fact that $h_{A_t}(z)$ is a bounded martingale and converges almost surely.
			
			Suppose now $(N,n)$ satisfy $N \geq 2$, $m \geq 3$. Then we can similarly to the proof of Proposition \ref{cledesc3} pick a generalized level line on some boundary component with boundary partition size bigger than $2$ such that by drawing it we either reduce the boundary partition size to $2$ for any subdomain with $N \geq 2$, or reduce the number of boundary components in each subdomain. Using a finite number of such lines we have reduced to either $(N,2)$, or $(N',3)$ with $N'$ smaller than 3.
			
			It remains to treat the case $(N,1)$, if all components satisfy $|u| \geq \lambda$, we are done. Otherwise, in any component with $u \in (-\lambda, \lambda)$ we can start a level line from any point for some short amount of time. This will either reduce the setting to $(N,3)$, $(N,2)$ or reduce the number of boundary components.

		\end{proof}
		Examining closely the proof the following holds:
		\begin{enumerate}[(i)]
			\item In the construction we only need to use level lines whose boundary values are in $[-a,b]$.
			\item For a fixed point $z\in D$, a.s. we only need a finite number of level lines to construct the connected component of $D\backslash \A^u_{-a,b}$ containing $z$.
			\item As none of the level lines is started inside $u^{-a,b}$ nor can touch $u^{-a,b}$, any connected component of $u^{-a,b}$ belongs entirely to the boundary of a single connected component of $D \backslash A_{-a,b}^u$. In particular each connected component $O$ of the complement of $\A_{-a,b}$ with $\partial O \cap \partial D \neq \emptyset$ has only finitely many intersection points and by Lemma 10 in \cite{ASW} we can assign them any values, in particular those that $u$ already takes on $\partial D$.
			\item Due to the construction $D\backslash \A_{-a,b}^u$ has at most $n$ non-simply connected components.
		\end{enumerate}
		
		\begin{proof}	\textsc{Measurability} of the sets $\A_{-a,b}^u$ with respect to the GFF just follows from the measurability of the level lines used in the construction.  \end{proof}

 
 \section{First passage sets of the 2D continuum GFF}\label{sec:: FPS}
 
 The aim of this section is to define the first passage sets of the 2D continuum GFF, prove its characterization and properties. We first state an axiomatic definition of the continuum FPS inspired by its heuristic interpretation:
 i.e. the FPS stopped at value $-a$ is given by all points in $\overline{D}$ that can be connected to the boundary via a path on which the values of the GFF do not decrease below $-a$. From this description, it is clear that it induces a Markovian decomposition of the GFF: the field outside of it is just a GFF with boundary condition equal to $-a$. In other words, the FPS is a local set, that we denote by $\A_{-a}$, and its harmonic function has to satisfy $h_{\A_{-a}}=-a$ as we stop at value $-a$. Finally, the question is how to translate the property for the values, as the GFF is not defined pointwise. The right way is to ask the distribution  $\Phi_{\A_{-a}}+a$ to be a positive measure.
 
 The set-up is again as follows: $D$ is a finitely-connected domain where no component is a single point and $u$ is a bounded harmonic function with piecewise constant boundary conditions. Here is the definition for general boundary values.
 
 \begin{defn}[First passage sets]\label{Def ES}
 	Let $a\in \R$ and $\Phi$ be a GFF in the multiple-connected domain $D$. We define the first passage set of  $\Phi+u$ of level $-a$ as the local set of $\Phi$ such that $\partial D \subseteq \A^u_{-a}$, with the following properties:
 	\begin{enumerate}
 		\item Inside each connected component $O$ of $D\backslash \A_{-a}^u$, the harmonic function $h_{\A_{-a}^u}+u$ is equal to $-a$ on $\partial O \backslash \partial D$ and equal to $u$ on $\partial O \cap \partial D$ in such a way that $h_{\A_{-a}^u}+u \leq -a$. 		
 		\item $\Phi_{\A^u_{-a}}-h_{\A_{-a}^u}\geq 0$, i.e., for any smooth positive test function $f$ we have 
 		$(\Phi_{\A^u_{-a}}-h_{\A_{-a}^u},f) \geq 0$.
 		\item Additionally, for any connected component $O$ of $D\backslash \A_{-a}^u$, for any $\eps > 0$ and $z \in \partial O \cap \partial D $ and for all sufficiently small open ball $U_z$ around $z$,  we have that a.s. \[h_{\A_{-a}^u}(z)+u(z) \geq \min\{ -a,\inf_{w \in U_z\cap \overline O} u(w)\} - \eps.\]  
 		\item Almost surely, $A$ contains no isolated points and each connected component of $A$ that does not intersect $\partial D$ has a neighbourhood that does not intersect any other connected component of A.
 	\end{enumerate}	
 \end{defn}
In Theorem \ref{FPSthm}, we show that there exists a unique set satisfying the conditions of Definition \ref{Def ES}.
 
 Notice that if $u\geq -a$, then the conditions (1) and (2) correspond more precisely to the heuristic and are equivalent to
 \begin{enumerate}
 	\item [(1)'] $h_{\A_{-a}^u}+u=-a$ in $D\backslash \A_{-a}^u$.
 	\item [(2)'] $\Phi_{\A_{-a}^u}+u+a\geq 0$.	
 \end{enumerate}
 
 Moreover, in this case the technical condition (3) is not necessary. This condition roughly says that nothing odd can happen at boundary values that we have not determined: those on the intersection $\partial \A_{-a}^u$ and $\partial D$. This condition enters in the case $u < -a$: 
 in \cite{ALS2} we want to take the limit of the FPS on metric graphs and it comes out that it is easier not to prescribe the value of the harmonic function at the intersection of $\partial D$ and $\partial\A_{-a}^u$.
 Notice that in contrast we did prescribe the values at intersection points for two-valued sets.
 \begin{rem}\label{bw}
 	One could similarly define excursions sets in the other direction, i.e. stopping the sets from above. We denote these sets by $\AB_b^u$. In this case the definition goes the same way except that (2) should now be replaced by $\Phi_{\AB_{b}^u}\leq h_{\AB_b^u}$ and (3) by\[h_{\AB_{b}^u}(z)+u(z) \leq \max\{ b,\sup_{w \in U_z\cap \overline O} u(w)\} + \eps.\] 
 	Let us also remark that in \cite{APS}  the sets $\AB_b$ are unluckily denoted by $\A_b$, and that they can be obtained as $\A_{-b}^{-u}$ of $-\Phi$.
 \end{rem}
 
 \begin{thm}\label{FPSthm}
 	The FPS $\A_{-a}^u$ of level $-a$ exists and is unique in the sense that if $(\Phi, A')$ is an FPS of level $-a$ for the GFF with boundary condition $u$, then $A'= \A_{-a}^u$ a.s.
 \end{thm}
 
 We start from the existence of the FPS. Here, we provide a purely continuum construction using the two-valued sets $\A_{-a,b}$. Another approach would be to consider the scaling limit of the metric graph FPS when the mesh size goes to zero, as is done in 
 \cite{ALS2}.
 
 \begin{prop} \label{Existence}
 	For $n > 2\lambda$, denote by $\A_{-a,n}^u$ the two-valued local sets coupled with the GFF $\Phi$ in the domain $D$. Then for every $a \geq 0$ the local set $\A^u_{-a}:=\overline{ \cup_{n\in \N}\A_{-a,n}^u}$ is an FPS of level $-a$. 
 \end{prop}
 \begin{proof}
 	
 	Let $\A_{-a}^u$ be as in the statement. Then $\A_{-a}^u$ is the closed union of nested measurable local sets so it is a measurable local set: it is a local set by Lemma \ref{BPLS} and measurable as a limit of measurable functions. 
 	
 	We first prove the condition (1) of the Definition \ref{Def ES}. Take a countable dense set in $D$, $(z_i)_{i\in \N}$, and note that almost surely for all $i\in \N$,  $z_i\notin \A^u_{-a}$. Consider $n > \sup u$. 	It suffices to show that for any $z_i$, there will be some finite $n$ such that the component of the complement of $\A_{-a,n}^u$ containing $z_i$ does not take the value $n$. Indeed, when this happens, then by the definition and uniqueness of two-valued sets above, it would take a value as described in (1) for all $\tilde n \geq n$. 
 	
 	Now, for $n > 2\lambda \vee \sup u$ the process $h_{\A_{-a,n}^u}(z_i)$ is a lower bounded martingale in $n$. Indeed, it is lower bounded by $-a \wedge \inf u$ and it is a martingale as we can construct the local set $\A_{-a,n+1}^u$ by first constructing $\A_{-a,n}^u$ and then constructing $\A_{-a-n,1}$ inside each connected component of $D \backslash \A_{-a,n}^u$ where the boundary value is equal to $n$. In particular it converges almost surely. It can, however, only converge when for some $n$ it belongs to the component of the complement of $\A_{-a,n}^u$ not taking the value $n$. Hence we deduce the condition (1).
 	
 	The condition (3) just comes from the fact that the value at the intersection points is prescribed by the definition of two-valued sets and it satisfies the appropriate condition.
 	
 	It remains to prove (2), i.e. that $\Phi_{\A_{-a}^u} -h_{\A_{-a}^u}\geq 0$. Note that if $n > 2\lambda \vee \sup u$, then by definition of $\A_{-a,n}^u$, the components $O$ of $D \backslash \A_{-a,n}^u$ containing a part of $\partial D$ on their boundary, take boundary value $-a$ on the part $\partial O \cap \A_{-a,n}^u$. In particular, they will not change as $n$ increases further and for all $z \in O$ we have that $\Phi_{\A_{-a,n}^u}=h_{\A_{-a}^u}$. In any other connected component $\widetilde O$ of $D \backslash \A_{-a}^u$, we have that if $z \in \widetilde O$, then either $z \in \A_{-a,n}^u$, or $\Phi_{\A_{-a,n}^u}(z)$ is equal to either $-a$ or $n$ (and thus in both cases greater or equal than $h_{\A_{-a}^u(z)} = -a$). As $\A_{-a,n}^u$ is thin, it thus follows that for all positive $f\in \CC_0^\infty$, we have that $(\Phi_{\A_{-a,n}^u}-h_{\A_{-a}^u},f)\geq 0$. We conclude using Lemma \ref{BPLS} (iii). 
 	
 \end{proof}
 
 Let us make the following observation about the construction above:
 
 \begin{enumerate}[(i)]
 	\item In the construction, we only need to use generalized level lines whose boundary values are in $[-a,\infty)$, moreover these generalized level lines never hit themselves.
 	\item For a fixed point $z\in D$, it will belong to a component of the complement of $\A^u_{-a,n}$ with value $n$ only for a finite number of $n$. Thus, we need only a finite number of level lines to construct the loop of $\A^u_{-a}$ surrounding $z$. 
 \end{enumerate}

 We now want to use these remarks and the techniques of \cite{ASW} to prove the uniqueness of the FPS:

 \begin{prop} [Uniqueness and Monotonicity of the FPS]\label{Uniqueness}
 	Let $(\Phi, A')$ be a FPS of level $-a$ for the GFF with boundary condition $u$. Then $A'= \A_{-a}^u$ a.s. Additionally, if $a\leq a'$ and $u\leq u'$, then $\A_{-a}^{u}\subseteq\A_{-a'}^{u'}$
 \end{prop}
 \begin{proof}
 	First let us prove that if $A$ is a local set such that almost surely $\Phi_{A}\geq 0$, then a.s. $A$ is a polar set. Given this condition, we have that $( \Phi_{A},1)\geq 0$, and due to the Markov property we know that $\E[(\Phi_{A},1)]=0$, thus a.s. $(\Phi_{A},1)=0$. Additionally, we know that $G_D\geq G_{D\backslash A}$, using again the strong Markov property we get that 
 	\begin{align*}
 	\iint_{D\times D}G_D(z,w) dz dw=\E\left[( \Phi,1)^2 \right] =\E\left[(\Phi^{ A},1)^{2}\right]=\iint_{D\backslash A\times D\backslash A}G_{D\backslash A}(z,w) dz dw.
 	\end{align*}
 	Hence almost everywhere $G_{D\backslash A}=G_{D}$, and thus $A$ is a polar set. 
 	
 	We now prove the uniqueness of the FPS. Assume first that $\A_{-a}^u \subseteq A'$. We claim that then $A'\backslash \A_{-a}^u$ is a polar set. Indeed, consider $B:=A'\backslash \A_{-a}^u$. From Lemma \ref{BPLS} (2), $B$ is a local set of the zero boundary GFF $\Phi^{\A_{-a}^u}$. Moreover, one can check that from our conditions on the FPS, it follows that $h_{\A_{-a}^u}+u\leq h_{A'}+u$ and hence $(\Phi^{\A_{-a}^u})_{B}\geq 0$. Thus, by the previous argument $B$ is polar. 
 	
 	Now, it suffices to prove that a.s. $\A_{-a}^u \subseteq A'$. Take $A'$ an FPS for $u' \geq u$ and $a' \geq a$. Suppose by contradiction, $\A_{-a}^u$ is not contained in $A'$. Then choosing a countable dense set in $D$, $(z_i)_{i\in \N}$, there must be some $z_i$ such that with positive probability during the construction of $ \A_{-a}^u$ a generalized level line enters the component $O(z_i)$ of the complement of $A'$ containing $z_i$. Thus, there should be some finite $n\in \N$  such that with positive probability, $\eta$, the $n^{\text {th}}$-level line pointed towards $z_i$, is the first one to enter $O(z_i)$ and $\eta \cap O(z_i) \neq \emptyset$. This is, however, in contradiction with Lemma \ref{donotenter} and the remark just after the proof: indeed, the boundary values of $h_{A'}$ inside $O(z)$ are equal to $-a'-u'\leq -a -u$ and by the observation (i) above the boundary values of $\eta$ are in $[-a,\infty)$. Thus, the uniqueness follows.
 	
 	Finally, monotonicity just follows from the construction given in Lemma \ref{Existence}.
 \end{proof}

 \begin{prop}\label{rmkmes}
 	$\Phi_{\A_{-a}^u} -h_{\A_{-a}^u}$ is almost surely a non-trivial positive measure, unless $u\leq -a$ on $D$.
  \end{prop}

\begin{proof}
	We know that it is a positive distribution, and thus by Theorem V in Section 1.4 of \cite{schwartz} in it is a positive measure. It remains to argue that this measure is almost surely non-trivial.
	
	Let us give two different arguments to see this, the first one for simplicity only for simply-connected domains and the second one for the general case, and proving a local statement:
	\begin{enumerate}
		\item Suppose that we work in $\D$ and with a zero boundary GFF. In this case, it is known that the circle average $(\rho_{1-r}^0,\Phi)$ around $0$ of radius $1-r$ converges to $0$ as $r \to 0$. But the circle average w.r.t. to $h_{\A_{-a}}$ is constantly equal to $-a$, and the variance of $(\rho_{1-r}^0, \Phi^{\A_{-a}^u})$ also converges to $0$.
		\item  A different way of seeing this is the following. 
		Since $$\E[\Phi_{\A_{-a}^u} -h_{\A_{-a}^u}]=\E[-h_{\A_{-a}^u}]>0,$$
		$\Phi_{\A_{-a}^u} -h_{\A_{-a}^u}$ is a non-trivial positive measure with positive probability. Then, in order to sample 
		$\A_{-a}^u$, we can first explore $\A_{-a/2}^u$ and then further explore 
		$\A_{-a/2}$ conditionally independently in each connected component of the complement of $\A_{-a/2}^u$. As there are infinitely many such components, we know that almost surely $\Phi_{\A_{-a}^u} -h_{\A_{-a}^u}$ is positive inside one of these components.
	\end{enumerate}

\end{proof}
  
\begin{rem}
 	Using the second argument, and the fact that the FPS has infinitely many small holes near the bounary, one can in fact argue that the FPS is non-trivial on every open dyadic square that it intersects. 
 \end{rem}
 \subsection{Extremal distance to interior points and boundary for $\A_{-a}^u$}

 We will now give an exact description of the law of the extremal distance of $\A_{-a}^u$ to interior points and to boundary components. This can be seen as a continuum analogue of Corollary 1 of \cite{LupuWerner2016Levy}. The proofs follow from Propositions \ref{BMinterior} and \ref{BProcess}, that give a way to parametrize the FPS using the distance to an interior point or a boundary component respectively. 
 
 \begin{prop} \label{LawELFPS}
 	Let $a>0$ and $D$ a n-connected domain. Moreover, let $u$ be a bounded harmonic function with piecewise constant boundary data and $u\geq -a$. Take $z\in D$ and $W_t$ a Brownian motion started from $u(z)$. Then $$G_D(z,z) - G_{D \backslash \A_{-a}^u}(z,z)$$ is distributed like the hitting time of the level $-a$ by $W_t$.
 \end{prop} 
 
 \begin{proof}
 	This follows exactly as the proof of Proposition 20 in \cite{ASW}. 
 \end{proof}
 
 Similarly, we can calculate the extremal distance between boundary components,
 analogously to Proposition 5 in \cite{LupuWerner2016Levy}.
 
 \begin{prop} \label{LawELBddFPS}
 	Let $a$ be a positive number, $D$ a finitely-connected domain, and $\B$ a union of connected components of $\partial D$. Moreover, let $u$ be a bounded harmonic function with piecewise constant boundary data changing finitely many times such that $u$ on $\partial D \backslash \B$ is a constant equal to $u_e \leq -a$. Let $\widehat{W}_t$ be a Brownian bridge with starting point:
 	$$u_s:=\ED(\B,\partial D\backslash \B)\iint_{\B \times\partial D\backslash\B} u(x)H_D(dx,dy)$$ endpoint $u_e$, and with length $\ED(\B,\partial D \backslash \B)$. 
 	If $u \geq -a$ on $\B$, then $$\ED(\B, \partial D \backslash \B) - \ED(\B \cup (\A_{-a}^u\backslash \partial D), \partial D \backslash \B)$$ has the law of the first hitting time of $-a$ by 
 	$\widehat{W}_t$. 
 	
 \end{prop}
 
 Before proving the proposition, we prove a deterministic lemma.
 \begin{lemma} Let $u$ be a bounded harmonic function in a domain $D$, that is constant equal to $u_e$ on $\partial D\backslash \B$. Then,
 	\begin{align}\label{e.Poisson Kernel}
 	\ED(\B,\partial D\backslash \B)\iint_{\B \times\partial D\backslash\B} u(x)H_D(dx,dy) = -\ED(\B,\partial D\backslash \B)\int_{\partial D\backslash\B}\partial_n u + u_e.\end{align}
 \end{lemma}
 \begin{proof}
 	By conformal invariance of the quantities we can always assume we are in a circle domain. Consider a function $\bar u$ that is equal to $u$ on $\B$ and $0$ on $\partial D \backslash \B$. 
 		As the boundary of $D$ is locally analytic, by definition the boundary Poisson kernel $H_D(dx,dy)$ for $y \in \partial D \backslash \B$ can be written as $\lim_{\eps \to 0}\eps^{-1}P_D^{y-n_y\eps}(dx)$, where $n_y$ is the outward unit normal vector at $y$. But for any $y \in \partial D \backslash \B$ we have that 
 		\[\int_\B u(x)P_D^{y - n_y \eps}(dx) = \bar u (y-n_y\eps).\]
 		As by definition $\bar u = 0$ on $\partial D \backslash \B$, we obtain that 
 		$$\iint_{(\partial D\backslash\B)\times \B} u(x)H_D(dx,dy) = -\int_{\partial D\backslash\B}\partial_n \bar u.$$
 		But now we can write $\bar u = u + \widetilde u$, where $\widetilde u$ is equal to the constant $u_e$ on $\partial D\backslash \B$ and to $0$ on $\partial B$. The last statement of Theorem \ref{thmEL} implies that $$-\int_{\partial D\backslash\B}\partial_n \tilde u = u_e \ED(\B,\partial D\backslash \B)^{-1},$$
 		from where the lemma follows.
 		
 \end{proof}

 \begin{proof}[Proof of Proposition \ref{LawELBddFPS}:]
 	By conformal invariance it suffices to work in a circle domain. From the construction of first passage sets (see the observations after the proof of Proposition \ref{Existence}), we know that the first passage set can be constructed by using only level lines with boundary values in $\geq -a$ that do not touch $\partial D\backslash \B$. We can parametrize the part of the construction, that always continues in the connected component containing  $\partial D \backslash \B$ on its boundary, using its extremal length to $\partial D\backslash \B$. We denote the resulting local set process by $(A_t)_{0\leq t \leq \tau}$. Here $\tau$ is the first time that this component stops growing. In other words $\tau:=\ED(\B,\partial D\backslash \B)-\ED(\B \cup (\A_{-a}^u\backslash \partial D),\partial D \backslash \B))$ and moreover $\tau$ is the first time that satisfies the following property: 
 	\begin{itemize}
 		\item[ \clock]Restricted to the connected component of $O$ of $D \backslash A_\tau$ such that $\partial D\backslash \B\subseteq \partial O$, $h_{A_\tau}+u$ is the bounded harmonic function with boundary value $-a$ in $\partial O\backslash (\partial D \backslash \B)$ and $u_e$ in $\partial D \backslash \B$.
 	\end{itemize}
 	
 	Using \eqref{e.Poisson Kernel} and Proposition \ref{BProcess} for the underlying GFF $\Phi$,  we deduce that 
 	\begin{equation}\label{equu}
 	\widehat W_t:=u_e+\ED(\B \cup (A_t\backslash \partial D),\partial D \backslash \B)\left(- \int_{\partial D \backslash \B} \partial_{n}(h_{A_t}+u)\right)
 	\end{equation}
 	
 	is a Brownian bridge from $u_s$ to $u_e$ and of length $\ED(\B, \partial D \backslash \B)$.
 	
 	We now show that condition \clock$\ $ implies that $\tau$ is equal to the first time $\widehat W$ hits level $-a$. First note that \clock$\ $ implies that $\widehat W_\tau = -a$. Indeed, we get directly from the last statement of Theorem \ref{thmEL} that
 		\begin{align*}
 		\widehat W_\tau=  u_e+(a+u_e)\ED(\B \cup (A_\tau\backslash \partial D),\partial D \backslash \B) \int_{\partial D\backslash \B} \partial_n \bar u = -a.
 		\end{align*}
 	Here $\bar u$ is the bounded harmonic function with values $0$ in $\partial D\backslash \B$ and $1$ in $\B \cup A_\tau$. Similarly, we can verify that for all times $t$ smaller than $\tau$, we have that $\widehat W_t > -a$: indeed, for any such time $h_{A_t} - a\bar u$ is positive in the connected component of $D \backslash A_t$ containing $\partial D \backslash \B$ on its boundary. Thus from \eqref{equu} we see that $\widehat W_t > \widehat W_\tau$.

 \end{proof}
 
 Let us point out the following corollary
 \begin{cor}\label{Not touch E}
 	Let $\A_{-a}^u$ be an FPS of $\Phi+u$, where $u$ is a bounded harmonic function with piecewise constant boundary data. Then $\A_{-a}^u\cap  D$ is at positive distance of any connected component of $\partial D$ where $u\leq-a$. 

 \end{cor}
 
 Notice that if only a part of the boundary satisfies $u \leq -a$, then we also know that the FPS stays at a positive distance of any point on this interval. Indeed, this follows from the level line construction: we know from the proof of Theorem \ref{two-value gen f} and the remarks following the proof that for $\A_{-a,b}^u$ any connected set of $u^{-a,b}$ is entirely part of the boundary of a component of $D \backslash \A_{-a,b}$; on the other hand we also know that any component with 
 $h_{\A_{-a,b}^u}+u\leq -a$ in the complement of $\A_{-a,b}^u$, will also be a component of the complement of $\A_{-a}^u$. Putting this together we conclude. 
 \begin{cor}\label{Not touch E2}
 	Let $\A_{-a}^u$ be an FPS with boundary condition $u$ of $\Phi$, where $u$ is a bounded harmonic function with piecewise constant boundary data. Then $\A_{-a}^u\cap  D$ is at positive distance of any connected $J\subseteq \partial D$ where there is an open neighborhood $J_\epsilon$ such that $u(x)\leq-a$ for all $x\in J_\epsilon\cap \partial D$.
 \end{cor}
 
 Finally, let us mention that one can, via the same proof, also prove an analogue of Proposition \ref{LawELFPS} for two-valued local sets. Notice that it is possible to identify the exact law only in the case of the annulus, as explained by Figure \ref{f.ona}.
 
   \begin{figure}[h!]
   	\includegraphics[width=0.5\textwidth]{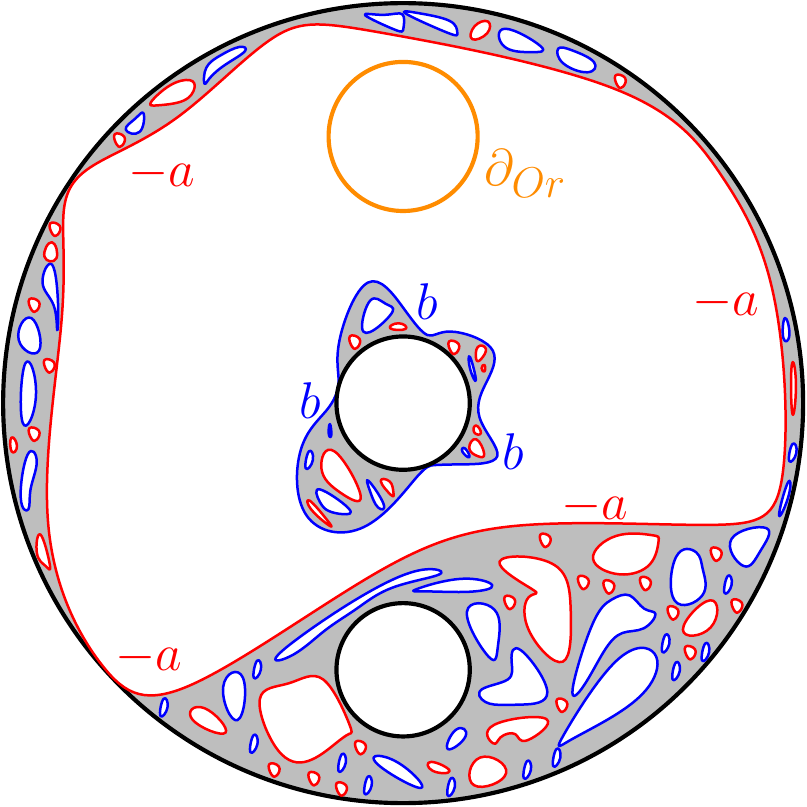}
   	\caption{Note that if we want to measure the extremal distance to $\partial_{Or}$, we may see more than one boundary value on the boundary of the connected component of $D\backslash \A_{-a,b}^u$ having $\partial_{Or}$ as a part of its boundary. This is not the case for FPS, or when the domain is an annulus. }
   	\label{f.ona}
   \end{figure}
 
 \begin{prop} \label{LawELBdd}
 	Let $a, b$ be positive with $a+b \geq 2\lambda$ , $D$ a n-connected domain, and $\B$ a union of connected components of $\partial D$. Moreover, let $u$ be a bounded harmonic function with piecewise constant boundary data changing finitely many times such that $u$ on $\partial D \backslash \B$ is a constant equal to $u_e \notin (-a,b)$. Let $\widehat{W}_t$ be a Brownian bridge with starting point:
 	$$u_s:=\ED(\B,\partial D\backslash \B)\iint_{\B \times\partial D\backslash\B} u(x)H_D(dx,dy)$$ endpoint $u_e$, and with length $\ED(\B,\partial D \backslash \B)$. 
 	If	 $u \in [-a,b]$ on $\B$, then 
 	\[\ED(\B, \partial D \backslash \B) - \ED(\B \cup (\A_{-a}^u\backslash \partial D), \partial D \backslash \B)\]
 	is stochastically bounded by the law of the first hitting time of $\{-a,b\}$ by 
 	$\widehat{W}_t$. Furthermore if $D$ is an annular domain, then they have exactly the same law.
 \end{prop}

 \subsection{Level lines as boundaries of FPS}
 
 Now, let us see that level line can be identified with the boundary of certain FPS. Let $D$ be finitely connected domain and $\partial_{\rm ext} D$ be the outermost connected component of $\partial D$, that is to say the one that separates $D$ from infinity. We consider two boundary points
 $x_{0}\neq y_{0}\in\partial_{\rm ext}D$ that split 
 $\partial_{\rm ext} D$ in two boundary arcs, $\mathcal{B}_{1}$ and $\mathcal{B}_{2}$ such that $y_0, \mathcal{B}_2, x_0$ come in clockwise order. Assume that $u$ is a bounded harmonic function with piecewise boundary values which are smaller than or equal to $-\lambda$ on $\B_2$, $\inf_{\mathcal{B}_{1}}u>-\lambda$ and 
 $\inf_{\partial D\backslash\partial_{\rm ext}D}u\geq \lambda$. Note that thanks to Lemma \ref{lemext} there is a generalized level line $\eta$ of $\Phi + u$ starting at $y_0$ and targeted at $x_0$.
 
 \begin{lemma} \label{LL as boundary}
 	Let $u$ be as above and $\Phi$ be a GFF in $D$. Then, there exists a connected component  $O$ of $D\backslash \A_{-\lambda}^u$ such that  $\partial_{\rm ext} O$ is equal to the union of $\B_2$ and the generalized level line of $\Phi + u$ from $y_0$ to $x_0$.
 \end{lemma}
 
 \begin{proof} 
 	
 	This basically just follows from the uniqueness of the FPS, i.e. Proposition \ref{Uniqueness}.
 	
 	First, construct $\eta$, the generalized level line of $\Phi + u$ from $y_0$ to $x_0$. By Lemma \ref{notouch}, it intersects $\B_2$ only at $y_0$ and $x_0$ and it does not intersect $\partial D\backslash \partial_{\rm ext} D$. Denote by $O$ the connected component of $D\backslash \eta$ containing $\B_2$ on its boundary: i.e. when we denote by $\partial_{\rm ext} O$ the outermost connected component of its boundary, then it holds that 
 	$\partial_{\text{ext}} O= \B_2\cup \eta$. To construct the FPS of level $-\lambda$ in $D$ it remains to construct the FPS of level $-\lambda$ inside each connected component of $D \backslash \eta([0,\infty])$. 
 	
 	If all of the boundary components of $\partial D \backslash \partial_{\rm ext} D$ are outside of $O$, then we can see that $O$ remains fixed in the remaining construction and thus by the uniqueness of FPS is a component of the complement of $\A_{-\lambda}^u$.
 	
 	If some of them are inside $O$, we have to make sure that the FPS of level $-\lambda$ constructed inside $O$ does not intersect $\B_2\cup \eta$. To do this, notice that the boundary values of $\partial_{\text{ext}} O= \B_2\cup \eta$ are bounded from above by $-\lambda$. Thus, by Corollary \ref{Not touch E}, the FPS of level $-\lambda$ inside $O$ stays at positive distance of $\partial_{\text{ext}} O$.
 	
 \end{proof}

 \begin{rem}
 	\label{RemLvlLine}
 	Let us remark that one can also obtain families of multiple level lines, as boundaries of the FPS $\A_{-\lambda}^u$. For instance, when $u$ changes finitely many times between $-\lambda$ and $\lambda$ in a simply connected domain, one can get in this way multiple commuting SLE$_{4}$ curves.	In \cite{PeltolaWu2017MultComSLE} the authors provide an explicit expression for probabilities of different pairings.
 \end{rem}
 
 \subsection{Uniqueness, monotonicity and local finiteness of $\A_{-a,b}^u$} In this subsection we will prove that $\A_{-a,b}^u$ is the only BTLS satisfying \hyperlink{tvs}{(\twonotes)}. The monotonicity then follows as in the proof of Proposition \ref{cledesc3}. 
 
 To prove the uniqueness of $\A_{-a,b}^u$, we show that it can represented as a simple function of $\A_{-a}^u$ and $\AB_{b}^u$, and then use the uniqueness of the FPS. 
 
 We say that a set $A\subseteq D$ is connected to the boundary if every connected component of $A\cup \partial D$ contains at least one boundary component of $\partial D$.
 
 \begin{prop}\label{EaEbAab}
 	Let $\Phi$ be a GFF in an $n$-connected domain $D$. Then almost surely the union of the components of $\A_{-a}^u\cap \AB_{b}^u$ that are connected to boundary is equal to $\A_{-a,b}^u$. In particular, if $D$ is simply connected we have that $\A_{-a,b}^u=\A_{-a}^u\cap \AB_{b}^u$ (see Figure \ref{Intersection}).
 \end{prop}
 
 \begin{proof}
 	From the monotonicity of two-valued sets (Theorem \ref{two-value gen f}) and the construction of the FPS (Lemma \ref{Existence}) we see that for any $n \geq a \vee b$, we have that $\A_{-a,b}^u\subseteq \A_{-a,n}^u \subseteq A_{-a}^u$ and, furthermore, $\A_{-a,b}^u\subseteq \A_{-n,b}^u \subseteq \AB_b^u$. Hence $\A_{-a,b}^u \subseteq \A_{-a}^u\cap \AB_{b}^u$. Moreover, as by construction $\A_{-a,b}^u$ is connected to the boundary, we deduce that $\A_{-a,b}^u$ is contained on the union of connected components of $\A_{-a}^u \cap \AB_b^u$ that are connected to the boundary.
 	
 	We will now prove the opposite inclusion. To do this, it suffices to show that for every connected component $O$ of $D\backslash \A_{-a,b}$, either it is contained in $D\backslash \A_{-a}^u$ or $D\backslash\AB_{b}^u$, or the closure $O\cap\A_{-a}^u\cap \AB_{b}^u$ does not intersect $\partial O$.
 	
 	First, let us consider the case where $h_{\A_{-a,b}^u}+u$ is less than or equal to $-a$ in $O$. By construction of the FPS (Lemma \ref{Existence}), we see that $O$ is also a connected component of $D\backslash \A_{-a}^u$. Similarly,  if $h_{\A_{-a,b}^u}+u$ is greater than or equal to $b$ in O, then $O$ is a component of $D\backslash \AB_{b}^u$.
 	
 	Thus, it remains to deal with the case when $O$ has more than one boundary component, and where $u$ takes values that less than or equal to $-a$ on some of the boundary components, and values that are greater than or equal to $b$ on all the other boundary components.
 	
 	Now, by the uniqueness of the FPS (Proposition \ref{Uniqueness}) we know that in order to construct the part of FPS $\A_{-a}^u$ intersecting $O$, we can first take $\A_{-a,b}^u$, and then construct the FPS of level $-a$ in $O$. But now, by Corollary \ref{Not touch E} we know that an FPS of level $-a$ inside $O$ cannot touch any part of the boundary where $u$ is smaller than or equal to $-a$. But similarly we see that the part of FPS $\AB_{b}^u$ intersecting $O$, cannot touch the boundary of $O$ where $u$ is larger than or equal to $b$. Hence the closure of $\A_{-a}^u \cap \AB_{b}^u \cap O$ cannot intersect the boundary of $O$ and hence is not connected to the boundary of $D$. The proposition follows.
 	
 \end{proof}
 
 \begin{figure}[ht!]	   
 	\centering
 	\includegraphics[width=2in]{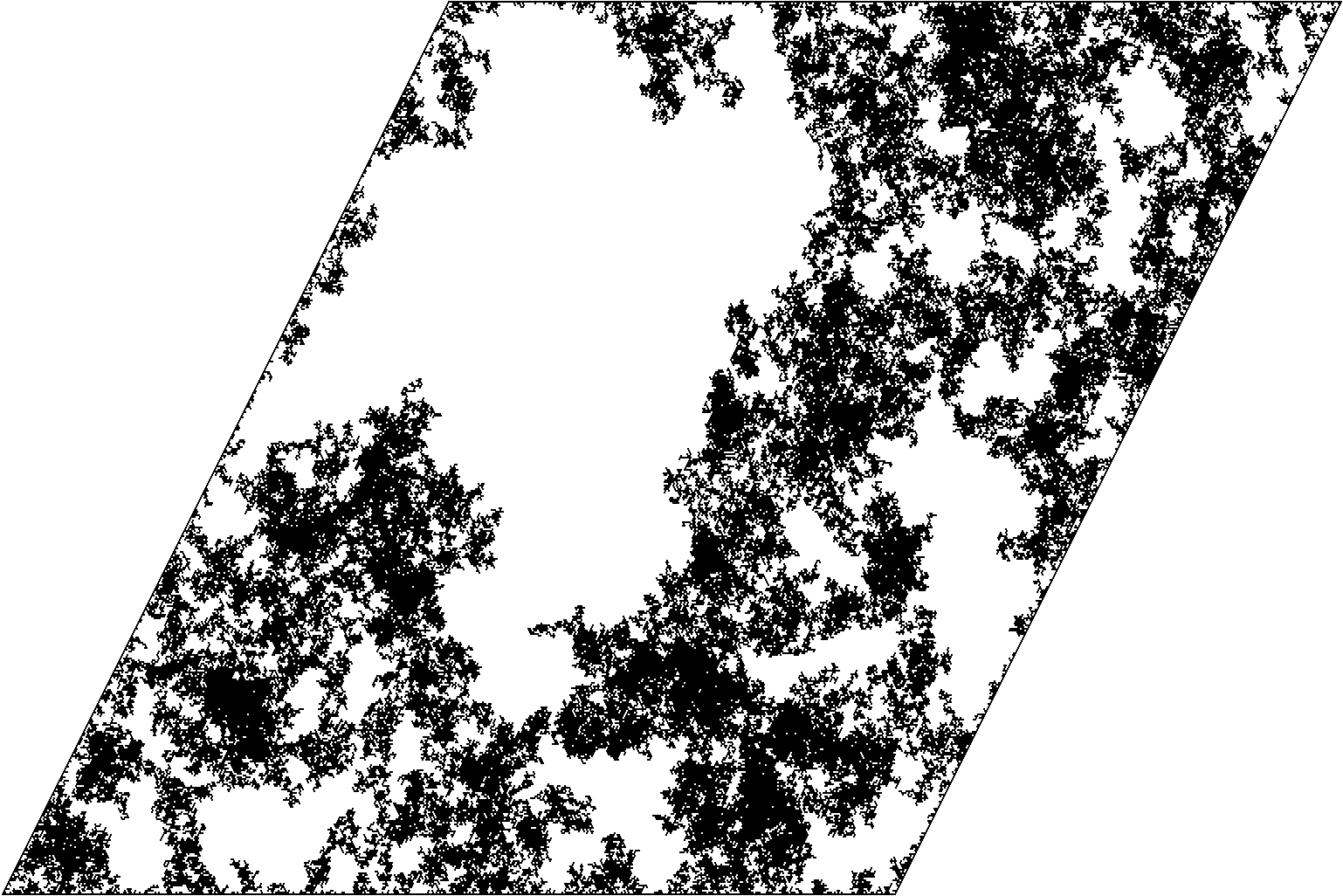}
 	\includegraphics[width=2in]{ALEll.png}
 	\includegraphics[width=2in]{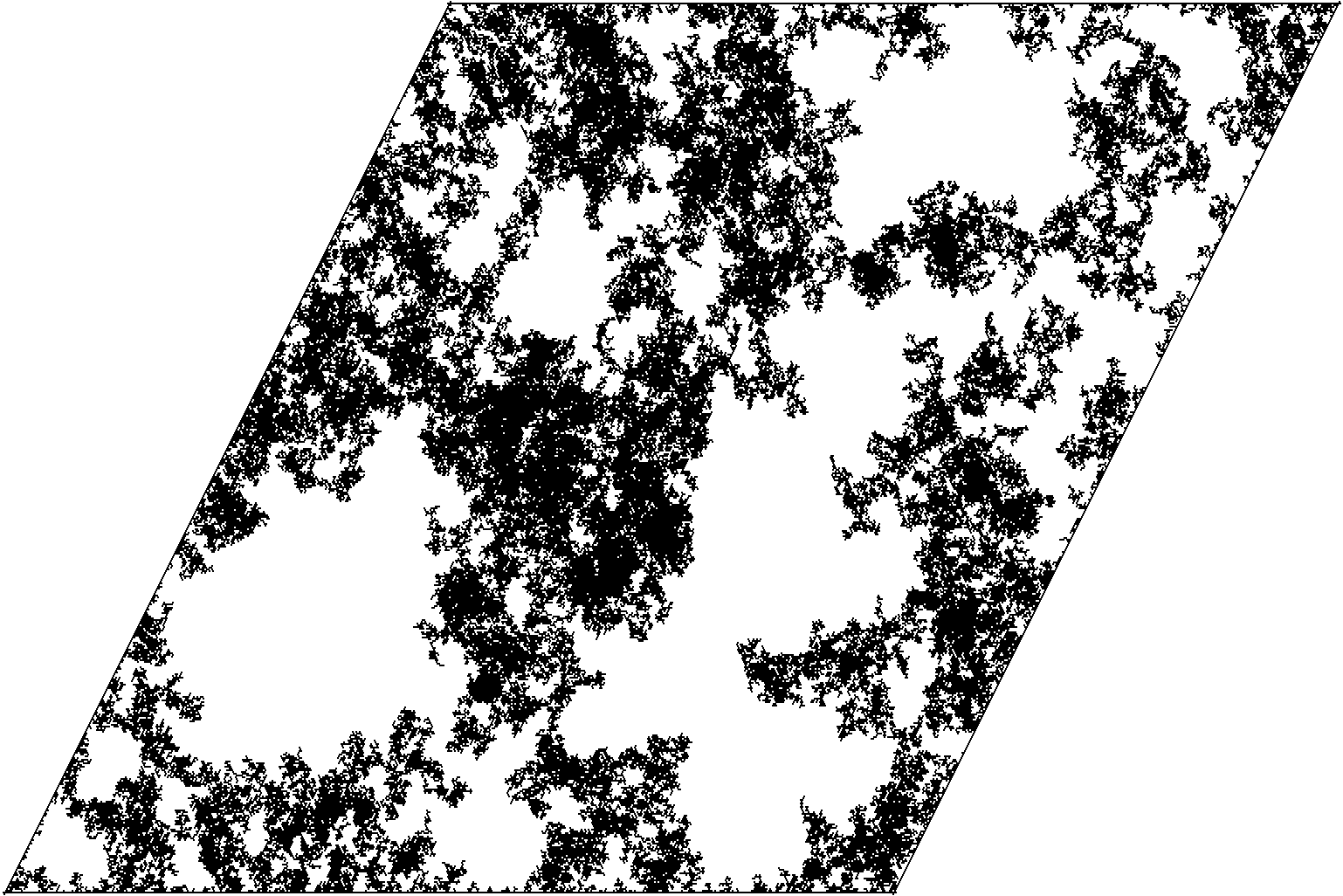}
 	\caption{Graphical evidence of Proposition \ref{EaEbAab}: On the left  $\A_{-\lambda}$, on the right {$\AB_{\lambda}$} and in the middle $\A_{-\lambda,\lambda}$. Simulation done by B. Werness.}
 	\label{Intersection}
 \end{figure}
 
 We are now ready to prove the uniqueness of two-valued sets $\A_{-a,b}^u$ for general boundary data in $n$-connected domains. See Theorem \ref{two-value gen f} for the setting and precise statement. 
 
 \begin{proof}[Uniqueness of two-valued sets]
 	In the proof of Theorem \ref{two-value gen f}, we showed the existence of a two-valued set. Denote this set by $\A_{-a,b}^u$.
 	Suppose $A'$ is another two-valued set coupled with the same GFF, i.e. it satisfies the condition \hyperlink{tvs}{(\twonotes)} given just before Theorem \ref{two-value gen f}. 
 	
 	First, notice that $A'$ has to be connected to the boundary: indeed, suppose for contradiction that there is a connected component $B$ of $A'$ that is not connected to the boundary. Consider a component $O$ of $D \backslash B$ that has $B$ on part of its boundary, and let $\mathcal{B} = \partial O \cap B$. WLOG suppose that the boundary conditions on $\mathcal{B}$ are equal to $-a$. Then, as in the last paragraph of proof of Proposition \ref{EaEbAab}, we see that the FPS of height $-a$ also contains $\B$, and that moreover $\B$ is also not connected to the boundary as a subset of the FPS. However, from the construction (Proposition \ref{Existence}) and uniqueness of the FPS (Proposition \ref{Uniqueness}) we know that the FPS is connected to the boundary.
 	
 	Note that in the proof of Proposition \ref{EaEbAab}, we only used the condition \hyperlink{tvs}{(\twonotes)} and the fact that $\A_{-a,b}^u$ is connected to the boundary. As the same conditions hold for $A'$, we can apply the same argument to conclude that almost surely the union of the components of $\A_{-a}^u\cap \AB_{b}^u$ that are connected to boundary is equal to $A'$. However, we know from Proposition \ref{EaEbAab} that it is equal to $\A_{-a,b}^u$ and the claim follows. 	
 \end{proof}
 
 Finally, let us state here the proposition about $\A_{-a,b}$ are locally finite. In fact this relies on a result from \cite{ALS2}, that proves the local finiteness of first passage sets, using a proof that uses the relation of FPS to Brownian loop-soups. We should maybe stress that none of \cite{ALS2} depends on local finiteness of $\A_{-a,b}$.
 
 \begin{prop}\label{prop:: locally finite}
 	Consider a finitely-connected domain with any piece-wise constant boundary condition $u$. For any $a,b \geq 0$, the two-valued set $\A_{-a,b}^u$ is almost surely locally finite. 
 \end{prop}
 \begin{proof}
 	Proposition 5.7 of \cite{ALS2} gives us the local finiteness of $\A_{-a}^u$ for any choice of $a \geq 0$, of a piece-wise constant boundary condition $u$ and of a finitely-connected domain $D$. But now, note that if $O$ is a simply connected component of $D\backslash \A_{-a}^u$, then it is also a connected component of either $D\backslash \A_{-a}^u$ or $D\backslash \AB_{b}^u$ - indeed, this follows from the uniqueness of the FPS, as when constructing say,  $\A_{-a}^u$ from $\A_{-a,b}^u$ one keeps all the components that have boundary condition less than or equal to $-a$. As there are only finitely many non-simply connected components of $D\backslash \A_{-a}^u$, we conclude.
 \end{proof}
 
 \begin{rem}
 	Let us stress that if one restricts oneself to the simply-connected case, then one can prove local finiteness of $\A_{-a}^u$ without relying on the construction or properties of FPS and TVS in multiply connected domains. In fact, the proof in \cite{ALS2} just uses properties of the Brownian loop-soup. In particular, this provides an alternative proof of Proposition \ref{prop::ALE locally finite}.
 \end{rem}
 

\section{The Minkowski content measure of the FPS}\label{sec:: Minkowski}

The aim of this section to identify the measure
\[\nu_{\A_{-a}^u} := \Phi_{\A_{-a}^u}-h_{\A_{-a}^u}\]
as a Minkowski content measure in a certain gauge: 

\begin{thm}\label{prop:: mes LQG}
	The measure $\nu_{\A_{-a}^u}$ is a measurable function of 
	$\A_{-a}^u$. Moreover, it is proportional to the Minkowski content measure in the gauge  
	$r\mapsto \vert \log(r)\vert^{1/2} r^{2}$. More precisely, almost surely for any continuous $f$ compactly supported in $D$,
	\[\nu_{\A_{-a}^u} = \lim_{r \to 0}\frac{1}{2}
	\vert\log(r)\vert^{1/2} \int_D f(z)  \1_{d(z,\A_{-a}^u)\leq r}dz.\]	
\end{thm} 

Let us stress that there are two non-obvious statements in this theorem:
\begin{enumerate}
	\item the fact that the measure $\nu_{\A_{-a}^u}$ can be obtained as a measurable function of $\A_{-a}^u$;
	\item and the identification of this function as the Minkowski content measure of $\A_{-a}^u$ in a certain gauge.
\end{enumerate}
As a simply corollary, we can deduce an almost sure statement on the dimension of $\A_{-a}^u$:

\begin{cor}
	If the harmonic function $u$ is not everywhere less or equal to $-a$
	(i.e. if $\A^{u}_{-a}\backslash\partial D$ is non-empty), then
	$\A^{u}_{-a}$ is a.s. of Minkowski dimension 2.
\end{cor}

Moreover, observe that the expected value of $\nu_{\A_{-a}}$ is equal to $a$ in the sense that for any continuous function $f$, the expectation of $(f, \nu_{\A_{-a}})$ is equal to $a$ times $ \int_D f(z) dz$. Thus, the fact that the measures $\Phi_{\A_{-a}}$ converge to $\Phi$ as $a \to \infty$, can be interpreted as saying that  the GFF is a limit of recentered Minkowski content measures of certain growing random sets.

\begin{cor}\label{cor::cute}
	For any $n \geq 1$, let $A_n$ be a random set with the law of the FPS $\A_{-n}$ in $\D$, and denote by $\nu_{A_n}$ its Minkowski content measure in the gauge $\frac{1}{2}\vert \log(r)\vert^{1/2} r^{2}$. Then, the recentered measures $\nu_{A_n} - \E[\nu_{A_n}]$ converge in law in $H^{-1}(\D)$ to a zero boundary GFF on $\D$. 
\end{cor}

\begin{proof}
To prove this corollary, we construct a coupling where $\nu_{A_n} - \E[\nu_{A_n}]$ converges in probability to a GFF in $\D$. We take $\Phi$ a GFF in $\D$, and we couple $A_n$ such that for all $n$, $A_n=\A_{-n}(\Phi)$. Now, note that for any smooth $f$,  we have that $\E\left[\int f d\nu_{A_n} \right]=\E\left[(\Phi_{\A_{-n}}+n,f) \right]=n(1,f)$, thus $\E\left[\nu_{A_{n}} \right]$ is equal to $n$ times Lebesgue measure on $\D$.  Furthermore, $\Phi-\nu_{A_{n}}+n$ has the law of a GFF in $D\backslash A_{n}$, conditionally on $A_n$. Thus, 
		\begin{equation}\label{e. min to GFF}
		\E\big[\|\Phi-\nu_{A_{n}}+n\|_{H^{-1}(\D)}^2\big]= \E\left[ \iint _{\D\times \D}G_{\D\backslash A_n}(w,z)G_{\D} (w,z) dw dz\right].
		\end{equation}
		As  a.s. $A_n$ is connected and $A_{n}\nearrow \D$ for the Hausdorff topology, the Beurling estimate implies that $G_{\D\backslash A_n} \to 0$ a.e. in $\D\times \D$. The dominated convergence theorem implies that \eqref{e. min to GFF} converges to 0, giving the result. \end{proof}

Let us also make two remarks.
\begin{rem}\label{rem::compactly supported}
	In Theorem \ref{prop:: mes LQG}, we restricted to functions supported compactly away from $\partial D$, because by construction,
	$\nu_{\A^{u}_{-a}}(\partial D)=0$, yet $\partial D$ might be irregular enough to have a positive, or even infinite Minkowski content measure. Thus,
	$\nu_{\A^{u}_{-a}}$ is rather the Minkowski content measure of $\A^{u}_{-a}\backslash\partial D$.
\end{rem}

\begin{rem}
	In fact, the theorem holds in a more general setting (with basically the same proof). Indeed, take $A$ to be a local set of the GFF $\Phi$ such that there exists a deterministic $K$ such that $\Phi_{A}=\nu_{A} + h_{A}\geq- K$, where $h_{A}$ is harmonic of $D\backslash A$ and $\nu_{A}$ is a non-negative measure supported on $A$. Then, $\nu_{A}$ is given by the Minkowski	content measure of $A\backslash\partial D$ as in Theorem \ref{prop:: mes LQG}. 
\end{rem}

The rest of this section is devoted to the proof of Theorem \ref{prop:: mes LQG}. We start by looking at the decomposition of the Gaussian multiplicative chaos (GMC) measure induced by the Markovian decomposition of the field w.r.t the FPS. Then, we use this decomposition to see how the geometry of the FPS encodes its height and to give a short, but somewhat unexpected proof of why $\Phi_{\A_{-a}^u}$ is a measurable function of $\A_{-a}^u$. Next, we observe that a careful analysis of the proof gives in fact an explicit expression of $\Phi_{\A_{-a}^u}-h_{\A_{-a}^u}$ in terms of $\A_{-a}^u$, and finally we identify this expression with the Minkowski content measure of $A$ in the gauge $r\mapsto \vert \log(r)\vert^{1/2} r^{2}$. 

In several of the arguments it will be more convenient to work with $\AB_{b}^u$, but notice that by the symmetry of the GFF this is equivalent.

\subsection{Decomposition of the GMC using the FPS}

In \cite{APS}, the decomposition of the GMC measures w.r.t. to the FPS \footnote{In that paper we only use a specific construction of the FPS, basically the one given by Prop \ref{Existence}, and properties stemming from that construction.} and w.r.t. more general local sets was used to view the GMC measure of the GFF as a multiplicative cascade. In particular, it was observed that in the case of the FPS one recovers a construction of \cite{Ai}.

As in \cite{APS} the decomposition of the GMC measures was stated w.r.t. the sets $\AB_{b}^u$ instead of $\A_{-a}^u$ (as this allows to consider positive values of $\gamma$), we will do the same here. 

Define $\tilde \gamma= \sqrt{2\pi} \gamma$ and $\F_{\AB_{b}^u}:=\sigma(\AB_{b}^u,\Phi_{\AB_{b}^u})$. Then from the Proposition 4.1 of \cite{APS} (see also Section 5 of the same paper for the discussion on multiply-connected domains) it follows that: 
\begin{prop}\label{CEV GMC}
	Take $0\leq \gamma <2$, i.e., $0\leq \tilde \gamma<2\sqrt{2\pi}$. Then, for all  continuous function $f$ compactly supported in $D$,
	\[ \E\left[(\mu_\gamma,f)\mid \F_{\AB_{b}^u}\right]= \int_{D\backslash \AB_{b}^u} f(z)\exp\left (\tilde \gamma h_{\AB_{b}^u}(z)+\frac{\tilde \gamma^2}{2}g_{D\backslash \AB_{b}^u}(z,z)\right ) dz\]
	where $g_D$ was defined in \eqref{EqLogSing}. Moreover, conditioned on $\AB_{b}^u$, we have that
	\begin{equation}\label{2ndeq}
	(\mu_\gamma,f) = \int_{D\backslash \AB_{b}^u} f(z)\exp\left (\tilde \gamma h_{\AB_{b}^u}(z)\right ) d\mu_\gamma^{\Phi^{\AB_b^u}}(z),
	\end{equation}
	where $\mu_\gamma^{\Phi^{\AB_b^u}}$	is the GMC measure for the field $\Phi^{\AB_b^u}$.
\end{prop}
\begin{rem}
	In fact the second equality of this proposition is not part of the statement of Proposition 4.1 in \cite{APS}, but follows directly from its proof.
\end{rem}
As from condition (1) of the definition of the FPS it follows that in the coupling $(\Phi, \AB^u_{b})$, the harmonic extension $h_{\AB^u_{b}}$ is a measurable function only of the set $\AB^u_{b}$, we infer a direct but useful corollary:
\begin{lemma} \label{mes LQG}
	$(\mu_\gamma)_{0\leq \gamma<2}$ is a measurable function of $\AB_{b}^u$ and $\Phi^{\AB_{b}^u}$.
\end{lemma}
Observe that this means that in order to construct the Liouville measure $\mu_\gamma$ we do not need the positive measure $\nu_{\AB^u_{b}}$.

\subsection{How to read the height of the FPS from its geometry}

We have seen that the FPS is measurable w.r.t. the underlying free field. 
A natural question is whether one can also determine $-a$ in an explicit way by just looking at the geometry set $\A_{-a}$. The following proposition says that this is indeed the case, and it presents some of the basic tools later used to show how to construct $\Phi_{\A_{-a}^u}$ as an explicit function of $\A_{-a}^u$. As this is more of a little side-story, we state and prove it in the simplest setting of the unit disk and the zero boundary GFF $\Phi$. A similar result can be stated and proved more generally for finitely-connected domains and general boundary conditions.

\begin{prop}\label{prop::FPSreada}
	Consider the FPS $\A_{-a}$ of the GFF $\Phi$ defined on $\D$ and let $\gamma \in (0,\sqrt 2)$. 
	Then, we have that $\int_0^{2\pi} \left(\frac{\crad(re^{i\theta},D\backslash\A_{-a})}{\crad(re^{i\theta},\D)}\right)^{\pi\gamma^2}d\theta$ converges in probability to $2\pi e^{-\gamma\sqrt{2\pi}a}$ as $r\to 1$.
\end{prop}

\begin{proof}
	It is slightly more convenient to work with $\AB_{b}$ instead of $\A_{-a}$. Denote by $$I_r := \int_0^{2\pi} \left(\frac{\crad(re^{i\theta},\D\backslash\AB_{b})}{\crad(re^{i\theta},\D)}\right)^{\pi\gamma^2}d\theta.$$ By taking the expectation in \eqref{2ndeq} and using \eqref{LVexp} we obtain that for any $z \in \D$, $$\E\left[ \left(\frac{\crad(z,\D\backslash\AB_{b})}{\crad(z,\D)}\right)^{\pi\gamma^2}e^{\gamma\sqrt{2\pi}b}\right] = 1.$$ Thus, the mean of $I_r$ is equal to $2\pi e^{-\gamma \sqrt{2\pi b}}$. It remains to show that its variance tends to zero as $r \to 1$. 
	
	Let $\Phi_\epsilon$ is the circle average process introduced in Section \ref{GMC}. By the Markov decomposition w.r.t. $\AB_{b}$, we can write $\Phi_\eps(z) = \Phi_\eps^{\AB_b}(z) + \Phi_{\AB_b,\eps}(z)$. The second term is equal to $b$ for any $z$ of distance larger than $\eps$ from $\AB_b$. Thus by the explicit formula for the Laplace transform of the Gaussian $\Phi_\eps^{\AB_b}(z)+\Phi_\eps^{\AB_b}(w)$, we have that for any $z \neq w \in \D$ and $\eps$ small enough
	\[\E\left[e^{\gamma\sqrt{2\pi}(\Phi_\eps(z)+\Phi_\eps(w))-\frac{\gamma^2}{2}2\pi(\E(\Phi_\eps(z)^2)+\E(\Phi_\eps(w)^2))} \lvert \AB_{b}\right]\]
	is equal to
	\[e^{\gamma^2 4\pi G_{D\backslash \AB_{b}}(z,w)}\left(\frac{\crad(z,\D\backslash\AB_{b})}{\crad(z,\D)}\right)^{\pi\gamma^2}\left(\frac{\crad(w,\D\backslash\AB_{b})}{\crad(w,\D)}\right)^{\pi\gamma^2}e^{2\gamma\sqrt{2\pi}b}, \]

	Thus, as the Green's function is always positive, using bounded convergence theorem $\E\left[ I_r^2\right]$ is upper bounded by
	\[ \lim_{\eps \to 0}e^{-2\gamma\sqrt{2\pi}b}\int_0^{2\pi}\int_0^{2\pi}\E \left[ e^{\gamma\sqrt{2\pi}(\Phi_\eps(re^{i\theta_1})+\Phi_\eps(re^{i\theta_2}))-\frac{\gamma^2}{2}2\pi(\E(\Phi_\eps(re^{i\theta_1})^2)+\E(\Phi_\eps(re^{i\theta_2})^2))}\right]d\theta_1d\theta_2.\]
	However, this expected value is again given by the Laplace transform of a Gaussian and we can get that the limit above is equal to
	\[ e^{-2\gamma\sqrt{2\pi}b}\int_0^{2\pi}\int_0^{2\pi}e^{\gamma^2 G_D(re^{i\theta_1},re^{i\theta_2})}d\theta_1d\theta_2<\infty \]
	as $\gamma<\sqrt 2$. Finally, using dominated convergence and the fact that a.e. $G_D(re^{i\theta_1},re^{i\theta_2})$ converges to $0$ as $r\to 1$, we have that $\E\left[I_r^2\right]$ converges as $r\to 1$ to $( 2\pi e^{-\gamma \sqrt{2\pi b}})^2$.
\end{proof}

\subsection{Measurability of $\Phi_{\A_{-a}^u}$ w.r.t. $\A_{-a}^u$}

We will now give a short argument to prove the measurability of $\Phi_{\A_{-a}^u}$ w.r.t. $\A_{-a}^u$: 

\begin{prop}\label{lem:mes}
	Let $\Phi$ be a GFF in $D$. We have that $\Phi_{\A_{-a}^u}$ is a measurable function of the set $\A_{-a}^u$.
\end{prop}

\begin{proof}
	
	In fact it is again clearer to prove the claim for the sets $\AB_b^u$.

	To prove the measurability of $\Phi_{\AB_{b}^u}$, consider $\mu_\gamma$ the $\gamma-$GMC measure corresponding to 
	$\sqrt{2\pi}\Phi$.
	The proof is based on three measurability statements:
	\begin{enumerate}
		\item $\Phi_{\AB_{b}^u}$ is a measurable function of $\Phi$: this follows from Lemma \ref{Existence}.
		\item $\mu_\gamma$ is a measurable function of $\AB_{b}^u$ and $\Phi^{\AB_{b}^u}$: this follows from Lemma \ref{mes LQG}.
		\item $\Phi$ is a measurable function of $(\mu_\gamma)_{0\leq \gamma<2}$: this follows from Theorem \ref{thm::Derivative} (see also Remark \ref{Measurability GMC}).
	\end{enumerate}    
	
	Thus, if $F$ is a bounded measurable function, we have that
	\begin{align*}
	F(\Phi_{\AB_{b}^u})=\E\left[F(\Phi_{\AB_{b}^u})\mid \Phi \right]=\E\left[F(\Phi_{\AB_{b}^u})\mid (\mu_\gamma)_{0\leq \gamma <2}\right]= \E\left[F(\Phi_{\AB_{b}^u})\mid \AB_{b}^u, \Phi^{\AB_{b}^u}  \right].
	\end{align*}
	In other words, $\Phi_{\AB_{b}^u}$ is determined by the couple $(\AB_{b}^u, \Phi^{\AB_{b}^u} )$. Furthermore, note that when $A$ is a local set, it follows from the definition that $\Phi^A$ and $\Phi_A$ are conditionally independent given $A$. Hence $$\E\left[F(\Phi_{\AB_{b}^u})\mid \AB_{b}^u, \Phi^{\AB_{b}^u}  \right] = \E\left[F(\Phi_{\AB_{b}^u})\mid \AB_{b}^u \right].$$ 
	This proves the proposition. 
\end{proof}

\subsection{Explicit expression for the measure $\nu_{\A_{-a}^u}$}
Let us now derive an explicit expression for the measure $\nu_{\A_{-a}^u}$. In fact, it is simpler to continue working with $\AB_{b}^u$, and thus write $\nu_{\AB_{b}^u} := h_{\AB_b^u} - \Phi_{\AB_b^u}$. Notice the difference in sign in the definition due to the fact of taking the FPS in the other direction. Our aim is to prove that:

\begin{prop}
	\label{LemFgamma}
	Almost surely for all $f$ continuous function compactly supported in $D$,
	\begin{equation}
	\label{EqNuF}
	(\nu_{\AB^{u}_{b}},f)=
	\lim_{\gamma\to 0}\frac{\gamma}{\sqrt{2\pi}}\int_{D\backslash \AB_{b}^u}
	f(z)|\log d(z,\AB_{b}^u)|d(z,\AB_{b}^u)^{\gamma^2/2}dz.
	\end{equation}
	By symmetry, the same holds for $\nu_{\A_{-a}^u}$.
\end{prop}

To prove this proposition we will first look at the step 3 in the proof of Proposition \ref{lem:mes} in more detail. Indeed, this step stems from \eqref{Derivative}, which identifies the GFF as the derivative of the GMC measures at $\gamma = 0$, i.e. for any compactly supported continuous $f$:
\[(\Phi,f) = \frac{1}{\sqrt {2\pi}}\partial_{\gamma} (\mu_\gamma,f)\mid_{\gamma=0}.\]
Let us take conditional expectation with respect to $\F_{\AB_b^u}$ and write
\[	(\Phi_{\AB_{b}^u},f)=\E\left[(\Phi,f)\mid \F_{\AB_{b}^u}\right]=\frac{1}{\sqrt{2\pi}}\E\left[\lim_{\gamma \to 0^{+}}\partial_\gamma 
(\mu_{\gamma},f) \mid \F_{\AB_{b}^u} \right].\]
Now, by Proposition \ref{thm::Derivative} $\partial_\gamma (\mu_{\gamma},f)$ converges in $\LL^1$ as $\gamma\to 0^{+}$ towards $\sqrt{2\pi}(\Phi,f)$ and thus one can interchange the limit as $\gamma\to 0^{+}$ and the conditional expectation. For the same reason, one can also change the order of the $\gamma-$derivative and the conditional expectation. We obtain:
\begin{equation}\label{petit derivative}
(\Phi_{\AB_{b}^u},f)= 
\frac{1}{\sqrt{2\pi}}\lim_{\gamma\to 0^{+}} \partial_\gamma \E\left[(\mu_{\gamma},f)\mid \F_{\AB_{b}^u}\right]
\end{equation}
We claim that:

\begin{lemma}
	\label{LemApplyComDeriv}
	Almost surely, for all compactly supported functions $f$ on $D$,
	\begin{equation}
	\label{EqMeas1}
	(\nu_{\AB_{b}^u},f):=(h_{\AB_{b}^u}-\Phi_{\AB_{b}^u},f)=
	\lim_{\gamma\to 0^{+}}
	-\sqrt{2\pi}\gamma
	\int_{D\backslash \AB_{b}^u}
	f(z)(g_{D\backslash \AB_{b}^u}(z,z))
	e^{\frac{\tilde\gamma^2}{2}g_{D\backslash \AB_{b}^u}(z,z)
	}dz.
	\end{equation}
	By symmetry, the same holds for $\nu_{\A_{-a}^u}$.
\end{lemma}

\begin{proof}
	By Proposition \ref{CEV GMC}, we have that $\partial_{\gamma}\E\left [(\mu_\gamma,f)\vert \AB^{u}_{b}\right ]$ is equal to
	\begin{equation*}
	\sqrt{2\pi}\int_{D\backslash \AB_{b}^u}f(z)\left[h_{\AB_{b}^u}(z)+\sqrt{2\pi}\gamma g_{D\backslash \AB_{b}^u}(z,z)\right]
	e^{\tilde \gamma h_{\AB_{b}^u}(z)
		+\frac{\tilde \gamma^2}{2}g_{D\backslash \AB_{b}^u}(z,z)
	}dz
	.
	\end{equation*}
	Using the fact that $\tilde \gamma h_{\AB_{b}^u}(z)
	+\tilde \gamma^2g_{D\backslash \AB_{b}^u}(z,z)/2
	$ converges to $0$ as $\gamma \to 0$ and \eqref{petit derivative}, we obtain the result for a fixed function $f$. Thus, this also holds for a countable dense family, for the uniform convergence on compact subsets, in the space of continuous functions on $D$. 
	Now, notice that as $g_{D\backslash \AB_{b}^u}(z,z) \leq g_D(z,z) < C_r$ for any compact subset $D_r \subset D$, the measure defined by the RHS of \eqref{EqMeas1} on any such compact subset $D_r$ is of bounded total mass. As the space of such measures is compact, they have a weak limit and \eqref{EqMeas1} holds simultaneously for all $f$ continuous on $D_r$. As we can take $D_r \to D$, the lemma follows.
\end{proof}

In order to deduce Proposition \ref{LemFgamma}, we will now make use of a simple estimate relating $g_D(z,z)$ as defined in  \eqref{EqLogSing} with the distance of $z$ to $\partial D$. This result follows directly from Koebe's quarter theorem in simply connected domains, but in the case of $n-$connected domains requires an argument.

\begin{lemma}
	\label{LemKoebe}
	For any $D$ open bounded connected finitely connected domain
	of $\C$ with non-polar boundary components, there is a constant $c=c(D)>1$ such that 
	\begin{displaymath}
	\forall z\in D,
	d(z,\partial D)\leq
	e^{2\pi g_{D}(z,z)}\leq c d(z,\partial D).
	\end{displaymath}
\end{lemma}

\begin{proof}
	If $D$ is simply connected, then
	$e^{2\pi g_{D}(z,z)}$ equals the conformal radius
	$\operatorname{CR}(z,D)$, and the results follows from Koebe's quarter theorem, with $c=4$.
	
	For the other cases, note that $2\pi g_{D}(z,z)$ is the value at $z$ of the
	harmonic extension of the boundary values
	$x\mapsto \log (\vert z-x\vert)$, $x\in\partial D$. This already implies that
	\begin{displaymath}
	2\pi g_{D}(z,z)\geq \log(d(z,\partial D)).
	\end{displaymath}
	Take $B^z$ a Brownian motion started at $z\in D$ and  define $T_{\partial D}$ as the first time $B^z$ hits $\partial D$.
	Let $\delta$ be the smallest diameter of an inner hole of $D$.
	With the Beurling estimate (Theorem 3.76 of \cite{LawC}) we get that 
	$$\mathbb{P}(|z-B^z_{T_{\partial D}}|\geq r d(z,\partial D))
	\leq Cr^{-1/2},$$
	where $C$ is a constant not depending on $z$, as long as the open disk $B(z,r d(z,\partial D))$ does not entirely contain an inner hole of $D$. 
	A sufficient condition for that is 
	$r d(z,\partial D)\leq\delta$. Thus,
	\begin{eqnarray*}
		2\pi g_{D}(z,z)&=&
		\log(d(z,\partial D))
		+\int_{1}^{\frac{\operatorname{diam}(D)}{d(z,\partial D)}}
		\mathbb{P}(|z-B^z_{T_{\partial D}}|\geq e^{s} d(z,\partial D))ds
		\\&\leq &
		\log(d(z,\partial D))
		+\int_{1}^{\frac{\delta}{d(z,\partial D)}}
		\mathbb{P}(|z-B^z_{T_{\partial D}}|\geq e^{s} d(z,\partial D))ds
		\\&&+\mathbb{P}(|z-B^z_{T_{\partial D}}|\geq \delta)
		\log\left(\dfrac{\operatorname{diam}(D)}{d(z,\partial D)}\right)
		\\&\leq&
		\log(d(z,\partial D))
		+C\int_{1}^{+\infty}
		e^{-\frac{s}{2}}ds
		+C\dfrac{d(z,\partial D)^{1/2}}{\delta^{1/2}}
		\log\left(\dfrac{\operatorname{diam}(D)}{d(z,\partial D)}\right).
		\qedhere
	\end{eqnarray*}
\end{proof}

We can now prove Proposition \ref{LemFgamma}.

\begin{proof}[Proof of Proposition \ref{LemFgamma}:]
	First, note that $D\backslash\A^{u}_{-a}$ has a.s. finitely many non simply connected components.
	By \ref{LemKoebe}, there is some $c>1$, such that for all
	$z\in D\backslash\A^{u}_{-a}$,
	\begin{equation*}
	d(z,\A^{u}_{-a})\leq
	e^ {2\pi g_{D\backslash\A^{u}_{-a}}(z,z)}\leq c d(z,\A^{u}_{-a}).
	\end{equation*}
	Plugging this into \eqref{EqMeas1} and taking $f$ positive, we obtain that on the one hand
	\begin{equation*}
	\limsup_{\gamma\to 0^{+}}
	-\dfrac{\gamma}{\sqrt{2\pi}}
	\int_{D\backslash\A^{u}_{-a}}
	f(z)\log(d(z,\partial D)+c)d(z,\partial D)^{\frac{\gamma^{2}}{2}}dz
	\leq(\nu_{\A^{u}_{-a}},f),
	\end{equation*}
	and on the other hand
	\begin{equation*}
	(\nu_{\A^{u}_{-a}},f) \leq\liminf_{\gamma\to 0^{+}}
	-\dfrac{\gamma}{\sqrt{2\pi}}c^{\frac{\gamma^{2}}{2}}
	\int_{D\backslash\A^{u}_{-a}}f(z)
	\log(d(z,\partial D))d(z,\partial D)^{\frac{\gamma^{2}}{2}}
	dz.
	\end{equation*}
	As both limits are equal to
	\eqref{EqNuF}, the proposition follows.
\end{proof}

\subsection{Identification with the Minkowski content measure } In this section, we finish the proof of Theorem \ref{prop:: mes LQG}, by showing that the measure $\nu_{\A_{-a}^u}$ agrees with the Minowski content in the gauge $r\mapsto |\log (r)|^{1/2}r^2$. The proof is ``deterministic'' and gives a rather general strategy for identifying two measures defined on a fractal set and satisfying coherent scaling.

Let us first introduce some notation: given $F$ a function on $(0,+\infty)$ and a compact subset $A$ of $\overline{D}$ with $0$ Lebesgue measure, we define the measure
\begin{displaymath}
\mathfrak{M}(F)=\mathfrak{M}_{A}(F)=
\1_{D\backslash A}
F(d(z,A)) d z.
\end{displaymath}
For $r\in(0,1)$, denote 
\begin{displaymath}
\JJ_{r}= \vert\log(r)\vert^{1/2}
\1_{(0,r)}.
\end{displaymath}
The Minkowski content measure of $A$ in the gauge
$\vert\log(r)\vert^{1/2} r^{2}$ would be the weak limit of
$\mathfrak{M}_{A}(\mathcal{J}_{r})$ as $r\to 0$, if it exists. Finally, for $\gamma>0$, denote
\begin{displaymath}
F_{\gamma}(s)=\dfrac{\gamma}{\sqrt{2\pi}}
\1_{s\in(0,1)}\vert\log(s)\vert s^{\frac{\gamma^{2}}{2}}.
\end{displaymath}
Proposition \ref{LemFgamma} can be then interpreted as saying that $\nu_{\A_{-a}^u} =\lim_{\gamma \to 0}\mathfrak{M}(F_\gamma)$. Notice that in this limit, as we are on a bounded domain, the extra indicator function plays no role. The missing part of Theorem \ref{prop:: mes LQG} then follows from:

\begin{prop}
	Suppose that $(\MM(F_\gamma),f)$ converges as $\gamma \to 0$ for all $f$ continuous functions compactly supported in $D$. Then, also $(\MM(\JJ_r),f)$ converges as $r\to 0$ and moreover 
	$$\dfrac{1}{2}\lim_{r\to 0}(\mathfrak{M}(\mathcal{J}_{r}),f)=
	\lim_{\gamma\to 0} (\MM(F_\gamma),f),$$ for all continuous functions compactly supported in $D$.
\end{prop}

To get some insight into why this proposition might be true, notice that both families $\mathcal{J}_r$ and $F_\gamma$ satisfy the same scaling property:
\begin{linenomath}
	\begin{align*} \label{scaling}	
	&\mathcal{J}_{r'}(s)=\beta^{-1/2}\mathcal{J}_{r}(s^{\beta}), ~~\text{ for }
	\beta = \dfrac{\vert \log(r)\vert}{\vert\log(r')\vert}; \\
	&F_{\gamma'}(s)=\beta^{-1/2}F_{\gamma}(s^{\beta}), \text{ for }
	~~\beta = \dfrac{\gamma'^{2}}{\gamma^{2}}.
	\end{align*}
\end{linenomath}
Moreover, for 
$\gamma = |\log r|^{-1/2}$, one can notice that the maximum of $F_\gamma$ is also of order $|\log r|^{-1/2}$ and this maximal value is taken at distance $O(r)$. 

We prove the proposition in two steps: first we show that if both the limits of $\MM(F_{\gamma})$ as $\gamma \to 0$ and of $\MM(\JJ_r)/2$ as $r \to 0$ exist, then they are equal. Then, we show that if the limit of $\MM(F_{\gamma})$ exists then so does the limit of $\MM(\JJ_r)$. From now on, let us take $A$ a closed set such that for all compactly supported $f$, $(\MM(F_\gamma),f)=(\MM_A(F_\gamma),f)$ converges as $\gamma$ tends to 0.

\subsubsection{Relationship between $\MM(F_\gamma)$ and $\MM(\mathcal J_r)$} 

\begin{lemma}
	Assume that both $\lim_{r\to 0}\MM(\JJ_r)$ and $\lim_{\gamma\to 0}\MM(F_\gamma)$ exist, then
	\[\lim_{\gamma\to 0}(\mathfrak{M}(F_{\gamma}),f)=
	\dfrac{1}{2}\lim_{r\to 0}(\mathfrak{M}(\mathcal{J}_{r}),f).\]
\end{lemma}
\begin{proof} Write $F_{\gamma}(s)$ as
	\begin{eqnarray*}
		F_{\gamma}(s)=	-\int_{0}^{1}F'_{\gamma}(r)\1_{s\leq r} dr
		&=& -\int_{0}^{1}F'_{\gamma}(r)\vert\log(r)\vert^{-1/2}\mathcal{J}_{r}(s) dr\\
		&=&\dfrac{\gamma}{\sqrt{2\pi}}
		\int_{0}^{1}r^{\frac{\gamma^{2}}{2}-1}
		\Big(1-\dfrac{\gamma^{2}}{2}\vert\log(r)\vert\Big)
		\vert\log(r)\vert^{-1/2}\mathcal{J}_{r}(s) dr.
	\end{eqnarray*}
	By changing the order of integration, we get
	\begin{equation*}
	(\mathfrak{M}(F_{\gamma}),f)=
	\dfrac{\gamma}{\sqrt{2\pi}}
	\int_{0}^{1}r^{\frac{\gamma^{2}}{2}-1}
	\Big(1-\dfrac{\gamma^{2}}{2}\vert\log(r)\vert\Big)
	\vert\log(r)\vert^{-1/2}
	(\mathfrak{M}(\mathcal{J}_{r}),f) dr.
	\end{equation*}
	Further, by the change of variable $r=\exp(-t^2/\gamma^{2})$ and symmetrization, this equals
	\[	
	\int_{-\infty}^{+\infty}(1-t^2/2)\frac{e^{-t^2/2}}{\sqrt{2\pi}}
	(\mathfrak{M}(\mathcal{J}_{\exp(-t^2/\gamma^{2})}),f) dt.
	\]
	Now, if the limit $\lim_{r\to 0}\MM(\JJ_r)$ exists, then 
	$\sup_{r\in(0,1)}\vert (\MM(\mathcal{J}_{r}),f)\vert<\infty$. Thus, from dominated convergence it follows that
	\begin{displaymath}
	(\nu_{\A^{u}_{-a}},f)=\lim_{\gamma\to 0}(\mathfrak{M}_{\A^{u}_{-a}}(F_{\gamma}),f)=
	c\lim_{r\to 0}(\mathfrak{M}_{\A^{u}_{-a}}(\mathcal{J}_{r}),f),
	\end{displaymath}
	where
	\begin{displaymath}
	c=	\int_{-\infty}^{+\infty}(1-t^2/2)\frac{e^{-t^2/2}}{\sqrt{2\pi}}dt=\dfrac{1}{2}.
	\qedhere
	\end{displaymath}
\end{proof}

\subsubsection{Convergence of $\MM(\JJ_r)$} 
In the last paragraph, we saw that a certain integrated version of the Minkowski content measure $\MM(\JJ_r)$ converges. We now strengthen it to full convergence. Throughout the section we suppose that $(\MM(F_\gamma),f)$ converges as $\gamma\to 0$ for all $f$ continuous compactly supported in $D$.

The idea of the proof is to approximate $\mathcal{J}_{r}$ by linear combinations
$\sum_{k=1}^{n}c_{k}F_{\gamma_{k}}$. As $\mathcal{J}_{r}$ is not continuous on $[0,1]$ and does not take value $0$ at $r=0$, one cannot expect a uniform approximation over the whole interval. Thus we have to use an uniform approximation over a subinterval of $[0,1]$ for certain smoothed versions of $\JJ_r$, and to argue that the cut out parts do not matter.

As a first step let us show that the Minkowski content measure is bounded. In fact we will show something a tiny bit stronger, that is useful for us in the later approximations:

\begin{lemma}
	\label{LemBoundedMink}
	Let $q_r(s) = \1_{s\in (0,1)}\frac{|\log (r \vee s)|}{|\log r|^{1/2}}$.
	Then, for all $f$ continuous compactly supported in $D$,
	$(\MM(q_r),f)$ stays bounded as $r\to 0$.
\end{lemma}

\begin{proof}
	It is enough to prove the result for $f$ non-negative.
	Let $\gamma_{0}>0$. There is $C>0$ such that
	\begin{displaymath}
	\1_{s\in (1/4,1/2)}\leq C F_{\gamma_{0}}(s)~~\text{and}~~
	\1_{s\in (1/2,1)}|\log s|\leq C F_{\gamma_{0}}(s).
	\end{displaymath}
	Thus, by the scaling of $F_\gamma$,
	\begin{displaymath}
	\1_{(2^{-2^{k+1}},2^{-2^{k}})}(s) = \1_{(1/4,1/2)}(s^{2^{-k}})
	\leq C F_{\gamma_{0}}(s^{2^{-k}})
	=C 2^{-\frac{k}{2}}F_{2^{-\frac{k}{2}}\gamma_{0}}(s).
	\end{displaymath}
	From where one can see that
	\begin{displaymath}
	q_{1/2}\leq C\sum_{k=1}^{+\infty}
	2^{-\frac{k}{2}}F_{2^{-\frac{k}{2}}\gamma_{0}},
	\end{displaymath}
	and as $q_r$ satisfies the scaling 
	$q_r(s) = \beta^{-1/2}q_{1/2}(s^\beta)$ with $\beta = \frac{\log 2}{|\log r|}$, it follows that
	\begin{displaymath}	
	q_r\leq C\sum_{k=1}^{+\infty}
	2^{-\frac{k}{2}}F_{2^{-\frac{k}{2}}\beta\gamma_{0}}.
	\end{displaymath}
	Thus, the lemma follows from the convergence of $(\MM(F_\gamma),f)$.
\end{proof}

For $\a > 2$, let us now define $\JJ_{1/2,\a}$ to be a positive smooth function that is bounded by $|\log 2|^{1/2}$, equals $|\log 2|^{1/2}$ over the interval $[1/2^\a,1/2-1/2^\a]$ and takes the value $0$ outside of the interval $(1/2^{2\a}, 1/2)$. Finally, let $\JJ_{r,\a}(s) = \beta^{-1/2} \JJ_{1/2,\a}(s^\beta)$ with $\beta = \frac{\log 2}{|\log r|}$ as above.

We claim that the convergence of $\MM(\JJ_{r,\a})$ as $r \to 0$ implies that of $\MM(\JJ_r)$:

\begin{claim} Assume that for all $\a > 2$ and all continuous $f$ compactly supported in $D$, $(\MM(\JJ_{r,\a}),f)$ converges as $r \to 0$. Then, also $(\MM(\JJ_r),f)$ converges as $r \to 0$.
\end{claim}

\begin{proof}
	By definition $\JJ_{r,\a}(s) \leq \JJ_r(s)$ for all $s \in (0,1)$. Moreover, similarly
	\[\JJ_r(s) \leq \log (1+r^\a) \JJ_{r+r^{\a/2},\a}(s) + \a^{-1/2}q_{r^\a}(s).\]
	Hence for all positive $f$,
	\[(\MM(\JJ_{r,\a}),f) \leq (\MM(\JJ_r),f) \leq (1+r^\a)(\MM(\JJ_{r+r^{\a/2},\a}),f) + \a^{-1/2}(\MM(q_{r^\a}),f).\]
	Thus, by letting first $r \to 0$ and then $\a \to \infty$, we conclude by Lemma \ref{LemBoundedMink} 
\end{proof}
Finally, it remains to show that
\begin{claim}
	For all $\a > 2$ and all continuous $f$ compactly supported in $D$, we have $(\MM(\JJ_{r,\a}),f)$ converges as $r \to 0$.
\end{claim}

\begin{proof}
	Notice that for any $n-$th degree polynomial $Q_n$ with $0$ constant coefficient, we can write $|\log(s)| Q_{n}(s)$ over $(0,1)$ as $\sum_{k=1}^{n}c_{k,n}F_{\gamma_{k}}(s)$, where  $\gamma_{k}:=\sqrt{2k}$ and $c_{k,n}\sqrt{2\pi}/\gamma_{k}$ is the coefficient in
	$Q_{n}$ of the power
	$s^{k}$. 
	Thus, our aim will be to approximate $\JJ_{1/2,\a}$ via polynomials. To do this, let \[\widehat \JJ_{1/2,\a}(s) = |\log s|^{-1}\JJ_{1/2,\a}(s).\] We claim there exists a sequence of polynomials $Q_{n}$ such that
	\begin{equation*}
	\label{EqAppoxBern}
	\lim_{n\to +\infty}\sup_{s\in(0,1)}
	\vert\log(s)\vert
	\dfrac{\vert \widehat \JJ_{1/2,\a}(s)
		-Q_{n}(s)\vert}{q_{1/2}(s)} =0.
	\end{equation*}	
	Indeed, this is just a consequence of the fact that the polynomials having zeros at 0 and 1 are dense in $$\{f\in \mathcal{C}^1([0,1]): f(0)=f(1)=0\}$$ for the norm 
	$\|f\|_{\mathcal{C}^1}=\max_{[0,1]}|f|+\max_{[0,1]}|f'|$. One way to see this is to approximate $f'$ by a polynomial $P_n$ and to define $Q_n(x)=\int_0^x (P_n(y))dy-x\int_0^1 P_n(y)dy$.
	
	Thus we obtain that for any $\eps > 0$, there is some $n \in \N$ large enough such that for all $s \in (0,1)$ we have 
	\[|\JJ_{1/2,\a}(s) - \sum_{k=1}^{n}c_{k,n}F_{\gamma_{k}}(s)| \leq \eps q_{1/2}(s).\]
	Now take $\beta = \frac{\log 2}{|\log r|}$ as before. Using the scaling of each of the terms, we see that over $s \in (0,1)$,
	\[|\JJ_{r,\a}(s^\beta) - \sum_{k=1}^{n}c_{k,n}F_{\beta^{1/2}\gamma_{k}}(s^\beta)| \leq \eps q_{r}(s^\beta).\]
	Hence for all $\epsilon>0$ there exists $n\in \N$ large enough such that
	\[\left | (\mathfrak{M}(\JJ_{r,\a}),f)-\left (\mathfrak{M}
	\big (\sum_{k=1}^{n}c_{k,n}F_{\beta^{1/2}\gamma_{k}}\big )
	,f\right ) \right |\leq \epsilon (\MM(q_{r}),\vert f\vert). \]
	Combining now Lemma \ref{LemBoundedMink} and the fact that as $\beta \to 0$ the sequence $ \left (\mathfrak{M}
	\big (\sum_{k=0}^{n}c_{k,n}F_{\beta^{1/2}\gamma_{k}}\big )
	,f\right )$ converges, we conclude  that $(\mathfrak{M}(\JJ_{r,\a}),f)$ has a finite limit as $r\to 0$.
\end{proof}

\section*{Acknowledgements}
The authors wish to thank F. Viklund for helpful comments on an earlier draft,  W. Werner for his vision, and B. Werness for his beautiful simulations and interesting discussions. Also, the authors are very thankful to the two anonymous referees for their careful reading and helpful remarks. This work was partially supported by the SNF grants SNF-155922 and SNF-175505. A. Sepúlveda was supported by the ERC grant LiKo 676999. The authors are thankful to the NCCR Swissmap. T. Lupu acknowledges the support of Dr.  Max Rössler,  the Walter Haefner Foundation and the ETH Zurich Foundation. The work of this paper was finished during the memorable visit of J.Aru and A. Sepúlveda to Paris in May 2018, on the invitation by T. Lupu, funded by PEPS "Jeunes chercheuses et jeunes chercheurs" 2018 of INSMI.

\bibliographystyle{alpha}	
\bibliography{biblio}

\end{document}